\documentclass[11pt,a4paper]{article}
\usepackage[margin= 1in]{geometry}
\input{macros}
\title{Learning the score under shape constraints}
\author{Rebecca M. Lewis$^*$, Oliver Y. Feng$^\dagger$, Henry W. J. Reeve$^\ddagger$, Min Xu$^\sharp$ \\ 
and Richard J. Samworth$^\flat$\\ \\
$^*$Department of Statistics, University of Oxford\\
$^\dagger$Department of Statistics, London School of Economics and Political Science\\
$^\ddagger$School of Artificial Intelligence, Nanjing University\\
$^\sharp$Department of Statistics, Rutgers University\\
$^\flat$Statistical Laboratory, University of Cambridge}
\date{\today}

\begin{document}

\maketitle

\begin{abstract}
Score estimation has recently emerged as a key modern statistical challenge, due to its pivotal role in generative modelling via diffusion models.  Moreover, it is an essential ingredient in a new approach to linear regression via convex $M$-estimation, where the corresponding error densities are projected onto the log-concave class. Motivated by these applications, we study the minimax risk of score estimation with respect to squared $L^2(P_0)$-loss, where~$P_0$ denotes an underlying log-concave distribution on $\mathbb{R}$. Such distributions have decreasing score functions, but on its own, this shape constraint is insufficient to guarantee a finite minimax risk. We therefore define subclasses of log-concave densities that capture two fundamental aspects of the estimation problem.  First, we establish the crucial impact of tail behaviour on score estimation by determining the minimax rate over a class of log-concave densities whose score function exhibits controlled growth relative to the quantile levels. Second, we explore the interplay between smoothness and log-concavity by considering the class of log-concave densities with a scale restriction and a $(\beta,L)$-H\"older assumption on the log-density for some $\beta \in [1,2]$. We show that the minimax risk over this latter class is of order $L^{2/(2\beta+1)}n^{-\beta/(2\beta+1)}$ up to poly-logarithmic factors, where $n$ denotes the sample size.  When $\beta < 2$, this rate is faster than could be obtained under either the shape constraint or the smoothness assumption alone.  Our upper bounds are attained by a locally adaptive, multiscale estimator constructed from a uniform confidence band for the score function.  This study highlights intriguing differences between the score estimation and density estimation problems over this shape-constrained class. 
\end{abstract}

\section{Introduction}

The \emph{score function} $\psi_0$ associated with a differentiable density $f_0$ on $\mathbb{R}$ is the derivative of $\log f_0$, so that $\psi_0 = f_0'/f_0$.  Score estimation, often carried out via \emph{score matching} \citep{cox1985penalty, hyvarinen05score}, has emerged over the last few years as a key statistical challenge in modern artificial intelligence \citep{{hyvarinen2007extensions,vincent2011connection,lyu2012interpretation,song2020sliced,yu2020simultaneous,yu2022generalized,lederer2023extremes,benton2024denoising}}.  One of the main reasons is that it forms the bedrock of \emph{diffusion models}, which are now routinely employed in generative modelling to sample synthetic data (often images) that mimic instances in an existing database \citep{sohl2015, song19generative,ho2020, song2020sliced, dhariwal2021, song21score}.

Of course, the score function has played a fundamental role in classical parametric statistical theory over the last century \citep{fisher1922mathematical,vdV1998asymptotic}.  Maximum likelihood estimators are typically found by solving \emph{score equations}, while \emph{score tests} \citep{cox1979theoretical} provide a very general strategy for hypothesis testing in parametric models.  The original motivation of \citet{hyvarinen05score} for score matching came from \emph{non-normalised} statistical models, such as `energy-based' models \citep{song2021train} employed in statistical physics and graphical modelling, where the normalisation constant, or \emph{partition function}, is an intractable function of the unknown parameter.  Score matching avoids this difficulty, since this unknown partition function cancels in the ratio that defines the score function, and therefore represents a computationally tractable alternative to maximum likelihood.  As another application, score matching turns out to be equivalent to Stein's unbiased risk estimator of the mixing distribution in empirical Bayes denoising \citep{ghosh2025stein}. 

This paper concerns the estimation of decreasing score functions $\psi_0$, which arise from log-concave densities $f_0$.  Beyond the overarching rationale for studying score estimation mentioned above, our particular motivation for considering this class of score functions comes from the \emph{antitonic score matching} technique for linear regression developed by \citet{feng24asm}, with accompanying \texttt{R} package \citet{kao24asm}.  That work shows how, in linear models with real-valued responses, it is possible to improve on the ordinary least squares estimator via a convex $M$-estimator whose data-driven convex loss function is chosen to minimise the asymptotic variance of the resulting estimator of the vector of regression coefficients.  A crucial step in this procedure is to estimate the score function of the error density using the residuals obtained from a pilot regression coefficient estimator. Although it suffices for the purposes of the antitonic efficiency theory of \citet{feng24asm} to have a consistent score estimator, the finite-sample performance is nevertheless affected by the quality of score estimation. 

Our main goal, then, is to determine the optimal rates of score estimation over subclasses of decreasing score functions on the real line. Throughout, we will measure the accuracy of a score estimate $\hat{\psi} \colon \R \to \R$ using the loss function
\begin{equation}
\label{eq:Loss}
\mathcal{L}(\hat{\psi},\psi_0) := \int_{-\infty}^\infty (\hat{\psi} - \psi_0)^2 f_0. 
\end{equation}
This choice of loss is motivated by several considerations: first, the total variation sampling error of a score-based generative modelling algorithm can often be bounded in terms of the cumulative score estimation error with respect to density-weighted $L^2$ losses
\citep{block2022, lee2022, chenM2023, chenH2023, lee2023, benton2024, li2024, li2024towards}. Moreover,~$\mathcal{L}$ is the relevant loss for consistency of score estimation in~\citet{feng24asm}, and is seen more generally in semiparametric estimation problems with nuisance parameters \citep[Chapter~25.8]{vdV1998asymptotic}. In high-density (bulk) regions where $f_0$ may be regarded as a constant,~$\mathcal{L}$ essentially amounts to a squared error loss for derivative estimation (already a harder task than that of estimating $f_0$, and controlled by the smoothness of $f_0'$).  A distinguishing aspect of the score estimation problem, however, is that since $\psi_0 = f_0'/f_0$, errors in estimating $f_0'$ are magnified in regions where~$f_0$ is small. Moreover, these are precisely the regions where there is a relative scarcity of data points, and where we can therefore expect inaccurate pointwise estimates $\hat{\psi}_n(x)$ of $\psi_0(x)$, especially when $|\psi_0(x)|$ grows rapidly as $x$ approaches the boundary of the support of~$f_0$. More positively, however, these potentially large contributions $\bigl(\hat{\psi}_n(x) - \psi_0(x)\bigr)^2$ to the loss in the tails of the distribution are downweighted to some extent by $f_0(x)$.  In fact, our analysis reveals that the interplay between these two competing factors varies considerably between log-concave densities with different tail behaviours.  Surprisingly, it turns out that light-tailed log-concave densities pose the greatest challenges for score estimation, since their score functions grow rapidly relative to the quantile level.  

To gain further insight into the influence of the low-density regions on the overall score estimation accuracy with respect to $\mathcal{L}$, consider any estimator $\hat{\psi}_n$ of $\psi_0$ based on independent and identically distributed observations $X_1,\dotsc,X_n$. In the absence of any parametric assumptions on $\psi_0$, there is no clear rationale for any specific method of extrapolating the estimator beyond the range of the data. For simplicity, then, suppose first that $\hat{\psi}_n = 0$ on $\R \setminus [X_{(1)},X_{(n)}]$, where~$X_{(1)}$ and $X_{(n)}$ denote the minimum and maximum order statistics respectively. If $n \geq 2$, then
$[X_{(1)},X_{(n)}] \subseteq [F_0^{-1}(1/n), F_0^{-1}(1 - 1/n)]$ with probability $(1 - 2/n)^n$, where $F_0^{-1}$ denotes the quantile function. On this event,
\begin{equation}
\label{eq:zero-extrapolation}
\mathcal{L}(\hat{\psi}_n,\psi_0) \geq \int_{-\infty}^{F_0^{-1}(1/n)} \psi_0^2 f_0 + \int_{F_0^{-1}(1 - 1/n)}^\infty \psi_0^2 f_0,
\end{equation}
and this right-hand side may decay only very slowly with $n$.  Indeed, suppose that~$f_0$ is a $\mathrm{Beta}(a,a)$ density with $a > 2$, so that $f_0(x) \propto x^{a - 1}(1 - x)^{a - 1}$ for $x \in [0,1]$. Then $f_0$ is log-concave with decreasing score function $x \mapsto \psi_0(x) = \frac{a-1}{x} - \frac{a-1}{1 - x}$ on $(0,1)$, and has finite Fisher information for location $i(f_0) := \int_0^1 \psi_0^2 f_0 = \frac{4(2a-1)(a-1)}{a-2}$. Since $\lim_{u \to 0} u^{-1/a}F_0^{-1}(u) \in (0,\infty)$ and the density is symmetric about $1/2$, there exist $c_a,c_a' > 0$, depending only on $a$, such that
\[
\int_{-\infty}^{F_0^{-1}(1/n)} \psi_0^2 f_0 + \int_{F_0^{-1}(1 - 1/n)}^\infty \psi_0^2 f_0 \geq c_a\int_0^{n^{-1/a}} u^{a - 3} \laplaced u \geq c_a' n^{-(a - 2)/a}.
\]
Here, the exponent can be made arbitrarily close to~0, yielding arbitrarily slow rates in~\eqref{eq:zero-extrapolation}, when~$a$ is close to 2. The issue is not only that $\psi_0(x)^2 = \psi_0(1 - x)^2$ diverges quickly as $x \to 0$, but also that $f_0$ does not decay sufficiently rapidly at the edge of its support to downweight~$\psi_0^2$ adequately in~\eqref{eq:zero-extrapolation}. This poor estimation accuracy is not alleviated by instead performing constant or affine extrapolation of~$\hat{\psi}_n$ beyond the range of the data, since in the tail regions, no affine function is a sufficiently accurate approximation to the diverging score $\psi_0$ with respect to the loss function~$\mathcal{L}$.

Motivated by these considerations, we first show that over the class of all continuous log-concave densities, the minimax risk for score estimation with respect to $\mathcal{L}$ is of constant order even under the scale restriction $i(f_0) \leq r$ for any $r \in (0,\infty)$. To obtain consistent minimax rates, we subsequently define two types of subclasses of log-concave densities; see Definitions~\ref{def:tail-growth} and~\ref{def:holder} in Section~\ref{sec:main-results}. The first of these controls the growth of the score function relative to the quantiles of the distribution, so as to address the aforementioned issues with score estimation in the tails. More precisely, given class parameters $\gamma \in (0,1]$ and $L > 0$, we consider densities~$f_0$ whose score function $\psi_0$ satisfies
\begin{equation}
\label{eq:tail-growth-score}
|\psi_0(x)| \leq L\min\{F_0(x), 1 - F_0(x)\}^{-(1 - \gamma)/2}
\end{equation}
for all $x \in \R$, where $F_0$ denotes the corresponding distribution function.  This defining condition can be equivalently and conveniently expressed in terms of the \textit{density quantile function} $J_0 := f_0 \circ F_0^{-1} \colon [0,1] \to \R$. It turns out that this is a concave function with derivative $J_0' = \psi_0 \circ F_0^{-1}$, so that~\eqref{eq:tail-growth-score} can be stated as $|J_0'(u)| \leq L\min\{u, 1 - u\}^{-(1 - \gamma)/2}$ for $u \in (0,1)$. The requirement in~\eqref{eq:tail-growth-score} becomes progressively weaker as $\gamma$ decreases to 0, eventually encompassing all $\mathrm{Beta}(a,b)$ densities with $a,b > 2$.  In Section~\ref{subsec:tail-growth}, we establish that the minimax rate $\mathcal{M}_n$ of score estimation over the class of log-concave densities satisfying~\eqref{eq:tail-growth-score} takes the form
\begin{equation}
\label{eq:BasicMinimaxRate}
\mathcal{M}_n \asymp L^2 n^{-(\gamma \wedge 1/3)}
\end{equation}
up to a poly-logarithmic factor in $n$.  Thus, there is an unexpected elbow in the rate at $\gamma = 1/3$, which represents a point of transition: for $\gamma > 1/3$ the error from the bulk of the distribution drives the rate, whereas for smaller $\gamma$ the contribution to the loss from the tails dominates.  In fact, Theorem~\ref{thm:tail-growth} yields a more refined conclusion than~\eqref{eq:BasicMinimaxRate} in two respects: first, we further partition the class via an upper bound $r$ on the Fisher information and derive the sharp dependence on~$r$.  Second, working in the minimax quantile framework of \citet{ma2025high}, which provides a finer-grained understanding of the behaviour of the loss function than the minimax risk, we establish the sharp dependence on the quantile level.

In Section~\ref{subsec:holder}, we seek minimax rates faster than $n^{-1/3}$ by imposing smoothness assumptions on the score function in addition to its monotonicity. We establish that the minimax risk of score estimation over classes $\mathcal{F}_{\beta,L}$ of log-concave densities whose logarithms are $(\beta,L)$-H\"older on $\R$ for some $\beta \in [1,2]$ is $L^{2/\beta}n^{-\beta/(2\beta+1)}$, up to a poly-logarithmic factor in $n$ and $L$. As in Section~\ref{subsec:tail-growth}, we further partition the class in terms of the Fisher information to obtain optimal rates for these subclasses too. 
The analysis reveals two intriguing aspects of the score estimation problem relative to the corresponding density estimation problem.  First, the monotonicity and smoothness restrictions lead to improved rates compared with either assumption alone.  Indeed, in the absence of the shape constraint, we show in Proposition~\ref{prop:lower-bd-no-shape-c} that the minimax rate over the larger class of absolutely continuous densities whose logarithms are $(\beta,L)$-H\"older is at least $L^{2/\beta}n^{-2(\beta-1)/(2\beta+1)}$, which is strictly larger than $L^{2/\beta}n^{-\beta/(2\beta+1)}$ when $\beta \in [1,2)$. The reason for this is that a log-concave density whose logarithm is $(\beta,L)$-H\"older must contain regions on which the function is smoother and where the score can be estimated more accurately.  By contrast, a variant of the lower bound construction of \citet{kim2016global} shows that the minimax rate of $n^{-4/5}$ in squared Hellinger distance for log-concave density estimation can be realised even under the additional restriction of arbitrary smoothness.   
Second, we find that the minimax risk of score estimation over the class of log-concave densities whose logarithms are piecewise affine with at most $5$ affine pieces is of order $n^{-1/3}$, which matches the rate over the class $\mathcal{F}_{1,L}$, and so these densities have score functions that are among the hardest to estimate within $\bigcup_{\beta \in [1,2] }\mathcal{F}_{\beta,L}$.  On the other hand, \citet{kim2018} show that the maximum likelihood estimator adapts to these classes of densities, attaining a parametric rate with respect to the squared Hellinger distance up to a poly-logarithmic factor.

The upper bounds on the minimax rates arise from an adaptive multiscale estimator constructed in Section~\ref{sec:multiscale}. 
 Our rates 
 stem from two key properties of our estimator that hold on an event of high probability (Proposition~\ref{prop:scorePointwiseErr}). First, it is locally adaptive, in that it achieves improved pointwise rates for score estimation in regions where the log-density is smoother; and 
second, the estimator does not overestimate the score in absolute value at any point, which in particular affords protection against wildly inaccurate estimation in the tail regions.
Our approach also yields a valid confidence band for the score function that exploits the monotonicity constraint.


Alternative score estimators could also be considered. The log-concave maximum likelihood estimator has been shown to yield optimal rates for density estimation and it would be interesting to study its behaviour for score estimation. However, finite-sample pointwise estimation theory to parallel the asymptotic theory of \citet{balabdaoui2009} is currently unavailable for this estimator, and although some asymptotic results on its tail behaviour have emerged recently \citep{ryter2024tails}, these again do not suffice for our purposes. Score matching may provide another alternative, although the optimisation problem fails over the class of decreasing functions without additional regularisation. A benefit of our approach is that it not only provides an estimator, but also yields valid confidence bands for the score function. 


\subsection{Related work}

\textbf{Score estimation:} Nonparametric score estimation was first studied by \citet{cox1985penalty}, who applied integration by parts to the $L^2(P_0)$ loss function and combined the resulting objective with a roughness penalty.  He showed that this yields an estimator with a variational characterisation and established asymptotic guarantees on its rate of convergence over smoothness classes.  \citet{hyvarinen05score} coined the term \emph{score matching} for the Cox's (unpenalised) objective as a way to estimate an unknown finite-dimensional parameter in models where the probability density function is known only up to a multiplicative normalising constant. Score matching has since been generalised \citep{hyvarinen2007extensions, lyu2012interpretation}, and the asymptotic normality of the estimator and its generalisations, as well as their statistical efficiencies, have been studied by \citet{koehler2022statistical}, \citet{pabbaraju2023} and \citet{qin2024}.

With the primary aim of controlling the accuracy of a diffusion model as a distribution learner, $L^2(\probDistribution)$-score estimation has recently been considered in more general nonparametric contexts. 
\citet{oko2023diffusion}, \citet{zhang2024minimax}, \citet{wibisono2024optimal} and \citet{dou2024optimal} study $L^2(\probDistribution)$-score estimation of densities $f_t(\cdot) := m_t^{-1}f(\cdot /m_t)*\phi_t(\cdot)$ with $t \in [0,T]$ for some $T \in (0,\infty)$, where $f$ lies in a suitable nonparametric class of densities on $\mathbb{R}^d$, $\phi_t$ is the density of the $N_d(0,\sigma_t^2 I_d)$ distribution with $\sigma_t > 0$, and $m_t$ is a scale parameter. \citet{oko2023diffusion} consider~$f$ contained in a Besov space and supported on $[-1,1]^d$, with the additional restriction that $f$ is bounded above and below by positive constants and smooth on the boundary of its support. They establish an upper bound on the score estimation error of a neural network estimator.
When $m_t = 1$ and $\sigma_t = t$, \citet{zhang2024minimax} avoid the assumption of compact support and instead consider the class of $\alpha$-sub-Gaussian densities $f$. They show that the optimal rate of order $n^{-1}t^{-(d+2)/2}(t^{d/2}\vee \alpha^d)$ (up to poly-logarithmic factors) is achieved  by a kernel-based score estimator of $f_t$ with a constant bandwidth. The bounds of \citet{oko2023diffusion} and \citet{zhang2024minimax} rely crucially on $t$ being bounded away from zero which, in addition to the smoothness condition on the boundary region of \citet{oko2023diffusion}, ensures that the tails of the density $f_t$ do not decay to zero too quickly. The former condition restricts the analysis of  \citet{oko2023diffusion} and \citet{zhang2024minimax} to the study of diffusion models with early stopping which, when considering their ability to estimate a density, results in additional logarithmic factors. 

\citet{wibisono2024optimal} consider score estimation of sub-Gaussian densities $f$ that are fully supported on $\mathbb{R}^d$ and have a Lipschitz score function, showing that the optimal rate of score estimation is $n^{-2/(d+4)}$ up to logarithmic factors. Similarly to \citet{oko2023diffusion}, \citet{zhang2024minimax}, \citet{wibisono2024optimal} first consider estimation of the score of $f_t$ for an appropriately chosen $t>0$ when $m_t = 1$ and $\sigma_t = t$, although unlike the previously mentioned papers, \citet{wibisono2024optimal} also control the difference between the score functions of $f$ and~$f_t$. The rate $n^{-2/5}$ obtained when $d=1$ coincides with our rate for the related class of log-concave densities on $\mathbb{R}$ with Lipschitz score functions and Fisher information bounded above by one. \citet{wibisono2024optimal} also extend their analysis to sub-Gaussian densities that have $\beta$-H\"older score functions with $\beta \in (1,2]$, establishing an upper bound on the expected $L^2(\probDistribution)$-score estimation error of order $n^{-2(\beta-1)/(2\beta+d)}$. This coincides with the rate obtained by \citet{zhang2024minimax} for $\beta$-H\"older densities with $\beta > 1$. Our Proposition~\ref{prop:lower-bd-no-shape-c} derives a lower bound on the minimax rate for the class of absolutely continuous densities on $\mathbb{R}$ with $\beta$-H\"older score functions. More generally, while we also consider $\beta$-H\"older score functions in this work, we remove the sub-Gaussianity assumption of \citet{wibisono2024optimal}, and instead impose a monotonicity constraint on the score function that substantially changes the results and analysis. 

Continuing in the aforementioned setting with $m_t = 1$ and $\sigma_t = t$, \citet{dou2024optimal} consider densities that are supported on $[-1,1]$ and are $(\beta,L)$-H\"older on this support with $\beta,L > 0$. 
Under the assumption that the densities are bounded above and below by positive constants, they derive matching upper and lower bounds on the minimax rate of score estimation as a function of $t$. 
Notably, \citet{dou2024optimal} consider all $t\geq 0$, thus establishing the optimality of the diffusion model in this context without additional logarithmic factors.  When $t$ is large, the rate is of parametric order $1/(nt^2)$ due to the convolution of $f$ with a high-variance Gaussian density, and when $t$ is small, the rate is $n^{-2(\beta-1)/(2\beta+1)}\wedge t^{\beta-1}$. These rates are achieved by a careful combination of constant bandwidth kernel density estimators, consisting of a Gaussian kernel in the large $t$ regime, and a kernel that has knowledge of $\beta$ in the small $t$ regime. 




From a different perspective, score estimation has also been considered in the case where the data are supported on an unknown low-dimensional subspace \citep{oko2023diffusion, chenM2023} 
The computational aspects of score estimation have been studied by \citet{song2024}, who shows that accurate score estimation is computationally hard in the absence of strong assumptions on the data distribution.

\medskip
\noindent
\textbf{Shape-constrained estimation and inference:}  Estimation of log-concave densities has bene a central problem within the field of nonparametric inference under shape constraints over the last 10-15 years.  The primary focus has been on the log-concave maximum likelihood estimator (MLE); see \citet{walther2009inference}, \citet{samworth2018recent,samworth2025nonparametric} and \citet{dumbgen2024shape} for surveys of recent progress in this area.  Beyond the estimator's existence and uniqueness \citep{walther2002, dumbgen2011approximation}, this literature has studied its computation \citep{dumbgen2007, cule2010maximum, liu2018}, optimality for density estimation with respect to global error losses \citep{dumbgen2009,kim2016global,kur2019}, its adaptation properties to piecewise affine log-densities \citep{kim2018,feng2021,kur2025}, behaviour under model misspecification \citep{cule2010theoretical, barber2021local} and pointwise limiting distribution theory \citep{balabdaoui2009, doss2019univariate,deng2023inference}.  For alternative approaches to log-concave density estimation, see for example \citet{baraud2016} and~\citet{baraud2022}.

Previous works that employ multiscale techniques in shape-constrained contexts include \citet{dumbgen2001multiscale}, \citet{dumbgen2003optimal}, \citet{cai2015framework}, \citet{chatterjee2021regret}, \citet{walther2022}, \citet{cai2024estimation} and \citet{mukherjee2024optimal}.

\subsection{Notation}
\label{sec:notation}
We adopt the conventions $0 \times \infty := 0 =: 0 \times (-\infty)$ and $\inf\emptyset = \infty$, and let $[n] := \{1,\dotsc,n\}$ for $n \in \N$. We write $[x,y]$ for the interval with (not necessarily ordered) endpoints $x,y \in \R$, i.e.~the convex hull of $\{x,y\}$. For $x \in \R$, let $x_+ := \max\{0,x\}$ and $\sgn(x) := 2\mathbbm{1}_{\{x \geq 0\}} - 1$, and define $\log_+(x) := 1 \vee \log(x)$ when $x > 0$.  For real-valued sequences or functions $a,b$,  
we write $a \lesssim b$ or $b \gtrsim a$ if $a \leq Cb$ for some universal constant $C \in (0,\infty)$, while $a \asymp b$ means that both $a \lesssim b$ and $b \lesssim a$. More generally, we write $a \lesssim_{\alpha_1,\dotsc,\alpha_r} b$ or $b \gtrsim_{\alpha_1,\dotsc,\alpha_r} a$ if $a \leq Cb$ for some $C \in (0,\infty)$ depending only on given parameters $\alpha_1,\dotsc,\alpha_r$, and define $a \asymp_{\alpha_1,\dotsc,\alpha_r} b$ analogously. For real-valued functions $a,b$ defined on $(0,x_0)$ for some $x_0 > 0$, we write $a(x) \sim b(x)$ as $x \to 0$ if $\lim_{x \to 0} a(x)/b(x) = 1$.

We write $\lambda$ for Lebesgue measure on $\mathbb{R}$. For functions $g,h \colon \mathbb{R} \to \mathbb{R}$ such that $(x,z) \mapsto g(x)h(z-x)$ is integrable on $\mathbb{R}^2$, the \textit{convolution} of $g$ and $h$ is the function $(g*h)(\cdot) := \int_{-\infty}^\infty g(x)h(\cdot - x) \laplaced x$. We also recall the following facts from e.g.~\citet[Section~10.7]{samworth24modern}. A function $f \colon I \to \R$ is said to be \textit{absolutely continuous} on an interval $I \subseteq \R$ if for every $\epsilon > 0$, there exists $\delta > 0$ such that whenever $(x_1,y_1),\dotsc,(x_m,y_m) \subseteq I$ are pairwise disjoint with $\sum_{j=1}^m |y_j - x_j| < \delta$, we have $\sum_{j=1}^m |f(y_j) - f(x_j)| < \epsilon$. 
For an open set $U \subseteq \R$, we say that $f \colon U \to \R$ is \emph{locally absolutely continuous} on $U$ if it is absolutely continuous on every compact interval $I \subseteq U$.  This condition is equivalent to the existence of a Borel measurable function $g \colon U \to \R$ such that for every compact subinterval $I \subseteq U$, we have $\int_I |g| < \infty$ and $f(z_2) = f(z_1) + \int_{z_1}^{z_2} g$ for all $z_1,z_2 \in I$.

\section{Main results}
\label{sec:main-results}

\subsection{Minimax preliminaries}
\label{subsec:main-results-prelim}

Let $\mathcal{F}$ denote the class of all continuous, log-concave densities on $\R$; such densities are absolutely continuous by Lemma~\ref{lem:log-concave-continuity}. For any $f_0 \in \mathcal{F}$, the log-density $\phi_0 := \log f_0$ has well-defined left and right derivatives $\phi_0^{(\mathrm{L})} \geq \phi_0^{(\mathrm{R})}$ on $\support := \{x \in \R : f_0(x) > 0\}$. These are both decreasing functions that agree at all but at most countably many points \citep[][Theorem~23.1]{rockafellar97convex}. By a \textit{score function} we mean any $\psi_0 \colon \R \to [-\infty,\infty]$ satisfying
$\phi_0^{(\mathrm{R})}(x) \leq \psi_0(x) \leq \phi_0^{(\mathrm{L})}(x)$ for all $x \in \support$, with $\psi_0(x) := \infty$ for $x \leq \inf\support$ and $\psi_0(x) := -\infty$ for $x \geq \sup\support$. This is equivalent to $\psi_0 \colon \R \to [-\infty,\infty]$ being a decreasing function such that $\int_x^y \psi_0 = \phi_0(y) - \phi_0(x)$ for all $x,y \in \support$. 

The loss function $\mathcal{L}$ in~\eqref{eq:Loss} does not depend on the particular version of the score function because $\psi_0$ is uniquely determined except possibly at countably many points. When $\psi = f'/f$ is the score function of a distribution $P$ with a locally absolutely continuous density~$f$ on $\R$ satisfying $\support \subseteq \supp(f)$, the above quantity is the \textit{Fisher divergence} from $P$ to the distribution $\probDistribution$ with density $f_0$; see~\citet[Definition~1.13]{johnson2004information},~\citet[Section~2.2]{feng24asm} and references therein.

We are interested in estimating $\psi_0$ based on $X_1,\ldots,X_n \iid f_0$.  To this end, denote by $\hat{\Psi}_n$ the class of jointly measurable functions\footnote{We do not insist that $\hat{\psi}_n(\,\cdot\,;x_1,\dotsc,x_n)$ be a bona fide score function for any $(x_1,\dotsc,x_n) \in \R^n$.} $\hat{\psi}_n:\R \times \R^n \to \R$. For any $\mathcal{G}\subseteq \mathcal{F}$, the \textit{minimax risk} of score estimation over~$\mathcal{G}$ is defined as
\[
\mathcal{M}_n(\mathcal{G}) := \inf_{\hat{\psi}_n \in \hat{\Psi}_n} \sup_{f_0 \in \mathcal{G}} \mathbb{E}_{f_0}\bigl\{\mathcal{L}\bigl(\hat{\psi}_n(\cdot; X_1,\dots,X_n),\psi_0\bigr)\bigr\}.
\]
The usual minimax risk framework provides a benchmark for quantifying the difficulty of statistical problems over classes of interest and comparing the performance of estimators. However, the expected loss alone may not fully capture the variability and tail behaviour of $\mathcal{L}$ as a data-dependent quantity. With this in mind, we conduct a more fine-grained analysis of score estimation using the notion of minimax quantiles \citep{ma2025high}. For $\delta \in (0,1]$ and $\hat{\psi}_n \in \hat{\Psi}_n$, the \emph{$(1 - \delta)$th quantile} of $\mathcal{L}(\hat{\psi}_n,\psi_0)$ is
\[
\mathrm{Quantile}_{f_0}(1 - \delta,\hat{\psi}_n,\psi_0) := \inf\Bigl\{r \in [0,\infty): \mathbb{P}_{f_0}\bigl\{\mathcal{L}\bigl(\hat{\psi}_n(\cdot;X_1,\ldots,X_n),\psi_0\bigr) \leq r\bigr\} \geq 1 - \delta\Bigr\}.
\]
For $\mathcal{G} \subseteq \mathcal{F}$, the \emph{minimax $(1 - \delta)$th quantile} of $\mathcal{L}$ over~$\mathcal{G}$ is then defined as 
\[
\mathcal{M}_n(\delta,\mathcal{G}) := \inf_{\hat{\psi}_n \in \hat{\Psi}_n} \sup_{f_0 \in \mathcal{G}} \mathrm{Quantile}_{f_0}(1 - \delta,\hat{\psi}_n,\psi_0).
\]

The following proposition shows that further restrictions on the class $\mathcal{F}$ are necessary to obtain non-trivial minimax risk and quantile bounds. Recall the \emph{Fisher information (for location)} of $f_0$ is defined as $i(f_0) := \int_{-\infty}^\infty \psi_0^2 f_0$.

\begin{prop}
\label{prop:risk-inf}
For $r \in (0,\infty)$, define $\mathcal{F}_r := \{f_0 \in \mathcal{F} : i(f_0) \leq r\}$. Then for any $\delta \in (0,1/4]$, we have
\[
\mathcal{M}_n(\mathcal{F}_r) = r\mathcal{M}_n(\mathcal{F}_1) \asymp r \quad\text{and}\quad \mathcal{M}_n(\delta,\mathcal{F}_r) = r\mathcal{M}_n(\delta,\mathcal{F}_1) \asymp r.
\]
Consequently, $\mathcal{M}_n(\mathcal{F}) = \mathcal{M}_n(\delta,\mathcal{F}) = \infty$.
\end{prop}



Proposition~\ref{prop:risk-inf} and subsequent results reveal several stark differences between our score estimation problem and density estimation over the class of log-concave densities on $\R$, for which the minimax rate in squared Hellinger distance is known to be $n^{-4/5}$ \citep{kim2016global}. First, the loss function $\mathcal{L}$ is not scale invariant: indeed, if $\mathcal{G}_c := \{cf_0(c\,\cdot) : f_0 \in \mathcal{G}\}$ for some $\mathcal{G} \subseteq \mathcal{F}$ and $c > 0$, then
\begin{equation}
\label{eq:loss-scaling}
\mathcal{M}_n(\mathcal{G}_c) = c^2 \mathcal{M}_n(\mathcal{G}) \quad\text{and}\quad \mathcal{M}_n(\delta,\mathcal{G}_c) = c^2 \mathcal{M}_n(\delta,\mathcal{G}).
\end{equation}
Moreover, if $f_0 \in \mathcal{G}$, then $f_c(\cdot) := cf_0(c\,\cdot)$ satisfies $i(f_c) = c^2 i(f_0)$. 
This means that $\mathcal{M}_n(\mathcal{F}_r) = r\mathcal{M}_n(\mathcal{F}_1)$ and $\mathcal{M}_n(\delta,\mathcal{F}_r) = r\mathcal{M}_n(\delta,\mathcal{F}_1)$. These considerations highlight the appropriateness of Fisher information (as opposed to e.g.~the inverse variance) as a scale parameter tailored to the loss function $\mathcal{L}$.  In particular, it is the error of the estimator $\tilde{\psi} \equiv 0$, so $\mathcal{M}_n(\mathcal{F}_r) \vee \mathcal{M}_n(\delta,\mathcal{F}_r) \leq r$.

The more interesting aspect of Proposition~\ref{prop:risk-inf} is the conclusion that $\mathcal{M}_n(\mathcal{F}_1) \asymp \mathcal{M}_n(\delta,\mathcal{F}_1) \gtrsim 1$,
which stems from the fact that for every~$\epsilon > 0$, there exist $f_0,f_1 \in \mathcal{F}_1$ with $\KL(f_0,f_1) := \int_{-\infty}^\infty f_0 \log(f_0/f_1) < \epsilon$ but Fisher divergence of order 1; see the discussion after Theorem~\ref{thm:tail-growth} for further details. For further insight into the mismatch
between the Kullback--Leibler and Fisher divergences, see \citet{feng24asm} and~\citet{ghosh2025stein}.

Proposition~\ref{prop:risk-inf} clarifies that additional restrictions on the log-concave class are necessary for the minimax risk to converge to zero with the sample size. We consider two types of conditions that limit the amount that a score function can vary. First, we consider subclasses of $\mathcal{F}$ where the growth of the score function is controlled by the quantile level, establishing the fundamental impact that tail regions have on the estimation error. On the other hand, by also considering subsets of $\mathcal{F}$ that satisfy an additional H\"older assumption, we explore the interaction between smoothness and monotonicity, and its effect on the accuracy of score estimation in the bulk.

\subsection{Tail growth condition}
\label{subsec:tail-growth}

As alluded to in the introduction, the density quantile function $J_0 := f_0 \circ F_0^{-1}$ is particularly well-suited to analysis of $L^2(P_0)$ score estimation for several reasons. First, the fact that $J_0' = \psi_0 \circ F_0^{-1}$ Lebesgue almost everywhere (Lemma~\ref{lem:density-quantile-basics}\textit{(a)}) means that for any Borel measurable $\psi \colon \R \to \R$, a quantile transformation yields
\[
\mathcal{L}(\psi,\psi_0) = \int_0^1 \bigl\{\psi\bigl(F_0^{-1}(u)\bigr) - \psi_0\bigl(F_0^{-1}(u)\bigr)\bigr\}^2 \laplaced u = \int_0^1 (\psi \circ F_0^{-1} - J_0')^2,
\]
thereby rebalancing the pointwise contributions to the overall integrated error. In particular, this standardisation reduces the effect of large pointwise errors in low-density regions, and enables us to compare densities with different supports (not necessarily the entire real line) on a canonical scale; see the examples in Section~\ref{subsec:examples}. Moreover, our analysis relies crucially on the property that a Lebesgue density $f_0$ on $\R$ is log-concave if and only if $J_0$ is concave \citep[Proposition~A.1(c)]{bobkov1996extremal}.

Our first main results concern subclasses of $\mathcal{F}$ defined in terms of envelope functions for $|J_0'|$ that grow polynomially in the tails.

\begin{definition}
\label{def:tail-growth}
Given $\gamma \in (0,1]$ and $L > 0$, define $\mathcal{J}_{\gamma,L}$ to be the set of all $f_0 \in \mathcal{F}$ whose density quantile functions $J_0$ satisfy
\begin{equation}
\label{eq:tail-growth}
|J_0'(u)| \equiv |(\psi_0 \circ F_0^{-1})(u)| \leq \frac{L}{\{u \wedge (1 - u)\}^{(1 - \gamma)/2}} \quad\text{for all }u \in (0,1).
\end{equation}
Moreover, given $r > 0$, let $\mathcal{J}_{\gamma,L,r} := \{f_0 \in \mathcal{J}_{\gamma,L} : i(f_0) \leq r\}$.
\end{definition}

The class $\mathcal{J}_{1,L}$ consists precisely of those $f_0 \in \mathcal{F}$ with score functions uniformly bounded by~$L$ in absolute value, i.e.~with $L$-Lipschitz log-densities $\phi_0$, while smaller values of $\gamma$ in~\eqref{eq:tail-growth} allow the score function to grow more rapidly in the tails relative to the quantile level. 
The limit $\gamma \to 0$ represents the furthest extent to which the right-hand side of~\eqref{eq:tail-growth} remains a square-integrable function of $u \in (0,1)$, so that by Lemma~\ref{lem:density-quantile-basics}\textit{(b)},
\begin{align}
\label{eq:envelope-integral}
\sup_{f_0 \in \mathcal{J}_{\gamma,L}} i(f_0) = \sup_{f_0 \in \mathcal{J}_{\gamma,L}} \int_0^1 (J_0')^2 \leq 2\int_0^{1/2} \frac{L^2}{u^{1 - \gamma}} \laplaced u \leq \frac{2^{1 - \gamma}L^2}{\gamma} < \infty.
\end{align}
We emphasise that densities in~$\mathcal{J}_{\gamma,L}$ for $\gamma < 1$ need not be supported on the whole of $\R$; for instance, we verify in Example~\ref{ex:beta-DQ} below that a $\mathrm{Beta}(a,b)$ density belongs to $\mathcal{J}_{\gamma,L}$ for some $L > 0$ if and only if $a \wedge b \geq 2/(1 - \gamma)$.

\begin{theorem}
\label{thm:tail-growth}
For $n \geq 3$ and $\delta \in (0,1)$, let $\varepsilon_{n,\delta} := n^{-1}\log(n^2/\delta)$.  Given $\gamma \in (0,1]$ and $L,r > 0$, the minimax $(1 - \delta)$th quantile for score estimation over $\mathcal{J}_{\gamma,L,r}$ satisfies
\begin{equation}
\label{eq:tail-growth-full}
\mathcal{M}_n(\delta,\mathcal{J}_{\gamma,L,r}) \lesssim_\gamma \bar{r}\min\biggl\{\Bigl(\frac{L^{2/\gamma}}{\bar{r}^{1/\gamma}} \varepsilon_{n,\delta}\Bigr)^{\gamma \wedge \frac{1}{3}}\log_+^{\Ind_{\{\gamma = 1/3\}}}\Bigl(\frac{1}{\varepsilon_{n,\delta}}\Bigr),\,1\biggr\}, 
\end{equation}
where $\bar{r} := r \wedge L^2$. In particular,
\begin{equation}
\label{eq:tail-growth-simple}
\mathcal{M}_n(\delta,\mathcal{J}_{\gamma,L}) \lesssim_\gamma L^2\min\biggl\{\varepsilon_{n,\delta}^{\gamma \wedge \frac{1}{3}}\log_+^{\Ind_{\{\gamma = 1/3\}}}\Bigl(\frac{1}{\varepsilon_{n,\delta}}\Bigr),\,1\biggr\}.
\end{equation}
These bounds are attained by the adaptive multiscale estimator $\hat{\psi}_{n,\delta}$ defined by~\eqref{eq:scoreEstimateShapeConstrained} below.
\end{theorem}

As we will see in Section~\ref{sec:multiscale}, $\hat{\psi}_{n,\delta}$ is adaptive in the sense that its definition does not depend on any of the class parameters $\gamma$, $L$ or $r$. The following upper bound on the minimax risk is a straightforward consequence of Theorem~\ref{thm:tail-growth} and its proof, together with the fact that the trivial estimator $\tilde\psi \equiv 0$ always satisfies $\sup_{f_0 \in \mathcal{J}_{\gamma,L,r}} \int_{-\infty}^\infty (\tilde\psi - \psi_0)^2 f_0 = \sup_{f_0 \in \mathcal{J}_{\gamma,L,r}} i(f_0) \lesssim_\gamma L^2$ by~\eqref{eq:envelope-integral}.

\begin{corollary}
\label{cor:tail-growth-expectation}
In the setting of Theorem~\ref{thm:tail-growth}, the minimax risk for score estimation over $\mathcal{J}_{\gamma,L,r}$ satisfies
\[
\mathcal{M}_n(\mathcal{J}_{\gamma,L,r}) \lesssim_\gamma \bar{r}\min\Bigl\{\Bigl(\frac{L^{2/\gamma}}{\bar{r}^{1/\gamma}} \cdot \frac{1}{n}\Bigr)^{\gamma \wedge \frac{1}{3}}\log^{c_\gamma}n,\,1\Bigr\},
\]
where $c_\gamma := (\gamma \wedge \frac{1}{3}) + \Ind_{\{\gamma = 1/3\}}$.
\end{corollary}

The conclusions of Theorem~\ref{thm:tail-growth} and Corollary~\ref{cor:tail-growth-expectation} match the following minimax lower bounds in all parameters $\gamma,L,r,n,\delta$, up to logarithmic factors.

\begin{theorem}
\label{thm:tail-growth-lower-bd}
Let $n \in \N$, $\delta \in (0,1/4]$, $\gamma \in (0,1]$ and $L,r > 0$, and write $\tilde{\varepsilon}_{n,\delta} := n^{-1}\log(1/\delta)$. The minimax $(1 - \delta)$th quantile and minimax risk for score estimation over $\mathcal{J}_{\gamma,L,r}$ satisfy
\begin{align*}
\mathcal{M}_n(\delta,\mathcal{J}_{\gamma,L,r}) \gtrsim_\gamma \bar{r}\min\biggl\{\biggl(\frac{L^{2/\gamma}}{\bar{r}^{1/\gamma}}\tilde{\varepsilon}_{n,\delta}\biggr)^{\gamma \wedge \frac{1}{3}},\,1\biggr\} \;\;\;\text{and}\;\;\;
\mathcal{M}_n(\mathcal{J}_{\gamma,L,r}) \gtrsim_\gamma \bar{r}\min\biggl\{\Bigl(\frac{L^{2/\gamma}}{\bar{r}^{1/\gamma}}\cdot \frac{1}{n}\Bigr)^{\gamma \wedge \frac{1}{3}},\,1\biggr\}.
\end{align*}
\end{theorem}

\begin{figure}
\begin{center}
\subfigure{\includegraphics[width=0.45\textwidth]{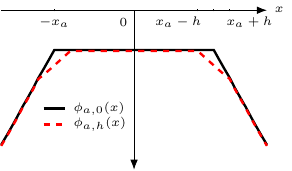}}
\hfill
\subfigure{\includegraphics[width=0.45\textwidth]{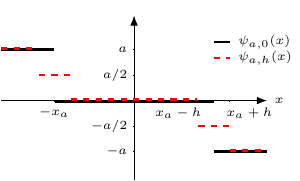}}
\end{center}
\caption{The functions $\phi_{a,0}$ and $\phi_{a,h}$ (left) and corresponding scores $\psi_{a,0}$ and $\psi_{a,h}$ (right), where $\phi_{a,h'} \colon \R \to \R$ is defined for $h' \in \{0,h\}$ by $\phi_{a,h'}(x) := -\log(2a) + \int_0^x \psi_{a,h'}$ so that $f_{a,0}(x) = e^{\phi_{a,0}(x)} = e^{-a(|x| - x_a)_+}/(2a)$ and $f_{a,h}(x) \propto e^{\phi_{a,h}(x)}$. 
}
\label{fig:lower-bd-tails-c1}
\end{figure}

From Theorem~\ref{thm:tail-growth-lower-bd}, we see that as $\gamma$ decreases to 0, the deteriorating worst-case tail behaviour leads to arbitrarily slow minimax rates that scale as $n^{-\gamma}$ when $L,r,\delta$ are regarded as constants. For $\gamma \in (0,1/3]$, the lower bound of order $n^{-\gamma}$ in Theorem~\ref{thm:tail-growth-lower-bd} is based on a perturbation of a `flattened' $\mathrm{Beta}(b,b)$ density with $b = 2/(1 - \gamma)$, whose tail behaviour matches the envelope function in~\eqref{eq:tail-growth} and whose bulk region is modified to satisfy the Fisher information constraint. The overall dependence on the sample size, namely $n^{-(\gamma \wedge \frac{1}{3})}$, reveals an elbow at $\gamma = 1/3$ because a lower bound of order $n^{-1/3}$ holds over densities in the class with scores uniformly bounded in absolute value by $L^{1/\gamma}$.  It is instructive to mention at this point that this latter lower bound is established by considering the two-parameter subfamily $\{f_{a,h} : a \geq \sqrt{2},\,h \in [0,1/a]\}$ of flattened Laplace densities in Lemma~\ref{lem:fa-2pt}, where $f_{a,h}$ has score function $\psi_{a,h}$ given by
\begin{align*}
\psi_{a,h}(x) := -\sgn(x)\Bigl(a\Ind_{\{|x| \geq x_a + h\}} + \frac{a}{2}\Ind_{\{|x| \in (x_a - h, x_a + h)\}} \Bigr)
\end{align*}
and $x_a := a - 1/a \geq 1/a$; see Figure~\ref{fig:lower-bd-tails-c1}. 
For each $a$, we have $f_{a,0},f_{a,h} \in \mathcal{J}_{\gamma,a^\gamma,2}$ with $\KL(f_{a,0},f_{a,h}) < 2ah^3$, but the Fisher divergence is at least $ah/(2e)$. Thus by choosing the largest value of $h \in (0,1/a]$ with $n\mathrm{KL}(f_{a,0},f_{a,h}) \leq 1$, namely $h^* := (2an)^{-1/3} \wedge a^{-1}$, we deduce from Le Cam's two-point lemma \citep[e.g.][Lemma~8.4]{samworth24modern} that
\begin{equation}
\label{eq:fa-2pt}
\mathcal{M}_n(\mathcal{J}_{\gamma,a^\gamma,2}) \geq \mathcal{M}_n(\{f_{a,0},f_{a,h^*}\}) \gtrsim ah^* \asymp \Bigl(\frac{a^2}{n}\Bigr)^{1/3} \wedge 1.
\end{equation}
Taking $a = L^{1/\gamma}$ when $L \geq 2^{\gamma/2}$ yields the lower bound on the minimax risk in Theorem~\ref{thm:tail-growth-lower-bd} when $r = 2$. A minimax quantile bound follows by virtually identical reasoning based on a high-probability version of Le Cam's two-point lemma \citep[Corollary~6]{ma2025high}.
This construction relies on $\psi_{a,0}$ having a large jump at a point $x_a$ where $f_{a,0}$ takes its maximum value. In a relatively high-density neighbourhood of $x_a$, we may therefore apply a large pointwise perturbation to $\psi_{a,0}$ that abides by the monotonicity constraint and has negligible effect on the original distribution. Moreover, the lower bound in~\eqref{eq:fa-2pt} applies to piecewise log-affine densities in $\mathcal{J}_{\gamma,L,2}$ with at most 5 affine pieces, so their score functions are among the hardest to estimate when $\gamma \in (1/3,1]$. This reveals another key difference between score and density estimation; for the latter, the log-concave maximum likelihood estimator adaptively achieves a near-parametric squared Hellinger rate of order $k\log^{5/4}(n)/n$ over log-$k$-affine densities.

The dependence of the minimax bounds on $L$ and $r$ arises from the following scaling property: for $f_0 \in \mathcal{F}$ and $\lambda > 0$, the density $x \mapsto \lambda f_0(\lambda x)$ has density quantile function $u \mapsto \lambda J_0(u)$ and Fisher information $\lambda^2 i(f_0)$. Therefore,
\[
\mathcal{J}_{\gamma,L,r} = \bigl\{x \mapsto r^{1/2}f_0(r^{1/2}x) : f_0 \in \mathcal{J}_{\gamma,Lr^{-1/2},1}\bigr\},
\]
so $\mathcal{M}_n(\mathcal{J}_{\gamma,L,r}) = r\mathcal{M}_n(\mathcal{J}_{\gamma,Lr^{-1/2},1})$ and $\mathcal{M}_n(\delta,\mathcal{J}_{\gamma,L,r}) = r\mathcal{M}_n(\delta,\mathcal{J}_{\gamma,Lr^{-1/2},1})$ by~\eqref{eq:loss-scaling}. Moreover,~\eqref{eq:envelope-integral} means that $\mathcal{J}_{\gamma,L,r} = \mathcal{J}_{\gamma,L}$ for all $r \geq 2^{1 - \gamma}L^2/\gamma$, so the minimax rates feature~$\bar{r} = r \wedge L^2$ instead of $r$.

\subsection{H\"older classes}
\label{subsec:holder}

To improve upon the minimax rates in Section~\ref{subsec:tail-growth}, which are at best of order $n^{-1/3}$, we must exclude the densities $f_{a,h}$ in Section~\ref{subsec:main-results-prelim}, whose score functions are discontinuous. To this end, we impose a H\"older condition that in particular enforces higher regularity in high-density regions.

\begin{definition}
\label{def:holder}
For $\beta \in [1,2]$ and $L > 0$, denote by $\mathcal{F}_{\beta,L}$ the set of $f_0 \in \mathcal{F}$ whose score functions satisfy
\begin{align*}
|\psi_0(x) - \psi_0(y)| \leq L|x - y|^{\beta - 1}
\end{align*}
for all $x,y \in \R$.
Moreover, for $r > 0$, let
\[
\mathcal{F}_{\beta,L,r} := \{f_0 \in \mathcal{F}_{\beta,L} : i(f_0) \leq r\}.
\]
\end{definition}
When $\beta = 1$, the class $\mathcal{F}_{1,L,r}$ coincides with $\mathcal{J}_{1,L,r}$ and contains densities with discontinuous score functions. On the other hand, when $\beta > 1$, the log-density $\phi_0$ is necessarily differentiable (in fact, $\beta$-H\"older continuous), and $\psi_0$ is required to be $(\beta-1)$-H\"older continuous.

\begin{theorem}
\label{thm:holder-minimax-upper}
For $n \geq 3$, $\delta \in (0,1)$, $\beta \in [1,2]$ and $L,r > 0$, the minimax $(1 - \delta)$th quantile for score estimation over $\mathcal{F}_{\beta,L,r}$ satisfies
\begin{equation}
\label{eq:holder-minimax-upper}
\mathcal{M}_n(\delta,\mathcal{F}_{\beta,L,r}) \lesssim_\beta r'\min\biggl\{\Bigl(\frac{L^{2/\beta}}{r'}\varepsilon_{n,\delta}\Bigr)^{\beta/(2\beta + 1)}\log_+^{\frac{\beta - 1}{2\beta + 1}}\Bigl(\frac{L^{2/\beta}}{r'}\Bigr) + \frac{L^{2/\beta}}{r'}\varepsilon_{n,\delta}\log_+^{3 - \frac{2}{\beta}}\Bigl(\frac{1}{\varepsilon_{n,\delta}}\Bigr),\,1\biggr\},
\end{equation}
where $r' := r \wedge L^{2/\beta}$. In particular,
\[
\mathcal{M}_n(\delta,\mathcal{F}_{\beta,L}) \lesssim_\beta L^{2/\beta}(\varepsilon_{n,\delta} \wedge 1)^{\beta/(2\beta + 1)},
\]
and these bounds are all attained by the adaptive multiscale estimator $\scoreEstimateShapeConstrained$ in~\eqref{eq:scoreEstimateShapeConstrained}.
\end{theorem}

Similarly to Theorem~\ref{thm:tail-growth}, the dependence of this bound on $L$ and $r$ respects the scaling property
\[
\mathcal{F}_{\beta,L,r} = \bigl\{x \mapsto r^{1/2}f_0(r^{1/2}x) : f_0 \in \mathcal{F}_{\beta,Lr^{-\beta/2},1}\bigr\},
\]
so that $\mathcal{M}_n(\delta,\mathcal{F}_{\beta,L,r}) = r\mathcal{M}_n(\delta,\mathcal{F}_{\beta,Lr^{-\beta/2},1})$ by~\eqref{eq:loss-scaling}. Moreover, by~\eqref{eq:holder-fisher-information} in Lemma~\ref{lem:holder-density-quantile}, $\mathcal{F}_{\beta,L,r} = \mathcal{F}_{\beta,L}$ whenever $r \geq 8L^{2/\beta}$, which enables us to obtain bounds in terms of $r'$ instead of~$r$ throughout.  Corollary~\ref{cor:holder-minimax-expec} below follows from Theorem~\ref{thm:holder-minimax-upper} by arguing similarly to the proof of Corollary~\ref{cor:tail-growth-expectation} in Section~\ref{subsec:tail-growth}.
\begin{corollary}
\label{cor:holder-minimax-expec}
In the setting of Theorem~\ref{thm:holder-minimax-upper}, the minimax risk of score estimation over $\mathcal{F}_{\beta,L,r}$ satisfies
\[
\mathcal{M}_n(\mathcal{F}_{\beta,L,r}) \lesssim_\beta r'\min\biggl\{\Bigl(\frac{L^{2/\beta}}{r'} \cdot \frac{\log^{4 - \frac{2}{\beta}}n}{n}\Bigr)^{\beta/(2\beta + 1)},\,1\biggr\}.
\]
\end{corollary}
Theorem~\ref{thm:holder-minimax-upper} and Corollary~\ref{cor:holder-minimax-expec} are complemented by the following lower bounds.
\begin{theorem}
\label{thm:holder-lower-bd}
For $n \in \N$, $\delta \in (0,1/4]$, $\beta \in [1,2]$ and $L,r > 0$, we have
\begin{align*}
\mathcal{M}_n(\delta,\mathcal{F}_{\beta,L,r}) &\gtrsim_\beta r'\min\biggl\{\biggl(\frac{L^{2/\beta}}{r'} \cdot \frac{1}{n}\biggr)^{\beta/(2\beta + 1)} + \biggl(\frac{L^{2/\beta}}{r'} \cdot \tilde{\varepsilon}_{n,\delta}\biggr)^{\beta/(\beta + 2)},\,1\biggr\}, \\
\mathcal{M}_n(\mathcal{F}_{\beta,L,r}) &\gtrsim_\beta r'\min\biggl\{\biggl(\frac{L^{2/\beta}}{r'} \cdot \frac{1}{n}\biggr)^{\beta/(2\beta + 1)},\,1\biggr\}.
\end{align*}
where $r' := r \wedge L^{2/\beta}$.
\end{theorem}

Theorems~\ref{thm:holder-minimax-upper} and~\ref{thm:holder-lower-bd} establish that when $L,r,\delta$ are regarded as constants, the minimax rate of score estimation over $\mathcal{F}_{\beta,L,r}$ is of order $n^{-\beta/(2\beta + 1)}$ up to logarithmic factors. This rate improves from $n^{-1/3}$ when $\beta = 1$ (as seen in Section~\ref{subsec:tail-growth} for $\gamma = 1$) to $n^{-2/5}$ when $\beta = 2$, in which case the score function is required to be $L$-Lipschitz. 

The log-concavity and smoothness assumptions imposed in Definition~\ref{def:holder} lead to faster rates than those obtained under either assumption alone, and we now discuss these two issues in turn.  Under log-concavity alone, Proposition~\ref{prop:risk-inf} reveals that consistent minimax estimation is not possible without further assumptions.   
On the other hand, in the presence of log-concavity, a $\beta$-H\"older restriction with $\beta \in (1,2]$ has two beneficial effects.  First, it regulates the tail behaviour via Lemma~\ref{lem:holder-density-quantile}, but second, and arguably more importantly from our perspective, it rules out log-concave densities whose score functions have discontinuities (in particular in relatively high-density regions). These include the densities $f_{a,h}$ mentioned in Section~\ref{subsec:tail-growth} whose piecewise constant score functions give rise to the minimax lower bound of order $n^{-1/3}$ in~\eqref{eq:fa-2pt}.  Together with Theorem~\ref{thm:holder-minimax-upper} and Corollary~\ref{cor:holder-minimax-expec}, this implies that, in addition to the comment regarding $\mathcal{J}_{\gamma,L,r}$ in Section~\ref{subsec:tail-growth}, piecewise log-affine densities with at most 5 affine pieces are among the hardest score functions to estimate within $\bigcup_{\beta \in [1,2]} \mathcal{F}_{\beta,L,r}$; see also the very similar proof of the $\beta = 1$ case of Theorem~\ref{thm:holder-lower-bd}. 

Now turning to the setting where only H\"older smoothness is assumed, a minimax lower bound of order $L^{2/\beta}n^{-2(\beta-1)/(2\beta+1)}$ holds over the class of locally absolutely continuous densities $f_0$ on $\R$ with $(\beta - 1,L)$-H\"older score functions; see Proposition~\ref{prop:lower-bd-no-shape-c}. Thus, no worst-case rate is achievable when $\beta = 1$, and the rate is always worse than that obtained in Theorem~\ref{thm:holder-minimax-upper} for $\beta \in [1,2)$. The crucial observation that explains this difference is that for a $\beta$-H\"older log-density with $\beta \in [1,2)$, an additional monotonicity or bounded variation constraint forces the score function to have smoother regions where the `local H\"older exponent' exceeds $\beta - 1$. For instance when $\beta = 1$, a uniformly bounded score function may be discontinuous on a set of positive Lebesgue measure,
but if it has bounded variation over $\R$ then it is necessarily differentiable
Lebesgue almost everywhere \citep[Theorem~7.2.7]{dudley2002}.
For general $\beta \in [1,2]$, Proposition~\ref{prop:Holder-basic} manifests the heterogeneous smoothness of a decreasing $(\beta - 1)$-H\"older function in a precise sense discussed in Section~\ref{subsec:multiscale-ptwise}.
This result informs the construction of the multiscale estimator $\hat{\psi}_{n,\delta}$, which attains the optimal rates by adapting to the local regularity of the score function, and is pivotal to our proof strategy for Theorems~\ref{thm:tail-growth} and~\ref{thm:holder-minimax-upper}; see Section~\ref{sec:multiscale}.


By contrast with the above discussion, in log-concave density estimation with respect to squared Hellinger distance, a variant of the univariate lower bound construction in \citet[Theorem~1]{kim2016global} shows that the global minimax rate of order $n^{-4/5}$ is attained over $\mathcal{F}_{\beta,L,r}$ for any fixed $\beta,L,r$ as above. In other words, as far as integrated rather than pointwise or supremum-norm error is concerned, the rates for density estimation cannot be improved by combining concavity of the log-density (implying twice differentiability almost everywhere by \citet{aleksandorov1939almost} or \citet[p.~31]{schneider2013convex}) with H\"older smoothness for any $\beta \in [1,2]$. The difference between score and density estimation in this regard can be explained in part by the effect of the discontinuities and cusps of a decreasing score function. These turn out to drive the minimax rate of score estimation (as demonstrated by our lower bounds) and also contribute disproportionately to the error of locally adaptive smoothing procedures including our estimator~$\hat{\psi}_{n,\delta}$.
For instance, such estimators incur a pointwise bias of order~1 in an $n^{-1/3}$-neighbourhood of a discontinuity; see Example~\ref{ex:laplace-ptwise} in Section~\ref{subsec:multiscale-ptwise}. Insisting on H\"older smoothness everywhere (not just almost everywhere differentiability) reduces the aggregate influence of points of low regularity. On the other hand, their effect is intrinsically less severe in density estimation because of a more favourable local bias--variance tradeoff to that in derivative estimation. 
It turns out that the pointwise bias is only of order $n^{-1/3}$ in the aforementioned $n^{-1/3}$-neighbourhood of a kink in the log-density, and ruling out this behaviour via global smoothness has no effect on the worst-case integrated error in the univariate case.

\subsection{Examples}
\label{subsec:examples}

\begin{figure}
\begin{center}
\subfigure[$J_0$]{\includegraphics[width=0.45\textwidth]{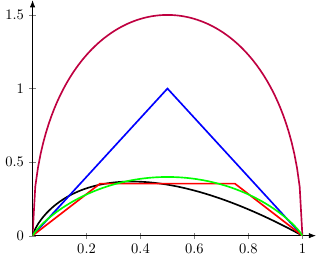}}
\hfill
\subfigure[$|J_0'|$]{\includegraphics[width=0.45\textwidth]{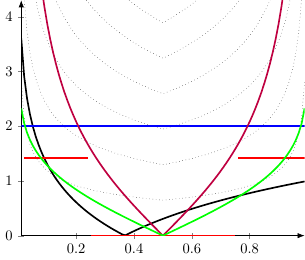}}
\end{center}
\caption{The density quantile function $J_0$ and the absolute value of its derivative $|J_0'|$ for the Laplace with $L=2$ (blue), flattened Laplace with $a=\sqrt{2}$ (red), Subbotin with $\beta=2$ (green), Gumbel (black) and Beta with $a=b=2.001$ (purple) densities. The dotted lines on the right-hand plot are the functions $L\bigl\{u \wedge (1-u)\bigr\}^{-(1-\gamma)/2}$ with $\gamma = 1/4$ and $L \in \{0.5,1,1.5,2,2.5,3\}$.
}
\label{fig:J0}
\end{figure}

To illustrate the range of densities $f_0 \in \mathcal{F}$ encompassed by our subclasses, we present the following prototypical examples. In each case, we compute the score function $\psi_0$ and describe the behaviour of the density quantile function $J_0 = f_0 \circ F_0^{-1}$ and its derivative $J_0' = \psi_0 \circ F_0^{-1}$. The functions are plotted in Figure~\ref{fig:J0}.

\begin{example}[Laplace]
\label{ex:laplace-DQ}
This score function is bounded and has a discontinuity at the origin: for $L > 0$, we have
\begin{alignat*}{3}
f_0(x) &= \frac{Le^{-L|x|}}{2}, &\qquad \psi_0(x) &= -L\sgn(x) &&\quad\text{for }x \in \R, \\
J_0'(u) &= L\sgn\Bigl(\frac{1}{2} - u\Bigr), &\qquad J_0(u) &= L\min\{u, 1 - u\} &&\quad\text{for }u \in (0,1).
\end{alignat*}
Therefore, $f_0 \in \mathcal{J}_{1,L} = \mathcal{F}_{1,L}$, but $f_0 \notin \bigcup_{L' \in (0,\infty)} \mathcal{F}_{\beta,L'}$ for $\beta \in (1,2]$.
\end{example}

\begin{example}[Flattened Laplace]
Recalling the density $f_0 \equiv f_{a,0}$ in Section~\ref{subsec:main-results-prelim}, which can be defined for any $a \geq 1$, we have
\begin{alignat*}{3}
f_0(x) &= \frac{e^{-a(|x| - x_a)_+}}{2a}, &\qquad \psi_0(x) &= -a\sgn(x)\Ind_{\{|x| \geq x_a\}} &&\quad\text{for }x \in \R, \\[6pt]
J_0'(u) &= a\sgn\Bigl(\frac{1}{2} - u\Bigr)\Ind_{\{u \wedge (1 - u) \leq 1/(2a^2)\}}, &\qquad J_0(u) &= a\min\Bigl\{u, 1 - u, \frac{1}{2a^2}\Bigr\} &&\quad\text{for }u \in (0,1),
\end{alignat*}
where $x_a := a - 1/a \geq 0$. Thus $i(f_0) = \int_{-\infty}^\infty \psi_0^2 f_0 = \int_0^1 (J_0')^2 = 1$, and for $\gamma \in (0,1]$ and $L,r > 0$, we have $f_0 \in \mathcal{J}_{\gamma,L,r}$ if and only if $a \leq (2^{(1 - \gamma)/2}L)^{1/\gamma}$ and $r \geq 1$. As in Example~\ref{ex:laplace-DQ}, $f_0 \notin \bigcup_{L' \in (0,\infty)} \mathcal{F}_{\beta,L'}$ for $\beta \in (1,2]$. Therefore, to refine the 
conclusion $\mathcal{M}_n(\mathcal{F}_1) \asymp 1$ of Proposition~\ref{prop:risk-inf} arising from~\eqref{eq:fa-2pt} for large $a$, the effect of our class restrictions is to place an upper bound on $a$ in Section~\ref{subsec:tail-growth}, or to rule out $\{f_{a,0} : a \geq 1\}$ altogether in Section~\ref{subsec:holder} when $\beta \in (1,2]$.
\end{example}

\begin{example}[Subbotin]
\label{ex:gaussian-DQ}
Here we consider a different extension of the family of Laplace densities. For $\beta \in [1,\infty)$, consider the density and corresponding score function given by
\[
f_0(x) = \frac{\beta^{(\beta - 1)/\beta} e^{-|x|^\beta/\beta}}{2\Gamma(1/\beta)}, \qquad \psi_0(x) = -\sgn(x)|x|^{\beta - 1}
\]
for $x \in \R$. The function $f_0$ is also known as a \emph{generalised Gaussian} density. 
For $x \geq 0$, we have
\begin{align*}
F_0(-x) = 1 - F_0(x) = \int_x^\infty f_0 \leq \frac{\beta^{(\beta - 1)/\beta}}{2\Gamma(1/\beta)} \int_x^\infty e^{-z^\beta/\beta} \laplaced z \leq \frac{\beta^{(\beta - 1)/\beta} e^{-x^\beta/\beta}}{2x^{\beta - 1}\Gamma(1/\beta)},
\end{align*}
and integration by parts yields
\begin{align*}
\int_x^\infty f_0 = \frac{\beta^{(\beta - 1)/\beta}}{2\Gamma(1/\beta)}\biggl(\frac{e^{-x^\beta/\beta}}{x^{\beta - 1}} - \int_x^\infty \frac{(\beta - 1)e^{-z^\beta/\beta}}{z^\beta} \laplaced z\biggr) \geq \frac{\beta^{(\beta - 1)/\beta} e^{-x^\beta/\beta}}{2x^{\beta - 1}\Gamma(1/\beta)}\Bigl(1 - \frac{\beta - 1}{x^\beta}\Bigr).
\end{align*}
Therefore, as $u \to 0$, the quantile and density quantile functions satisfy
\begin{align*}
F_0^{-1}(u) &= -F_0^{-1}(1 - u) \sim -\Bigl\{\beta\log\Bigl(\frac{\beta^{(\beta - 1)/\beta}}{2\Gamma(1/\beta)u}\Bigr)\Bigr\}^{1/\beta}, \\
J_0'(u) &= -J_0'(1 - u) = (\psi_0 \circ F_0^{-1})(u) \sim \Bigl\{\beta\log\Bigl(\frac{\beta^{(\beta - 1)/\beta}}{2\Gamma(1/\beta)u}\Bigr)\Bigr\}^{(\beta - 1)/\beta}.
\end{align*}
Hence for all $u \in (0,1)$, we have
\[
|J_0'(u)| \lesssim_\beta \log_+^{(\beta - 1)/\beta}\Bigl(\frac{1}{u \wedge (1 - u)}\Bigr) \quad\text{and}\quad J_0(u) \lesssim_\beta u\log_+^{(\beta - 1)/\beta}\Bigl(\frac{1}{u \wedge (1 - u)}\Bigr).
\]
If $\beta \in [1,2]$, then $f_0 \in \mathcal{F}_{\beta,1}$, and $J_0'$ and $J_0$ behave similarly to their respective envelope functions over $\mathcal{F}_{\beta,1}$ in Lemma~\ref{lem:holder-density-quantile}.
\end{example}

\begin{example}[Gumbel]
\label{ex:gumbel-DQ}
For $x \in \R$, we have
\begin{align*}
f_0(x) = e^{-x - e^{-x}}, \qquad \psi_0(x) = e^{-x} - 1, \qquad F_0(x) = e^{-e^{-x}}.
\end{align*}
Thus $f_0$ decays exponentially in the right tail, and its left tail is considerably lighter than in the previous examples. For $u \in (0,1)$, we have $F_0^{-1}(u) = -\log\log(1/u)$, so
\begin{align*}
J_0'(u) = (\psi_0 \circ F_0^{-1})(u) = \log\Bigl(\frac{1}{eu}\Bigr), \qquad J_0(u) = (f_0 \circ F_0^{-1})(u) = u\log\Bigl(\frac{1}{u}\Bigr).
\end{align*}
\end{example}

The key point is that in both Examples~\ref{ex:gaussian-DQ} and~\ref{ex:gumbel-DQ}, $|J_0'(u)|$ and $|J_0'(1 - u)|$ grow at most logarithmically in $1/u$ as $u \to 0$, so for every $\gamma \in (0,1)$, there exists $L > 0$ such that $f_0 \in \mathcal{J}_{\gamma,L}$. In particular, in the left tail of the Gumbel density, the score function increases exponentially but this growth rate is still fairly benign \textit{relative to the quantile level}, in comparison with the next example.

\begin{example}[Beta]
\label{ex:beta-DQ}
For $a,b > 1$, the $\mathrm{Beta}(a,b)$ density is continuous with
\[
f_0(x) = \frac{x^{a - 1}(1 - x)^{b - 1}}{\mathrm{B}(a,b)}, \qquad \psi_0(x) = \frac{a - 1}{x} - \frac{b - 1}{1 - x}
\]
for $x \in (0,1)$, where $\mathrm{B}(a,b) := \Gamma(a)\Gamma(b)/\Gamma(a + b)$ denotes the beta function.  The distribution function satisfies
\begin{align*}
F_0(x) \leq \int_0^x \frac{z^{a - 1}}{\mathrm{B}(a,b)} \laplaced z \leq \frac{x^a}{a\mathrm{B}(a,b)}, \qquad 1 - F_0(x) \leq \int_x^1 \frac{(1 - z)^{b - 1}}{\mathrm{B}(a,b)} \laplaced z = \frac{(1 - x)^b}{b\mathrm{B}(a,b)},
\end{align*}
with $F_0(x) \sim x^a/\bigl(a\mathrm{B}(a,b)\bigr)$ and $1 - F_0(1 - x) \sim x^b/\bigl(b\mathrm{B}(a,b)\bigr)$ as $x \to 0$. Thus, as $u \to 0$, we have
\begin{alignat*}{2}
F_0^{-1}(u) &\sim \bigl(a\mathrm{B}(a,b)u\bigr)^{1/a}, &\qquad F_0^{-1}(1 - u) &\sim 1 - \bigl(b\mathrm{B}(a,b)u\bigr)^{1/b}, \\
|J_0'(u)| &\asymp_{a,b} \frac{1}{u^{1/a}}, &\qquad |J_0'(1 - u)| &\asymp_{a,b} \frac{1}{u^{1/b}}, \\
J_0(u) &\asymp_{a,b} u^{(a - 1)/a}, &\qquad J_0(1 - u) &\asymp_{a,b} u^{(b - 1)/b},
\end{alignat*}
so $F_0^{-1}(u) \lesssim_{a,b} u^{1/a}$ and $F_0^{-1}(1 - u) \lesssim_{a,b} u^{1/b}$ for all $u \in (0,1)$. In contrast to the previous examples, $|J_0'(u)|$ and $|J_0'(1 - u)|$ exhibit polynomial growth in $1/u$ as $u \to 0$. Since for $\gamma \in (0,1]$ we have 
\begin{align*}
\sup_{u \in (0,1)} \bigl(u \wedge (1 - u)\bigr)^{\frac{1 - \gamma}{2}} |J_0'(u)| &= \sup_{x \in (0,1)} \bigl(F_0(x) \wedge F_0(1 - x)\bigr)^{\frac{1 - \gamma}{2}} |\psi_0(x)| \asymp_{a,b} \sup_{x \in (0,1)} \frac{(x^{a \wedge b})^{\frac{1 - \gamma}{2}}}{x},
\end{align*}
it follows that $f_0 \in \mathcal{J}_{\gamma,L}$ for some $L > 0$ if and only if $a \wedge b \geq 2/(1 - \gamma)$.
\end{example}

\section{Multiscale confidence band and optimal estimator}
\label{sec:multiscale}

\subsection{Construction and validity of the confidence band}

To achieve the minimax optimal rates over the subclasses in the previous section, we seek an estimator that exploits the monotonicity of the score function, adapts to heterogeneous smoothness in the bulk, and avoids erratic behaviour in the tails, all without knowledge of class parameters. To this end, we first use kernel smoothing to construct a valid confidence band for $\psi_0$ whose width adjusts to local properties of the unknown underlying function. Our approach relies on the following result. Let $\kernelFunction[] \colon \R \to [0,1]$ be the triangular kernel given by $\kernelFunction[](x) := (1-|x|)\mathbbm{1}_{\{|x| \leq 1\}}$, and for a bandwidth $h > 0$, write 
$\kernelFunction[h](\cdot) := h^{-1}\kernelFunction[](\cdot/h)$. 

\begin{lemma}
\label{lem:biasControl}
Let $h > 0$ and let $f_0$ be a density that is locally absolutely continuous on $\R$. Then $f_h := K_h * f_0$ is a density that is absolutely continuous as a function on $\R$, with continuous derivative $f_h' = K_h' * f_0 = K_h * f_0'$. Moreover, if $f_0 \in \mathcal{F}$, then for all $x \in \R$, the score function\footnote{Following the convention adopted in Section~\ref{sec:notation}, we define $\psi_h(z) := \infty$ whenever $z + h \leq \inf \support$ and $\psi_h(z) := -\infty$ whenever $z -h \geq \sup \support$.} $\psi_h := f_h'/f_h$ satisfies
\begin{equation}
\label{eq:psi_h-psi_0}
\psi_h(x + h) \leq \psi_0(x) \leq \psi_h(x - h).
\end{equation}
\end{lemma}

Here,~\eqref{eq:psi_h-psi_0} relies only on $\psi_0$ being decreasing, not any specific version of the score. Given $\delta > 0$ and $X_1,\dotsc,X_n \iid \probDistribution$, Lemma~\ref{lem:biasControl} suggests constructing a $(1 - \delta)$-level confidence band for~$\psi_0$ from upper and lower confidence limits for $\psi_h(\cdot - h)$ and $\psi_h(\cdot + h)$ respectively, for each $h > 0$ in turn, and then intersecting all such bands in a \textit{multiscale} fashion. It suffices to build a family of confidence bands $\bigl[\scoreConfidenceBoundByBandwidthCentre[-1](\cdot), \scoreConfidenceBoundByBandwidthCentre[1](\cdot)\bigr]$ indexed by $h > 0$, each of which is valid for $\psi_h(\cdot)$ simultaneously on some event $\goodEvent$ of probability at least $1 - \delta$, i.e.
\begin{equation}
\label{eq:simultaneous-validity}
\goodEvent \subseteq \bigl\{\scoreConfidenceBoundByBandwidthCentre[-1](z) \leq \psi_h(z) \leq \scoreConfidenceBoundByBandwidthCentre[1](z) \;\text{for all }z \in \mathbb{R},\,h > 0\bigr\}.
\end{equation}
On $\goodEvent$,~\eqref{eq:psi_h-psi_0} then ensures that for all $x \in \R$, we have
\begin{equation}
\label{eq:score-upper-confidence}
\psi_0(x) = \inf_{y \leq x} \psi_0(y) \leq \inf_{y \leq x}\,\inf_{h > 0}\,\psi_h(y - h) \leq \inf_{y \leq x}\,\inf_{h > 0}\,\scoreConfidenceBoundByBandwidthCentre[1](y - h) =: \scoreConfidenceBoundShapeConstrained[1](x)
\end{equation}
and similarly $\psi_0(x) \geq \sup_{y \leq x}\,\sup_{h > 0}\,\scoreConfidenceBoundByBandwidthCentre[-1](y + h) =: \scoreConfidenceBoundShapeConstrained[-1](x)$. Therefore, $\bigl[\scoreConfidenceBoundShapeConstrained[-1](\cdot),\scoreConfidenceBoundShapeConstrained[1](\cdot)\bigr]$ becomes our valid multiscale confidence band for $\psi_0(\cdot)$. For notational brevity,
we use an index $\omega \in \{-1,1\}$ below to define lower and upper confidence limits respectively in a unified way.

We proceed to construct confidence bands satisfying~\eqref{eq:simultaneous-validity} using kernel density estimators and their derivatives. For $x_0 \in \R$, define
\[
\hat{f}_{n,h}(x_0) := \frac{1}{n}\sum_{i=1}^n \kernelFunction[h](x_0 - X_i),
\]
and denote by $\hat{f}_{n,h}'(x_0)$ denote the (weak) derivative of $x \mapsto \hat{f}_{n,h}(x)$ at $x_0$. Write $\empiricalProb$ for the empirical distribution of $X_1,\dotsc,X_n$, so that $\empiricalProb(B) := n^{-1}\sum_{i=1}^n \one_{\{X_i \in B\}}$ for Borel sets $B \subseteq \R$. Given $\delta \in (0,1)$,
let $\epsilonByNDeltaMod := 2\log(n^2/\delta)/(9n) = 2\varepsilon_{n,\delta}/9$. For $x_0 \in \R$, $h > 0$, $\omega \in \{-1,1\}$ and $\avgDerivSignNoArg \equiv \avgDerivSign{n}{h}(x_0) := \sgn\bigl(\hat{f}_{n,h}'(x_0)\bigr)$,
three key quantities in the construction are
\begin{align*}
\scoreConfidenceBoundByBandwidthWidth(x_0) &:=3\sqrt{\epsilonByNDeltaMod\empiricalProb([x_0-h, x_0+h]) + \epsilonByNDeltaMod^2} + \frac{1 + 4\epsilonByNDeltaMod}{n},\\
\scoreConfidenceBoundByBandwidthCentreMagnitudeNumerator(x_0) &:= (1 - 2\epsilonByNDeltaMod)\,h^2|\hat{f}_{n,h}'(x_0)| + 2\omega\scoreConfidenceBoundByBandwidthWidth(x_0),\\
\scoreConfidenceBoundByBandwidthCentreDenominator(x_0) &:= h\biggl\{(1 - 2\epsilonByNDeltaMod)\,h\hat{f}_{n,h}(x_0) + \omega\scoreConfidenceBoundByBandwidthWidth(x_0) + \Bigl(3 - \frac{4}{n}\Bigr)\epsilonByNDeltaMod - \frac{1}{n}\biggr\}.
\end{align*}
With $c_{n,\delta} := (1 + 4\epsilonByNDeltaMod)(1 - 2/n)$, these choices are designed to ensure that when $n \geq 3$,
\[
\hat{\mathsf{M}}_{n,\delta,h}(\cdot) := \biggl[\frac{\avgDerivSignNoArg\scoreConfidenceBoundByBandwidthCentreMagnitudeNumerator[-\avgDerivSignNoArg](\cdot)}{c_{n,\delta}},\frac{\avgDerivSignNoArg\scoreConfidenceBoundByBandwidthCentreMagnitudeNumerator[\avgDerivSignNoArg](\cdot)}{c_{n,\delta}}\biggr], \quad \quad \hat{\mathsf{D}}_{n,\delta,h}(\cdot) := \biggl[\frac{\scoreConfidenceBoundByBandwidthCentreDenominator[-1](\cdot)}{c_{n,\delta}},\frac{\scoreConfidenceBoundByBandwidthCentreDenominator[1](\cdot)}{c_{n,\delta}}\biggr]
\]
are $(1 - \delta)$ confidence bands for $h^2f_h'$ and $h^2f_h$ respectively that are simultaneously valid over all $h > 0$; see Lemma~\ref{lem:confidenceBandStatisticalErrorControl}. The theory to derive these confidence bands relies on a key concentration inequality for integrals of unimodal functions with respect to $\probDistribution$, given in Lemma~\ref{lem:uniformBoundedFunctions}.
The next step is to combine these bands to yield a confidence band for $\psi_h$. Indeed, let 
\begin{align*}
\scoreConfidenceBoundByBandwidthCentre[1](x_0) &:= \sup\,\bigl\{m/d : m \in \hat{\mathsf{M}}_{n,\delta,h}(x_0),\,d \in \hat{\mathsf{D}}_{n,\delta,h}(x_0)\bigr\} \\
\scoreConfidenceBoundByBandwidthCentre[-1](x_0) &:= \inf\,\bigl\{m/d : m \in \hat{\mathsf{M}}_{n,\delta,h}(x_0),\,d \in \hat{\mathsf{D}}_{n,\delta,h}(x_0)\bigr\},
\end{align*}
so that for $\omega \in \{-1,1\}$, we have
\[
\scoreConfidenceBoundByBandwidthCentre(x_0) := 
\begin{cases} \avgDerivSignNoArg\scoreConfidenceBoundByBandwidthCentreMagnitudeNumerator[\omega\avgDerivSignNoArg](x_0)/\scoreConfidenceBoundByBandwidthCentreDenominator[-\omega\avgDerivSignNoArg](x_0) &\text{ if }\scoreConfidenceBoundByBandwidthCentreDenominator[-1](x_0) > 0\text{ and }\scoreConfidenceBoundByBandwidthCentreMagnitudeNumerator[-1](x_0) > 0\\
\avgDerivSignNoArg\scoreConfidenceBoundByBandwidthCentreMagnitudeNumerator[\omega \avgDerivSignNoArg](x_0)/\scoreConfidenceBoundByBandwidthCentreDenominator[-1](x_0) &\text{ if }\scoreConfidenceBoundByBandwidthCentreDenominator[-1](x_0) > 0\text{ and }\scoreConfidenceBoundByBandwidthCentreMagnitudeNumerator[-1](x_0) \leq 0\\
\omega\times \infty &\text{ otherwise}.
\end{cases}
\]
Then $\bigl[\scoreConfidenceBoundByBandwidthCentre[-1](\cdot),\scoreConfidenceBoundByBandwidthCentre[1](\cdot)\bigr]$ is a $(1 - \delta)$-level confidence band for $\psi_h$. 
Finally, as described at the start of this section, we convert this into a confidence band for $\psi_0$ by taking
\[
\hat{\mathcal{I}}_{n,\delta}(x_0) := \bigcap_{h > 0} \bigcap_{h' \geq h}\,\bigl[\hat{\psi}_{n,\delta,h}^{(-1)}(x_0 + h'),\hat{\psi}_{n,\delta,h}^{(1)}(x_0 - h')\bigr]
\]
to be our final multiscale confidence interval for $\psi_0(x_0)$. Aggregating over all bandwidths in this way seeks to ensure that the width of the interval adapts optimally to the local regularity of $\psi_0$ at $x_0$. More succinctly, we can write $\hat{\mathcal{I}}_{n,\delta}(x_0) = \bigl[\scoreConfidenceBoundShapeConstrained[-1](x_0), \scoreConfidenceBoundShapeConstrained[1](x_0)\bigr]$, where the expressions for $\scoreConfidenceBoundShapeConstrained[1](x_0),\scoreConfidenceBoundShapeConstrained[-1](x_0)$ in \eqref{eq:score-upper-confidence} and the line below are equivalent to
\begin{align}
\label{eq:scoreEstimateShapeConstrainedInfimum}
\scoreConfidenceBoundShapeConstrained(x_0) = \omega\,\inf\bigl\{\omega\,\scoreConfidenceBoundByBandwidthCentre(z) : (z,h) \in \R \times (0,\infty),\,\omega (x_0 - z) \geq h\bigr\}
\end{align}
for $\omega \in \{-1,1\}$.

The desirable coverage and width properties of our confidence band hold simultaneously on a basic high-probability event defined in terms of the Kullback--Leibler divergence between certain Bernoulli distributions.  More precisely, define $\kullbackLeibler \colon [0,1] \times [0,1] \to [0,\infty]$ as follows: for $p \in (0,1)$ and $q \in [0,1]$, let
\begin{align}
\label{eq:kl-bernoulli}
\kullbackLeibler(p,q) :=  p\log\Bigl(\frac{p}{q}\Bigr) + (1-p)\log\Bigl(\frac{1-p}{1-q}\Bigr),
\end{align}
while for $q \in [0,1]$ define $\kullbackLeibler(0,q) := -\log(1-q)$ and $\kullbackLeibler(1,q) := -\log q$. Given $n \in \N$ and $\delta \in (0,1)$, let $\goodEvent$ denote the event that for every interval $A \subseteq \R$, we have
\begin{align}
\label{eq:goodEvent}
\inf_{t \in [0,1]} \kullbackLeibler\biggl(\empiricalProb(A),\Bigl(1 - \frac{2}{n}\Bigr)\,\probDistribution(A) + \frac{2t}{n}\biggr) < \frac{\log(n^2/\delta)}{n}.
\end{align}
The specific definition of the $\kullbackLeibler$ function is tied to a reformulation of~\citet[Theorem~1]{hoeffding1963probability} for bounded random variables (Lemma~\ref{lem:kl-binomial}), whose purpose is to ensure that the right-hand side of~\eqref{eq:goodEvent} does not depend on $P_0(A)$.

The following result confirms that the confidence band $\hat{\mathcal{I}}_{n,\delta}$ achieves the desired coverage.

\begin{prop}
\label{prop:confidenceBandsAreValid}
For any Borel probability distribution $\probDistribution$ on $\mathbb{R}$, any $n \in \mathbb{N}$ and any $\delta \in (0,1)$, we have $\mathbb{P}(\goodEvent) \geq 1 - \delta$. If in addition $P_0$ has density $f_0 \in \mathcal{F}$ and $n \geq 3$, then 
\begin{align*}
\goodEvent \subseteq \bigcap_{x \in \R} \bigl\{\psi_0(x) \in \hat{\mathcal{I}}_{n,\delta}(x)\bigr\}.
\end{align*}
\end{prop}

\subsection{Pointwise and integrated error of a multiscale score estimator}
\label{subsec:multiscale-ptwise}

For each $x_0 \in \R$, a natural estimator of $\psi_0(x_0)$ is formed by choosing the element of the confidence interval $\hat{\mathcal{I}}_{n,\delta}(x_0)$ with smallest absolute value, i.e.
\begin{equation}
\label{eq:scoreEstimateShapeConstrained}
\scoreEstimateShapeConstrained(x_0) := \argmin_{y \in \hat{\mathcal{I}}_{n,\delta}(x_0)} |y| =
\begin{cases}
\scoreConfidenceBoundShapeConstrained[-1](x_0) &\text{ if } \scoreConfidenceBoundShapeConstrained[-1](x_0) > 0 \\
\scoreConfidenceBoundShapeConstrained[1](x_0) &\text{ if } \scoreConfidenceBoundShapeConstrained[1](x_0) < 0 \\
0 & \text{ otherwise}.
\end{cases}
\end{equation}
On the event $\goodEvent$, the fact that $\psi_0$ lies in the confidence band means that the estimation error of $\scoreEstimateShapeConstrained(x_0)$ is at most the width of the confidence interval $\hat{\mathcal{I}}_{n,\delta}(x_0) = \bigl[\scoreConfidenceBoundShapeConstrained[-1](x_0),\scoreConfidenceBoundShapeConstrained[1](x_0)\bigr]$, and also that $|\scoreEstimateShapeConstrained(x_0) - \psi_0(x_0)| \leq |\psi_0(x_0)|$ by~\eqref{eq:scoreEstimateShapeConstrained}. We will bound the signed errors $\omega \bigl(\scoreConfidenceBoundShapeConstrained[\omega](x_0) - \psi_0(x_0)\bigr)$ for $\omega \in \{-1,1\}$. To outline our arguments, for $z \in \R$ and $h > 0$, define 
\begin{equation}
\label{eq:C-zh}
\begin{split}
\COneZH &:= 138 + 35\biggl(\frac{\probDistribution([z-h,z+h]) \vee \epsilonByNDeltaMod}{\{hf_h(z)\} \vee \epsilonByNDeltaMod}\biggr)^{1/2},\\
\CTwoZH &:= 374\max_{\omega \in \{-1,1\}}\biggl(\frac{\probDistribution([z,z+\omega h]) \vee \epsilonByNDeltaMod}{\{hf_h(z)\} \vee \epsilonByNDeltaMod}\biggr)^{1/2}
\end{split}
\end{equation}
and
\begin{align*}
\alpha_{z,h} &:= \COneZH^2\frac{\epsilonByNDeltaMod}{hf_h(z)}, \qquad
\beta_{z,h} &:= 4(\COneZH\vee \CTwoZH)\Bigl(|\psi_h(z)| \vee \frac{1}{h}\Bigr)\biggl\{\Bigl(\frac{\epsilonByNDeltaMod}{hf_h(z)}\Bigr)^{1/2} \vee \frac{\epsilonByNDeltaMod}{hf_h(z)}\biggr\}.
\end{align*}
On the event $\goodEvent$ of Proposition~\ref{prop:confidenceBandsAreValid} and for any $\omega \in \{-1,1\}$, the quantity $\alpha_{z,h}^{1/2}$ controls the relative error of $\scoreConfidenceBoundByBandwidthCentreDenominator(z)/c_{n,\delta}$ as an estimator of $h^2 f_h(z)$, provided that $\alpha_{z,h} \lesssim 1$; see Lemma~\ref{lem:scoreErrNumDenom}, where we also bound the error of $\hat{\tau}\scoreConfidenceBoundByBandwidthCentreMagnitudeNumerator[\omega\hat{\tau}]/c_{n,\delta}$ as an estimator of $h^2 f_h'(z)$. When $\alpha_{z,h} \leq 1/4$,
the quantity $\beta_{z,h}$ turns out to be an upper bound on the stochastic error $\omega\bigl(\scoreConfidenceBoundByBandwidthCentre(z) - \psi_h(z)\bigr)$, which is non-negative on $\goodEvent$ by \eqref{eq:simultaneous-validity}. Thus, to control $\omega\bigl(\scoreConfidenceBoundShapeConstrained[\omega](x_0) - \psi_0(x_0)\bigr)$, we seek to minimise the sum of $\beta_{z,h}$ and the bias $|\psi_h(z) - \psi_0(x_0)|$ over feasible pairs $(z,h) \in \R \times (0,\infty)$, namely those that satisfy $\omega (x_0 - z) \geq h$ and $\alpha_{z,h} \leq 1/4$. To this end, for $\omega \in \{-1,1\}$, let 
\begin{align*}
\mathcal{H}_{f_0}^{(\omega)}(x_0) &:=  
\bigl\{(z,h) \in \R \times (0,\infty) : \omega(x_0 - z) \geq h > 0,\,\alpha_{z,h} \leq 1/4\bigr\},\\
\Delta_{f_0}^{(\omega)}(x_0) &:= \inf_{(z,h) \in \mathcal{H}_{f_0}^{(\omega)}(x_0)}\bigl\{\bigl|\psi_0(z - \omega h) - \psi_0(x_0)\bigr| + \beta_{z,h}\bigr\}.
\end{align*}

\begin{prop}
\label{prop:scorePointwiseErr}
If $f_0 \in \mathcal{F}$, $n \geq 3$ and $\delta \in (0,1)$, then on the event $\goodEvent$, the following bounds hold for all $x_0 \in \R$:
\begin{enumerate}[(a)]
\item $0 \leq \omega \bigl(\scoreConfidenceBoundShapeConstrained[\omega](x_0) - \psi_0(x_0)\bigr)\leq \Delta_{f_0}^{(\omega)}(x_0)$; 
\item $\bigl|\scoreEstimateShapeConstrained(x_0) - \psi_0(x_0)\bigr| \leq \Delta^{(-1)}_{f_0}(x_0) + \Delta^{(1)}_{f_0}(x_0)$; 
\item $\bigl|\scoreEstimateShapeConstrained(x_0) - \psi_0(x_0)\bigr| \leq |\psi_0(x_0)|$.
\end{enumerate}
\end{prop}

On $\goodEvent$, Proposition~\ref{prop:scorePointwiseErr}\emph{(a)} bounds the distance between $\psi_0(x_0)$ and the relevant endpoint of the confidence interval by $\Delta^{(\omega)}_{f_0}(x_0)$, so it follows immediately that in \textit{(b)}, the pointwise error of our estimator is at most $\Delta^{(-1)}_{f_0}(x_0) + \Delta^{(1)}_{f_0}(x_0)$. As mentioned previously, \textit{(c)} is a direct consequence of the definition of the estimator in~\eqref{eq:scoreEstimateShapeConstrained} and the fact that $\psi_0$ lies in the confidence band on $\goodEvent$.

We now outline how the key properties above are used to establish that the integrated error $\mathcal{L}(\scoreEstimateShapeConstrained,\psi_0) = \int_{-\infty}^\infty (\scoreEstimateShapeConstrained - \psi_0)^2 f_0$ adaptively attains the upper bounds in Theorems~\ref{thm:tail-growth} and~\ref{thm:holder-minimax-upper}. On $\goodEvent$, we will apply parts \textit{(b)} and \textit{(c)} of Proposition~\ref{prop:scorePointwiseErr} on complementary subsets of the real line. First, our `bulk' region $[x_{\min},x_{\max}]$ has endpoints $x_{\min} \leq x_{\max}$ such that $F_0(x_{\min}) = 1 - F_0(x_{\max}) \asymp \varepsilon_{n,\delta} \wedge 1$, meaning that it typically contains all but logarithmically many data points. We will apply Proposition~\ref{prop:scorePointwiseErr}\textit{(b)} to every $x_0 \in [x_{\min},x_{\max}]$, for which we show in Lemma~\ref{lem:integrated-error-interval}\textit{(c)} that $\mathcal{H}_{f_0}^{(-1)}(x_0)$ and $\mathcal{H}_{f_0}^{(1)}(x_0)$ are non-empty. On the other hand, this pointwise bound becomes uninformative on the `tail' region $[x_{\min},x_{\max}]^c$ in which data is scarce, but Proposition~\ref{prop:scorePointwiseErr}\textit{(c)} nevertheless protects us against overestimating the magnitude of the score function.  Consequently, on $\goodEvent$, the contribution to the integrated estimation error from the tail region satisfies
\[
\int_{[x_{\min},x_{\max}]^c} (\scoreEstimateShapeConstrained - \psi_0)^2 f_0 \leq \int_{-\infty}^{x_{\min}} \psi_0^2 f_0 + \int_{x_{\max}}^\infty \psi_0^2 f_0,
\]
and our choices of $x_{\min}$ and $x_{\max}$ ensure that this does not dominate the bound obtained from the bulk, where the analysis is more delicate. We further partition $[x_{\min},x_{\max}]$ into subintervals $[a_j,a_{j+1}]$ defined for integer indices $j \in \{-j_{n,\delta},\dotsc,j_{n,\delta} - 1\}$ by
\[
a_j := F_0^{-1}\bigl(\Ind_{\{j \geq 0\}} - 2^{-(|j|+1)}\bigr),
\]
where $j_{n,\delta} \asymp \log_+(1/\varepsilon_{n,\delta})$ is such that $F_0(x_{\min}) = 1 - F_0(x_{\max}) = 2^{-(|j_{n,\delta}| + 1)}$. The sequence $(a_j)$ has several convenient properties arising from the log-concavity of $f_0$ (Lemma~\ref{lem:density-quantile-basics}), including that~$f_0$ varies by at most a multiplicative factor of 2 on each $[a_j,a_{j+1}]$. The purpose of this refinement of the bulk is to exploit upper bounds on $f_0(a_j)$ implied by the class definitions, which become stronger for larger values of $|j|$. Lemma~\ref{lem:integrated-error-interval}\textit{(a)} first makes the inequality in Proposition~\ref{prop:scorePointwiseErr}\textit{(b)} more explicit by showing that on $\goodEvent$, we have for all $j \in \{-j_{n,\delta},\dotsc,j_{n,\delta} - 1\}$ and $x_0 \in [a_j,a_{j+1}]$ that
\begin{align}
& \bigl(\scoreEstimateShapeConstrained(x_0) - \psi_0(x_0)\bigr)^2 f_0(x_0) \leq 2\sum_{\omega \in \{-1,1\}} \Delta_{f_0}^{(\omega)}(x_0)^2 f_0(x_0) \notag \\ 
\label{eq:inf-Prop-pointwise-2}
&\hspace{0.5cm} \asymp \sum_{\omega \in \{-1,1\}} \inf\Bigl\{\bigl(\psi_0(x_0 + \omega h) - \psi_0(x_0)\bigr)^2 f_0(a_j) \vee \frac{\varepsilon_{n,\delta}}{h^3} : \frac{C_0\varepsilon_{n,\delta}}{f_0(x_0)} \leq h \leq \kappa_j\Bigr\} =: \Gamma_{f_0}(x_0),
\end{align}
where $\kappa_j := (a_{j+1} - a_j)/8$ and $C_0$ is a universal constant that we specify in the proof of Lemma~\ref{lem:integrated-error-interval}. Having absorbed the multiplicative weight $f_0(x_0)$ into the infimum in~\eqref{eq:inf-Prop-pointwise-2}, we recognise the two terms as corresponding to density-weighted bias and stochastic error terms respectively, where the latter is the leading-order term arising from $\beta_{z,h}$ for $z = x_0 - \omega h/2$. The form of this bias--variance tradeoff is characteristic of nonparametric derivative estimation problems, where the variance bounds are of order $1/(nh^3)$ as opposed to $1/(nh)$ in conventional density and regression function estimation.

The optimal choice of~$h$ in~\eqref{eq:inf-Prop-pointwise-2} depends on the local smoothness of the log-density, which may vary significantly with $x_0$. To illustrate this, we now characterise for two specific densities $f_0$ the bandwidth $h_\omega(x_0)$ at which the minimal pointwise bound $\Gamma_{f_0}(x_0)$ is attained. We then integrate the latter over $[a_0,a_1] = [F_0^{-1}(1/2),F_0^{-1}(3/4)]$, on which Lemma~\ref{lem:integrated-error-interval} guarantees that $C_0\varepsilon_{n,\delta}/f_0(x_0) \leq \kappa_0 \asymp a_1 - a_0 \asymp 1$ in these examples.

\begin{example}[Standard Gaussian]
The score function $\psi_0$ is given by $\psi_0(x) = -x$ for $x \in \R$, so for $\omega \in \{-1,1\}$ and $h > 0$, we have
\[
\bigl(\psi_0(x_0 - \omega h) - \psi_0(x_0)\bigr)^2 f_0(x_0) \vee \frac{\varepsilon_{n,\delta}}{h^3} = h^2 f_0(x_0) \vee \frac{\varepsilon_{n,\delta}}{h^3}.
\]
This is minimised over $[C_0\varepsilon_{n,\delta}/f_0(x_0),\kappa_j]$ at $h_\omega(x_0) \asymp \bigl(\varepsilon_{n,\delta}/f_0(x_0)\bigr)^{1/5} \wedge \kappa_j$. Thus, $\Gamma_{f_0}(x_0) \asymp \varepsilon_{n,\delta}^{2/5}f_0(x_0)^{3/5} \vee \varepsilon_{n,\delta} \kappa_j^{-3}$ when $C_0\varepsilon_{n,\delta}(C_0^{1/4} \vee \kappa_j^{-1}) \leq f_0(x_0)$, so in particular $\Gamma_{f_0}(x_0) \asymp \varepsilon_{n,\delta}^{2/5}$ for $x_0 \in [a_0,a_1]$. Therefore,
\[
\int_{a_0}^{a_1} \Gamma_{f_0}(x_0) \laplaced x_0 \asymp \varepsilon_{n,\delta}^{2/5}.
\]
Here, the smoothness of $\psi_0$ is homogeneous over $\R$, so to obtain a tight bound on the integral above, it suffices to bound $\Gamma_{f_0}^{(\omega)}(x_0)$ for $x \in [a_0,a_1]$ by a quantity that does not depend on $x_0$. 
\end{example}

\begin{example}[Standard Laplace]
\label{ex:laplace-ptwise}
The score function is $\psi_0(\cdot) = -\sgn(\cdot)$ on $\R$, so if $x_0 \geq 0$ and $\omega = 1$, then
\begin{align*}
\bigl(\psi_0(x_0 - \omega h) - \psi_0(x_0)\bigr)^2 f_0(x_0) \vee \frac{\varepsilon_{n,\delta}}{h^3} = 4f_0(x_0)\one_{\{h > x_0\}} \vee \frac{\varepsilon_{n,\delta}}{h^3}.
\end{align*}
This is minimised over $(0,\kappa_j]$ at $h_\omega(x_0) \asymp \max\bigl\{x_0,\bigl(\varepsilon_{n,\delta}/\{4f_0(x_0)\}\bigr)^{1/3}\bigr\} \wedge \kappa_j$. On the other hand, if $\omega = -1$, then $h_\omega(x_0) = \kappa_j$. Analogous conclusions hold for $x_0 < 0$ by symmetry. Thus, for $x_0 \in \R$ such that $C_0\varepsilon_{n,\delta}(2C_0^{1/2} \vee \kappa_j^{-1}) \leq f_0(x_0)$, we have
\[
\Gamma_{f_0}(x_0) \asymp \Bigl\{\frac{\varepsilon_{n,\delta}}{|x_0|^3} \wedge f_0(x_0)\Bigr\} \vee \frac{\varepsilon_{n,\delta}}{\kappa_j^3}.
\]
In particular, when $j = 0$, we have $h_*(x_0) = \min_{\omega \in \{-1,1\}} h_\omega(x_0) \asymp \varepsilon_{n,\delta}^{1/3} \vee |x_0|$ and $\Gamma_{f_0}(x_0) \asymp (\varepsilon_{n,\delta}/|x_0|^3) \wedge 1$, so
\[
\int_{a_0}^{a_1} \Gamma_{f_0}(x_0) \laplaced x_0 \asymp \varepsilon_{n,\delta}^{1/3}.
\]
A significant contribution to this integral comes from the interval $[0,\varepsilon_{n,\delta}^{1/3}]$, on which the pointwise error bound is of order 1 due to the discontinuity of $\psi_0$ at 0. On the other hand, for $x_0 \geq \varepsilon_{n,\delta}^{1/3}$, the minimising bandwidth $h_*(x_0) \asymp |x_0|$ increases with the distance from $x_0$ to the discontinuity because $\psi_0$ is constant on $(0,\infty)$, and the bound on $\Gamma_{f_0}(x_0)$ improves. This simple example of heterogeneous smoothness illustrates the importance of spatially adaptive bandwidth selection to achieve optimal estimation accuracy over the subclasses in Section~\ref{sec:main-results}.
\end{example}

More generally, when $f_0 \in \mathcal{F}_{\beta,L}$ for $\beta \in [1,2)$ and $L > 0$, integrating a na\"ive pointwise inequality of the form
\begin{align}
\label{eq:sub-opt-bound}
\Gamma_{f_0}(x_0) \lesssim \inf\Bigl\{L^2h^{2(\beta-1)}f_0(x_0) \vee \frac{\varepsilon_{n,\delta}}{h^3} : \frac{C_0\varepsilon_{n,\delta}}{f_0(x_0)}\leq h \leq \kappa_j\Bigr\}
\end{align}
over $[a_j,a_{j+1}]$ leads to a suboptimal bound of order $\varepsilon_{n,\delta}^{(2\beta - 1)/(2\beta + 1)}$ on the integrated estimation error of $\hat{\psi}_{n,\delta}$ over this interval.  The issue is that~\eqref{eq:sub-opt-bound} does not exploit the heterogenous smoothness induced by the monotonicity of the score function; see the discussion in Section~\ref{subsec:holder}. We instead bound $\int_{a_j}^{a_{j+1}} \Gamma_{f_0}(x_0) \laplaced x_0$ using the refined approach in Proposition~\ref{prop:Holder-basic}, which yields the improved rate of order $\varepsilon_{n,\delta}^{\beta/(2\beta + 1)}$ seen in Theorem~\ref{thm:holder-minimax-upper}.

\bibliographystyle{apalike}
\bibliography{bib}
\newpage

\appendix

\section*{Appendix}

The proofs of the main results in Sections~\ref{sec:main-results} and~\ref{sec:multiscale} appear in Appendices~\ref{sec:proofs-main-results} and \ref{sec:proofs-multiscale} respectively. These two proof sections have minimal overlap and may be read independently: the upper bounds in Section~\ref{subsec:proofs-upper-bds} for the multiscale estimator~\eqref{eq:scoreEstimateShapeConstrained} depend only on the pointwise properties stated in Proposition~\ref{prop:scorePointwiseErr}, not on the specific construction studied in Appendix~\ref{sec:proofs-multiscale}. Auxiliary results used throughout are collected in Appendix~\ref{sec:auxiliary}.

\section{Proofs for Section~\ref{sec:main-results}}
\label{sec:proofs-main-results}

We present separately the proofs of the upper and lower bounds in Sections~\ref{subsec:proofs-upper-bds} and~\ref{subsec:proofs-lower-bds} respectively.

\subsection{Proofs of upper bounds}
\label{subsec:proofs-upper-bds}

First, we establish a result of independent interest that bounds the integrated error of a pointwise optimal estimator over a compact subinterval, exploiting the heterogeneous smoothness induced by a H\"older condition combined with monotonicity (or more generally a bounded variation constraint). Fix $\varepsilon > 0$ and let $g \colon [a - h_{\max}, b + h_{\max}] \to \R$ be a decreasing function, where $-\infty < a < b < \infty$ and $h_{\max} > 0$. For $x \in [a,b]$, $\omega \in \{-1,1\}$ and $h \in (0,h_{\max}]$, define 
\begin{align*}
\xi_{x,\omega}(h) &:= \omega\bigl(g(x) - g(x + \omega h)\bigr) \vee \Bigl(\frac{\varepsilon}{h^3}\Bigr)^{1/2}, \qquad \xi_x(h) := \xi_{x,-1}(h) \vee \xi_{x,1}(h).
\end{align*}

\begin{prop}
\label{prop:Holder-basic}
Suppose that $g \colon [a - h_{\max}, b + h_{\max}] \to \R$ is decreasing and that there exist $\beta \in [1,2]$ and $L_* > 0$ such that $|g(x) - g(y)| \leq L_*|x - y|^{\beta - 1}$ for all $x,y \in [a - h_{\max}, b + h_{\max}]$. Let
\begin{alignat*}{2}
h_{\beta,L_*} &:= \Bigl(\frac{\varepsilon}{L_*^2}\Bigr)^{1/(2\beta + 1)}, &\qquad h_{\min} &:= h_{\beta,L_*} \wedge h_{\max}, \\
V &:= g(a - h_{\max}) - g(b + h_{\max}), &\qquad \Gamma_{\max} &:= \Bigl(V^2 \wedge \frac{\varepsilon}{h_{\beta,L_*}^3}\Bigr) \vee \frac{\varepsilon}{h_{\max}^3}.
\end{alignat*}
Then $x \mapsto \inf_{h \in [h_{\min},h_{\max}]} \xi_x(h)^2 =: \Gamma(x)$ is a Borel measurable function from $[a,b]$ to $[0,\Gamma_{\max}]$, with
\begin{equation}
\label{eq:Holder-variation-measure}
\lambda\bigl(\{x \in [a,b] : \Gamma(x) > \zeta\}\bigr) \leq \frac{3V\varepsilon^{1/3}}{\zeta^{5/6}} \wedge (b - a)
\end{equation}
for all $\zeta \in [0,\Gamma_{\max}]$. Consequently,
\begin{align}
\label{eq:Holder-basic}
\int_a^b \Gamma(x) \laplaced x &\leq 18V\bigl\{(V\varepsilon)^{1/3} \wedge (L_*\varepsilon^\beta)^{1/(2\beta + 1)}\bigr\} + \frac{(b - a)\varepsilon}{h_{\max}^3}.
\end{align}
\end{prop}

It is instructive to compare this with the bound obtained using the H\"older condition alone, in the absence of monotonicity. To this end,
suppose only that $|g(x) - g(y)| \leq L_*|x - y|^{\beta - 1}$ for all $x,y \in [a - h_{\beta,L_*}, b + h_{\beta,L_*}]$, and moreover that $h_{\beta,L_*} \leq h_{\max}$ for ease of comparison. Then using the pointwise inequality $\xi_x(h_{\beta,L_*})^2 \leq (L_*h_{\beta,L_*}^{\beta-1})^2 \vee (\varepsilon/h_{\beta,L_*}^3) = \varepsilon/h_{\beta,L_*}^3$ for all $x \in [a,b]$, we have 
\begin{equation}
\label{eq:Holder-naive}
\int_a^b \inf_{h \in [h_{\min},h_{\max}]}\xi_x(h)^2 \laplaced x \leq \frac{(b - a)\varepsilon}{h_{\beta,L_*}^3} = (b - a)L_*^{6/(2\beta + 1)}\varepsilon^{2(\beta - 1)/(2\beta + 1)}.
\end{equation}
When $\varepsilon = 1/n$, this bound yields the optimal rate $n^{-2(\beta - 1)/(2\beta + 1)}$ for nonparametric estimation of the derivative of a univariate $\beta$-H\"older function \citep[e.g.][]{stone1980optimal,stone1982optimal}; see also Proposition~\ref{prop:lower-bd-no-shape-c}. The two terms in the definition of $\xi_{x,\omega}(h)$ correspond to the pointwise bias and standard deviation respectively of a kernel estimator with bandwidth $h$. On the other hand, both of the terms $V(L_*\varepsilon^\beta)^{1/(2\beta + 1)}$ and $(b - a)\varepsilon/h_{\max}^3$ in~\eqref{eq:Holder-basic} can yield tighter bounds than $(b - a)\varepsilon/h_{\beta,L_*}^3$ when $h_{\max} \gg h_{\beta,L_*}$ and the vertical increment $V$ is much smaller than its maximal value $L_*(b - a + 2h_{\max})^{\beta-1}$, as is typically the case in our applications of Proposition~\ref{prop:Holder-basic}. In particular, when $\varepsilon = 1/n$, the first of the aforementioned terms in~\eqref{eq:Holder-basic} yields a strictly faster rate $n^{-\beta/(2\beta + 1)}$ than that in~\eqref{eq:Holder-naive} for all $\beta \in [1,2)$. For instance, when $\beta = 1$ and $b - a$, $L_* = V$ and $h_{\max}$ are all regarded as universal constants, the bound in~\eqref{eq:Holder-basic} scales as $n^{-1/3}$ while that in~\eqref{eq:Holder-naive} is of constant order. 

\begin{proof}
First, we show that $\Gamma(x) \leq \Gamma_{\max}$ for all $x \in [a,b]$. Indeed, since $g$ is decreasing,
\[
\Gamma(x) \leq \xi_x(h_{\max})^2 \leq V^2 \vee \frac{\varepsilon}{h_{\max}^3}.
\]
Moreover, because $(L_*h^{\beta - 1})^2 \leq \varepsilon/h^3$ for all $h \in (0,h_{\beta,L_*}]$, it follows from the H\"older condition that
\[
\Gamma(x) \leq \xi_x(h_{\min})^2 \leq (L_*h_{\min}^{\beta - 1})^2 \vee \frac{\varepsilon}{h_{\min}^3} = \frac{\varepsilon}{h_{\min}^3} = \frac{\varepsilon}{h_{\beta,L_*}^3} \vee \frac{\varepsilon}{h_{\max}^3},
\]
so taking the minimum of these bounds yields $\Gamma(x) \leq \Gamma_{\max}$, as required.

To verify that $\Gamma \colon [a,b] \to [0,\Gamma_{\max}]$ is Borel measurable, 
we define $\mathcal{H} := \bigl((h_{\min},h_{\max}] \cap \Q\bigr) \cup \{h_{\min}\}$ and show that
\begin{equation}
\label{eq:Gamma-measurable}
\Gamma(x) = \inf_{h \in \mathcal{H}} \xi_x(h)^2
\end{equation}
for every $x \in [a,b]$. Indeed, given $\eta > 0$, we can find $h' \in [h_{\min},h_{\max}]$ such that $\xi_x(h')^2 < \Gamma(x) + \eta$. Since $h \mapsto \omega\bigl(g(x) - g(x + \omega h)\bigr)$ is increasing for $\omega \in \{-1,1\}$, since $h \mapsto (\varepsilon/h^3)^{1/2}$ is continuous, and since $\mathcal{H}$ is dense in $[h_{\min},h_{\max}]$, there exists $h'' \in \mathcal{H}$ with $h'' \leq h'$ such that $\xi_x(h'')^2 < \Gamma(x) + \eta$. Since $\eta$ was arbitrary, this establishes~\eqref{eq:Gamma-measurable}. Thus, $x \mapsto \Gamma(x)$ is a countable pointwise infimum of Borel measurable functions $x \mapsto \xi_x(h)$ over $h \in \mathcal{H}$, and hence is Borel measurable.

To derive the bound~\eqref{eq:Holder-variation-measure} on the Lebesgue measure of $\mathcal{S}_\zeta := \{x \in [a,b] : \Gamma(x) > \zeta\}$ for $\zeta \in [0,\Gamma_{\max}]$, first consider $\zeta \geq \varepsilon/h_{\max}^3$. Since $\Gamma_{\max} \leq \varepsilon/h_{\min}^3$, we have $h_\zeta := (\varepsilon/\zeta)^{1/3} \in [h_{\min},h_{\max}]$, so if $\max_{\omega \in \{-1,1\}} \omega\bigl(g(x) - g(x + \omega h_\zeta)\bigr) \leq \sqrt{\zeta}$, then $\Gamma(x) \leq \xi_x(h_\zeta)^2 \leq \zeta$. Thus, with
\begin{align*}
\mathcal{X}(\zeta) := \bigl\{x \in [a - h_{\max}, b + h_{\max} - h_\zeta] : g(x) - g(x + h_\zeta) > \sqrt{\zeta}\bigr\},
\end{align*}
each $x' \in \mathcal{S}_\zeta$ satisfies either $x' \in \mathcal{X}(\zeta)$ or $x' - h_\zeta \in \mathcal{X}(\zeta)$, i.e.~$\mathcal{S}_\zeta \subseteq \bigcup_{x \in \mathcal{X}(\zeta)} \{x, x + h_\zeta\}$. Let $\mathcal{X}_\circ(\zeta)$ be a maximal $h_\zeta$-separated subset of $\mathcal{X}(\zeta)$, so that any distinct $x,x' \in \mathcal{X}_\circ(\zeta)$ satisfy $|x-x'| > h_\zeta$. Then $\mathcal{X}(\zeta) \subseteq \bigcup_{x \in \mathcal{X}_\circ(\zeta)} [x - h_\zeta, x + h_\zeta]$, 
and
\begin{align*}
\sqrt{\zeta}|\mathcal{X}_\circ(\zeta)| < \sum_{x \in \mathcal{X}_\circ(\zeta)} \bigl(g(x) - g(x + h_\zeta)\bigr) \leq g(a - h_{\max}) - g(b + h_{\max}) = V,
\end{align*}
since $g$ is decreasing. Consequently,
\begin{align*}
\mathcal{S}_\zeta \subseteq \bigcup_{x \in \mathcal{X}(\zeta)} \{x, x + h_\zeta\} \subseteq \bigcup_{x \in \mathcal{X}_\circ(\zeta)} [x - h_\zeta,x + 2h_\zeta],
\end{align*}
so $\lambda(\mathcal{S}_\zeta) \leq 3h_\zeta|\mathcal{X}_\circ(\zeta)| < 3h_\zeta V/\sqrt{\zeta} = 3V\varepsilon^{1/3}/\zeta^{5/6}$ for $\zeta \in [\varepsilon/h_{\max}^3,\Gamma_{\max}]$. Moreover, $\lambda(\mathcal{S}_\zeta) \leq b - a$ trivially for all $\zeta \in [0,\Gamma_{\max}]$.
This yields~\eqref{eq:Holder-variation-measure}, and hence by Fubini's theorem,
\begin{align*}
\int_a^b \Gamma(x) \laplaced x = \int_a^b \int_0^{\Gamma_{\max}} \one_{\{\Gamma(x) > \zeta\}} \laplaced \zeta \laplaced x  &= \int_0^{\Gamma_{\max}} \lambda(\mathcal{S}_\zeta) \laplaced \zeta \leq \frac{(b - a)\varepsilon}{h_{\max}^3} + \int_{\varepsilon/h_{\max}^3}^{\Gamma_{\max}} \frac{3V\varepsilon^{1/3}}{\zeta^{5/6}} \laplaced \zeta \\
&\leq \frac{(b - a)\varepsilon}{h_{\max}^3} + 18V\varepsilon^{1/3}\Gamma_{\max}^{1/6} \Ind_{\{\Gamma_{\max} > \varepsilon/h_{\max}^3\}} \\
&\leq \frac{(b - a)\varepsilon}{h_{\max}^3} + 18V\bigl\{(V\varepsilon)^{1/3} \wedge (L_*\varepsilon^\beta)^{1/(2\beta + 1)}\bigr\},
\end{align*}
where the final inequality follows from the definition of $\Gamma_{\max}$. Therefore,~\eqref{eq:Holder-basic} holds.
\end{proof}

By \citet[Proposition~A.1(c)]{bobkov1996extremal}, a Lebesgue density $f_0$ on $\R$ is log-concave if and only if the corresponding \textit{density quantile function} $J_0 := f_0 \circ F_0^{-1}$ \citep{parzen1979nonparametric} is concave on $[0,1]$, where $F_0$ and $F_0^{-1} \colon [0,1] \to [-\infty,\infty]$ denote the corresponding distribution and quantile functions respectively. We will use the following elementary facts \citep[e.g.][Remark~3 and Lemma~17]{feng24asm}.

\begin{lemma}
\label{lem:density-quantile-basics}
If $f_0$ is an absolutely continuous Lebesgue density on $\R$ with score function $\psi_0$, then
\begin{enumerate}[(a)]
\item $J_0(0) = J_0(1) = 0$ and $J_0$ is absolutely continuous with $J_0' = \psi_0 \circ F_0^{-1}$ Lebesgue almost everywhere on $(0,1)$.  
\item $\int_x^{x'} \psi_0^2 f_0 = \int_{F_0(x)}^{F_0(x')} (J_0')^2$ for all $x,x' \in \R$.
\item $F_0^{-1}(v) - F_0^{-1}(u) = \int_u^v 1/J_0$ for $u,v \in (0,1)$.
\end{enumerate}
\end{lemma}

The proofs of our upper bounds rely on applying Proposition~\ref{prop:Holder-basic} on subintervals $\{[a_j,a_{j+1}]:j \in \Z\}$ defined as follows. 
Given a continuous log-concave density $f_0$, we henceforth take $\psi_0 = \phi_0^{(\mathrm{R})}$ and $J_0' \equiv J_0^{(\mathrm{R})} = \psi_0 \circ F_0^{-1}$ for concreteness, although our analysis does not depend on any particular version of the score function provided that it is decreasing. For $j \in \Z$, define
\[
u_j := 
\begin{cases}
2^{j-1} \quad &\text{if }j \leq 0 \\
1 - 2^{-(j+1)} \quad &\text{otherwise},
\end{cases}
\qquad\quad
a_j := F_0^{-1}(u_j),
\]
so that $u_j \wedge (1-u_j)=2^{-|j|-1}$, and by the definition of $J_0$ and Lemma~\ref{lem:density-quantile-basics}\textit{(a)},
\[
F_0(a_j) = u_j, \quad f_0(a_j) = J_0(u_j) \quad\text{and}\quad \psi_0(a_j) = J_0'(u_j).
\]
The sequence $(a_j)_{j \in \Z}$ has the following key properties.

\begin{lemma}
\label{lem:partition-properties}
For any continuous log-concave density $f_0$ and $j \in \Z$, we have
\begin{align}
\label{eq:partition-property-1}
|\psi_0(a_j)| = |J_0'(u_j)| &\leq \frac{J_0(u_j)}{u_j \wedge (1 - u_j)} = 2^{|j| + 1}f_0(a_j), \\
\label{eq:partition-property-2}
\sup_{x',x'' \in [a_j,a_{j+1}]} \frac{f_0(x')}{f_0(x'')} &= \sup_{u',u'' \in [u_j,u_{j+1}]} \frac{J_0(u')}{J_0(u'')} \leq 2, \\
\label{eq:partition-property-3}
\frac{a_j - a_{j-1}}{8} &\leq a_{j+1} - a_j \leq 8(a_j - a_{j-1}), \\
\label{eq:partition-property-4}
\frac{2^{-(|j|+3)}}{f_0(a_j)} \leq a_{j+1} - a_j &\leq \frac{2^{-|j|}}{f_0(a_j)} \wedge \frac{2^{-|j+1|}}{f_0(a_{j+1})} \leq \frac{2}{|\psi_0(a_j)| \vee |\psi_0(a_{j+1})|}.
\end{align}
\end{lemma}

\begin{proof}
Both~\eqref{eq:partition-property-1} and~\eqref{eq:partition-property-2} rely on the fact that $J_0$ is concave \citep[Proposition~A.1(c)]{bobkov1996extremal}, with $J_0(0) = J_0(1) = 0$ by Lemma~\ref{lem:density-quantile-basics}\textit{(a)}. It follows that $0 = J_0(0) \leq J_0(u) - uJ_0'(u)$ and $0 = J_0(1) \leq J_0(u) + (1 - u)J_0'(u)$ for $u \in (0,1)$, so
\[
|J_0'(u)| = \max\{J_0'(u),-J_0'(u)\} \leq \frac{J_0(u)}{u} \vee \frac{J_0(u)}{1 - u},
\]
and~\eqref{eq:partition-property-1} holds.  Now fix $j \in \Z$ and any $x',x'' \in [a_j,a_{j+1}]$ such that $x' \leq x''$. Then $u' := F_0(x')$ and $u'' := F_0(x'')$ satisfy $u' \leq u'' \leq u_{j+1} \leq 2u_j \leq 2u'$ and $1 - u'' \leq 1 - u' \leq 1 - u_j \leq 2(1 - u_{j+1}) \leq 2(1 - u'')$, 
so
\begin{align*}
\frac{f_0(x')}{u'} = \frac{J_0(u')}{u'} &\geq \frac{J_0(u'')}{u''} = \frac{f_0(x'')}{u''} \geq \frac{f_0(x'')}{2u'}, \\
\frac{f_0(x'')}{1 - u''} = \frac{J_0(u'')}{1 - u''} &\geq \frac{J_0(u')}{1 - u'} = \frac{f_0(x')}{1 - u'} \geq \frac{f_0(x')}{2(1 - u'')}.
\end{align*}
Thus, $1/2 \leq f_0(x')/f_0(x'') = J_0(u')/J_0(u'') \leq 2$, which establishes~\eqref{eq:partition-property-2}. Consequently, $1/2 \leq J_0(u)/J_0(u_j) \leq 2$ for all $u \in [u_{j-1},u_{j+1}]$. Moreover, 
$u_{j+1} - u_j = 2^{j-1}$ for $j \leq -1$ and $u_{j+1} - u_j = 2^{-(j+2)}$ for $j \geq 0$, so $1/2 \leq (u_{j+1} - u_j)/(u_j - u_{j-1}) \leq 2$ for all $j \in \Z$. Hence by Lemma~\ref{lem:density-quantile-basics}\textit{(c)},
\begin{align}
\label{eq:aj-differences-1}
a_j - a_{j-1} &= \int_{u_{j-1}}^{u_j} \frac{1}{J_0} \leq \frac{2(u_j - u_{j-1})}{J_0(u_j)} \leq \frac{4(u_{j+1} - u_j)}{J_0(u_j)} \leq \int_{u_j}^{u_{j+1}} \frac{8}{J_0} = 8(a_{j+1} - a_j), \\
\label{eq:aj-differences-2}
a_{j+1} - a_j &= \int_{u_j}^{u_{j+1}} \frac{1}{J_0} \leq \frac{2(u_{j+1} - u_j)}{J_0(u_j)} \leq \frac{4(u_j - u_{j-1})}{J_0(u_j)} \leq \int_{u_{j-1}}^{u_j} \frac{8}{J_0} = 8(a_j - a_{j-1}),
\end{align}
which yields~\eqref{eq:partition-property-3}. Moreover, by~\eqref{eq:partition-property-2}, $J_0(u_{j+1})/2 \leq J_0(u_j) \leq 2J_0(u_{j+1})$ for all $j \in \Z$.  Thus, by applying the final inequality in~\eqref{eq:aj-differences-1}, followed by the first inequality in~\eqref{eq:aj-differences-2} (when $j \leq -1$) and the first inequality in~\eqref{eq:aj-differences-1} with a shifted index (when $j \geq 0$), we obtain
\begin{align*}
\frac{2^{-(|j|+3)}}{J_0(u_j)} \leq \frac{u_{j+1} - u_j}{2J_0(u_j)} \leq a_{j+1} - a_j &\leq
\begin{cases}
\dfrac{2(u_{j+1} - u_j)}{J_0(u_j)} 
= \dfrac{2^{-j}}{J_0(u_j)}
&\text{if }j \leq -1 \\[10pt]
\dfrac{2(u_{j+1} - u_j)}{J_0(u_{j+1})} 
= \dfrac{2^{-(j+1)}}{J_0(u_{j+1})}
&\text{if }j \geq 0
\end{cases} \\
&= \dfrac{2^{-|j|}}{J_0(u_j)} \wedge \dfrac{2^{-|j+1|}}{J_0(u_{j+1})} \leq 
\frac{2}{|\psi_0(a_j)|} \wedge \frac{2}{|\psi_0(a_{j+1})|},
\end{align*}
where the final bound follows from~\eqref{eq:partition-property-1}. This completes the proof of~\eqref{eq:partition-property-4}.
\end{proof}
Since $J_0'$ is decreasing, it turns out that the integrals $\int_{u_j}^{u_k} J_0'$ and $\int_{u_j}^{u_k} (J_0')^2$ for $j,k \in \Z$ are always within a universal constant factor of their Riemann sum approximations with respect to the grid $(u_\ell)_{\ell=j}^k$. We will use the following two explicit bounds to incorporate constraints on the Fisher information in the proofs of our main results.

\begin{lemma}
\label{lem:J0-fisher-information-sums}
For $j,k \in \Z$ such that $j < k$, we have
\begin{align}
\label{eq:density-quantile-alt}
|J_0(u_j) - J_0(u_k)| &\leq \sum_{\ell=j}^k 2^{-|\ell|}|J_0'(u_\ell)|, \\
\label{eq:fisher-information-sum}
i(f_0) &\geq \frac{1}{8}\sum_{\ell \in \Z} 2^{-|\ell|}J_0'(u_{\ell})^2.
\end{align}
\end{lemma}

\begin{proof}
Since $2^{-(|\ell| + 2)} \leq u_\ell - u_{\ell - 1} \leq 2^{-(|\ell| + 1)}$ for $\ell \in \Z$ and $J_0'$ is decreasing on $(0,1)$, we have
\begin{align*}
|J_0(u_j) - J_0(u_k)|
\leq \sum_{\ell=j+1}^k \biggl|\int_{u_{\ell-1}}^{u_\ell} J_0'\biggr| &\leq \sum_{\ell=j+1}^k (u_\ell - u_{\ell-1})\bigl(|J_0'(u_{\ell-1})| \vee |J_0'(u_\ell)|\bigr) \leq \sum_{\ell=j}^k 2^{-|\ell|}|J_0'(u_\ell)|,
\end{align*}
so~\eqref{eq:density-quantile-alt} holds. Next, $J_0 = f_0 \circ F_0^{-1}$ is concave with $\sup_{u \in (0,1)} J_0(u) > 0$, and $J_0(0) = J_0(1) = 0$ by Lemma~\ref{lem:density-quantile-basics}\textit{(a)}, so $\lim_{u \searrow 0} J_0'(u) > 0 > \lim_{u \nearrow 1} J_0'(u)$ and hence $m := \max\{j \in \Z : J_0'(u_j) > 0\}$ is finite. Since $J_0'$ is decreasing, for $\ell \in \Z$ with $\ell \neq m + 1$, we have $|J_0'(u)| \geq |J_0'(u_{\ell-1})| \wedge |J_0'(u_{\ell})|$ for all $u \in (u_{\ell-1}, u_{\ell})$. Therefore, by Lemma~\ref{lem:density-quantile-basics}\textit{(b)},
\begin{align*}
i(f_0) = \sum_{\ell \in \Z} \int_{u_{\ell - 1}}^{u_\ell} (J_0')^2 &\geq \sum_{\ell \in \Z, \ell \neq m+1} 2^{-(|\ell| + 2)} \bigl(J_0'(u_{\ell-1})^2 \wedge J_0'(u_\ell)^2\bigr) \\
&=\frac{1}{4}\sum_{\ell=-\infty}^m 2^{-|\ell|} J_0'(u_{\ell})^2 + \frac{1}{4}\sum_{\ell=m+2}^\infty 2^{-|\ell|} J_0'(u_{\ell-1})^2 \\
&= \frac{1}{4}\sum_{\ell=-\infty}^m 2^{-|\ell|} J_0'(u_{\ell})^2 + \frac{1}{4}\sum_{\ell=m+1}^\infty 2^{-|\ell+1|} J_0'(u_\ell)^2 
\geq \frac{1}{8}\sum_{\ell \in \Z} 2^{-|\ell|}J_0'(u_{\ell})^2,
\end{align*}
which yields~\eqref{eq:fisher-information-sum}.
\end{proof}

Recall from~\eqref{eq:scoreEstimateShapeConstrained} the definition of our multiscale estimator $\scoreEstimateShapeConstrained$. The following lemma is a central plank of our proof strategy. The high-probability bound~\eqref{eq:integrated-error-interval} facilitates the application of~\eqref{eq:Holder-basic} in Proposition~\ref{prop:Holder-basic} 
by simplifying the pointwise inequality in Proposition~\ref{prop:scorePointwiseErr}\textit{(b)}. Lemma~\ref{lem:integrated-error-interval}\textit{(b)} serves a dual purpose: first, it is used to verify in Lemma~\ref{lem:integrated-error-interval}\textit{(c)} and later~\eqref{eq:Vj-min-bandwidth} that the restrictions on $h$ in~\eqref{eq:integrated-error-interval} are compatible with those in the definition of $\Gamma$ in Proposition~\ref{prop:Holder-basic}; second, we will bound $V_j$ uniformly over the subclasses of log-concave densities $f_0$ in the following sections, and then sum the resulting contributions from~\eqref{eq:Holder-basic} over $j \in \Z$ to obtain our main results.

\begin{lemma}
\label{lem:integrated-error-interval}
For $j \in \Z$, let $\kappa_j := (a_{j+1} - a_j)/8$ and define the function $g_j \colon [a_j - \kappa_j, a_{j+1} + \kappa_j] \to \R$ by
\begin{equation}
\label{eq:gj}
g_j(x) := \psi_0(x)f_0(a_j)^{1/2}.
\end{equation}
\begin{enumerate}[(a)]
\item Let $n \geq 3$ and $\delta \in (0,1)$. Then there exists a universal constant $C_0 > 0$ such that on the event $\goodEvent$ of probability at least $1 - \delta$, we have for all $j \in \Z$ that
\begin{align}
\label{eq:integrated-error-interval-1}
&\int_{a_j}^{a_{j+1}} \bigl(\scoreEstimateShapeConstrained(x) - \psi_0(x)\bigr)^2 f_0(x) \laplaced x \leq 2\sum_{\omega \in \{-1,1\}} \int_{a_j}^{a_{j+1}} \Delta_{f_0}^{(\omega)}(x)^2 f_0(x) \laplaced x \\
\label{eq:integrated-error-interval}
&\hspace{1.5cm} \lesssim \sum_{\omega \in \{-1,1\}} \int_{a_j}^{a_{j+1}} \inf\biggl\{\bigl(g_j(x - \omega h) - g_j(x)\bigr)^2 \vee \frac{\varepsilon_{n,\delta}}{h^3} : \frac{C_0\varepsilon_{n,\delta}}{f_0(x)} \leq h \leq \kappa_j\biggr\} \laplaced x.
\end{align}
\item Let $V_j := g_j(a_j - \kappa_j) - g_j(a_{j+1} + \kappa_j)$. Then
\[
V_j^2 \leq \bigl(|J_0'(u_{j-1})| + |J_0'(u_{j+2})|\bigr)^2 J_0(u_j) \leq 2^{2|j|+10}J_0(u_{j+1})^2 J_0(u_j) \leq 2^{2|j|+12}J_0(u_j)^3.
\]
\item If $2^{-|j|} \geq 2^{13/2}C_0^{3/2}\varepsilon_{n,\delta}$,
then $C_0\varepsilon_{n,\delta}/f_0(x) \leq \kappa_j$, so the integrand in~\eqref{eq:integrated-error-interval} is finite for all $x \in [a_j,a_{j+1}]$.
\end{enumerate}
\end{lemma}

\begin{proof}
\textit{(a)} The initial bound~\eqref{eq:integrated-error-interval-1} follows directly from Proposition~\ref{prop:scorePointwiseErr}\textit{(b)}. Next, for fixed $j \in \Z$,
$x \in [a_j,a_{j+1}]$ and $\omega \in \{-1,1\}$, we seek $(z,h) \in \mathcal{H}_{f_0}^{(\omega)}(x)$. Consider $z = x - \omega h$ and $h \in (0,\kappa_j/2]$. By~\eqref{eq:partition-property-3} in Lemma~\ref{lem:partition-properties},
\begin{equation}
\label{eq:max-bandwidth}
\kappa_j = \frac{a_{j+1} - a_j}{8} \leq (a_j - a_{j-1}) \wedge (a_{j+2} - a_{j+1}),
\end{equation}
so for any $(z,h)$ as above, we either have $[z-h, z+h] = [x - 2\omega h, x] \subseteq [a_{j-1},a_{j+1}]$ or $[x - 2\omega h, x] \subseteq [a_j,a_{j+2}]$. Since $K_h(\cdot)$ and $\Ind_{[-h,h]}/(2h)$ are densities supported on $[-h,h]$, both
\[
\frac{P_0([z-h, z+h])}{2h} = \frac{1}{2h}\int_{-h}^h f_0(x - \omega h - y) \laplaced y \quad\;\text{and}\quad\; f_h(z) = \int_{-h}^h f_0(x - \omega h - y)K_h(y) \laplaced y
\]
lie in $\bigl[\inf_{x' \in [x - 2\omega h, x]} f_0(x'), \sup_{x' \in [x - 2\omega h, x]} f_0(x')\bigr]$. Thus, by~\eqref{eq:partition-property-2} in Lemma~\ref{lem:partition-properties},
\begin{align*}
\frac{f_0(x)}{f_h(z)} 
\leq \frac{f_0(x)}{\inf_{x' \in [x - 2\omega h, x]} f_0(x')} &\leq \sup_{x',x'' \in [x - 2\omega h, x]} \frac{f_0(x')}{f_0(x'')} \leq \max_{k \in \{j,j+1\}} \sup_{x',x'' \in [a_{k-1},a_{k+1}]} \frac{f_0(x')}{f_0(x'')} \leq 4, \\
\max_{\omega \in \{-1,1\}}\frac{\probDistribution([z,z+\omega h])}{hf_h(z)} &\leq \frac{\probDistribution([z-h,z+h])}{hf_h(z)} \leq 2\sup_{x',x'' \in [x - 2\omega h, x]} \frac{f_0(x')}{f_0(x'')} \leq 8,
\end{align*}
so
\begin{align*}
\COneZH &= 138 + 35\biggl(\frac{\probDistribution([z-h,z+h]) \vee \epsilonByNDeltaMod}{\{hf_h(z)\} \vee \epsilonByNDeltaMod}\biggr)^{1/2} \leq 138 + 35 \times 2^{3/2} =: C_1, \\
\CTwoZH &= 374\max_{\omega \in \{-1,1\}}\biggl(\frac{\probDistribution([z,z+\omega h]) \vee \epsilonByNDeltaMod}{\{hf_h(z)\} \vee \epsilonByNDeltaMod}\biggr)^{1/2} \leq 374 \times 2^{3/2}
\end{align*}
are bounded above by universal constants. Moreover, $\psi_0$ is decreasing and $[z-h, z+h] \subseteq [a_{j-1},a_{j+2}]$, so by Lemma~\ref{lem:biasControl} together with~\eqref{eq:partition-property-4} in Lemma~\ref{lem:partition-properties} and~\eqref{eq:max-bandwidth},
\begin{align*}
|\psi_h(z)| \leq |\psi_0(z - h)| \vee |\psi_0(z + h)| &\leq |\psi_0(a_{j-1})| \vee |\psi_0(a_{j+2})| \\
&\leq \max_{k \in \{j-1,j+1\}}\frac{2}{a_{k+1} - a_k} \leq \frac{2}{\kappa_j} \leq \frac{1}{h}.
\end{align*}
Thus, if $x \in [a_j,a_{j+1}]$ and
\[
\frac{16C_1^2\epsilonByNDeltaMod}{f_0(x)} \leq h \leq \frac{\kappa_j}{2},
\]
then $z = x - \omega h$ satisfies
\begin{align*}
\frac{\epsilonByNDeltaMod}{hf_h(z)} &\leq \frac{4\epsilonByNDeltaMod}{hf_0(x)} \leq \frac{1}{4C_1^2} < 1, \qquad \alpha_{z,h} \leq \frac{C_1^2\epsilonByNDeltaMod}{hf_h(z)} \leq \frac{1}{4}, \\
\beta_{z,h} &\leq \frac{4 \times 374 \times 2^{3/2}}{h}\Bigl(\frac{\epsilonByNDeltaMod}{hf_h(z)}\Bigr)^{1/2} \lesssim 
\Bigl(\frac{\epsilonByNDeltaMod}{h^3 f_0(x)}\Bigr)^{1/2}.
\end{align*}
Hence $(z,h) \in \mathcal{H}_{f_0}^{(\omega)}(x)$, and
\begin{align*}
\Delta_{f_0}^{(\omega)}(x)^2 f_0(x) \lesssim \inf\biggl\{\bigl(\psi_0(x - 2\omega h) - \psi_0(x)\bigr)^2 f_0(x) + \frac{\varepsilon_{n,\delta}}{h^3} : \frac{16C_1^2\varepsilon_{n,\delta}}{f_0(x)} \leq h \leq \frac{\kappa_j}{2}
\biggr\}
\end{align*}
for $\omega \in \{-1,1\}$.  Therefore, by integrating over $x \in [a_j,a_{j+1}]$,~\eqref{eq:integrated-error-interval} holds with $C_0 = 32C_1^2$.

\medskip
\noindent
\textit{(b)} By~\eqref{eq:partition-property-2} in Lemma~\ref{lem:partition-properties}, $f_0(x) \leq 2f_0(a_j)$ for all $x \in [a_j,a_{j+1}]$. Moreover, since $[a_j - \kappa_j, a_{j+1} + \kappa_j] \subseteq [a_{j-1},a_{j+2}]$ by~\eqref{eq:max-bandwidth},
\begin{align*}
V_j := g_j(a_j - \kappa_j) - g_j(a_{j+1} + \kappa_j) &\leq 
\bigl(\psi_0(a_{j-1}) - \psi_0(a_{j+2})\bigr)f_0(a_j)^{1/2} \\
&= \bigl(J_0'(u_{j-1}) - J_0'(u_{j+2})\bigr)J_0(u_j)^{1/2}.
\end{align*}
Together with~\eqref{eq:partition-property-1} and~\eqref{eq:partition-property-2} in Lemma~\ref{lem:partition-properties}, this yields
\begin{align*}
V_j^2 \leq \bigl(|J_0'(u_{j-1})| + |J_0'(u_{j+2})|\bigr)^2 J_0(u_j) &\leq \bigl(2^{|j-1| + 1}J_0(u_{j-1}) + 2^{|j+2| + 1}J_0(u_{j+2})\bigr)^2 J_0(u_j) \\
&\leq (2 \times 2^{|j|+4})^2 J_0(u_{j+1})^2 J_0(u_j) \leq 2^{2|j|+12}J_0(u_j)^3.
\end{align*}
\textit{(c)} If $2^{-|j|} \geq 2^{13/2}C_0^{3/2}\varepsilon_{n,\delta}$, then by~\eqref{eq:partition-property-2} and~\eqref{eq:partition-property-4} in Lemma~\ref{lem:partition-properties}, every $x \in [a_j,a_{j+1}]$ satisfies
\[
\frac{C_0\varepsilon_{n,\delta}}{f_0(x)} \leq \frac{2C_0\varepsilon_{n,\delta}}{f_0(a_j)} \leq 2^{|j|+4}C_0\varepsilon_{n,\delta}(a_{j+1} - a_j) \leq \frac{a_{j+1} - a_j}{2^{5/2}C_0^{1/2}} < \frac{a_{j+1} - a_j}{8} = \kappa_j. \qedhere
\]
\end{proof}

\subsubsection{Proofs for Section~\ref{subsec:tail-growth}}

To gain further insight into the dependence on $L,r$ in the conclusions~\eqref{eq:tail-growth-full} and~\eqref{eq:tail-growth-simple} of Theorem~\ref{thm:tail-growth}, and motivate the choice of the index $j_\star$ in two crucial bounds~\eqref{eq:gamma-bulk-1} and~\eqref{eq:gamma-bulk-2} in the following proof, the following elementary considerations provide some helpful parallels. Given a non-negative sequence $(b_j)$ such that $b_j \leq L$ for all $j \in \N$ (corresponding to $|\psi_0|$ being bounded uniformly by $L$ when $\gamma = 1$), we have
\[
\sum_{j=1}^\infty 2^{-2j/3}b_j^2 \leq L^2\sum_{j=1}^\infty 2^{-2j/3} = \frac{L^2}{2^{2/3} - 1}.
\]
Suppose that we impose the additional restriction $\sum_{j=1}^\infty 2^{-j}b_j^2 \leq r$, which is analogous to the condition $i(f_0) \leq r$ via~\eqref{eq:fisher-information-sum}. First, we can replace $r$ with $r' := r \wedge L^2$ since $\sum_{j=1}^\infty 2^{-j}b_j^2 \leq L^2\sum_{j=1}^\infty 2^{-j} = L^2$ a priori. Then to bound the left-hand side of the display above, we can define $j^* := \floor{\log_2(L^2/r')} + 1$, so that
\begin{equation}
\label{eq:fisher-information-sum-comparison}
\sum_{j=1}^\infty 2^{-2j/3}b_j^2 \leq 2^{j^*/3}\sum_{j=1}^{j^*} 2^{-j}b_j^2 + L^2\sum_{j=j^*+1}^\infty 2^{-2j/3} \leq 2^{j^*/3}r' + \frac{2^{-2j^*/3}L^2}{2^{2/3} - 1} \lesssim r'\frac{L^{2/3}}{r'^{1/3}},
\end{equation}
where the choice of $j^*$ balances the two terms in the penultimate inequality. 
To see that this bound is tight up to a multiplicative universal constant, consider the sequence $(b_j)$ with $b_j = 0$ for $j \in [j^*]$ and $b_j = L$ for $j \geq j^* + 1$. The right-hand side of~\eqref{eq:fisher-information-sum-comparison} is akin to the multiplicative factor in front of $\varepsilon_{n,\delta}^{1/3}$ on the right-hand side of~\eqref{eq:tail-growth-full} when $\gamma = 1$.

\begin{proof}[Proof of Theorem~\ref{thm:tail-growth}]
Fix $f_0 \in \mathcal{J}_{\gamma,L,r}$ and let $j_{n,\delta} := \big\lfloor \log_2(2^{-13/2}C_0^{-3/2}/\varepsilon_{n,\delta})\big\rfloor_+$, where $C_0$ is the universal constant in Lemma~\ref{lem:integrated-error-interval}\textit{(a)}. If $j \in \{-j_{n,\delta},\dotsc,j_{n,\delta} - 1\}$, 
then necessarily $j_{n,\delta} \in \N$ and $2^{-|j|} \geq 2^{13/2}C_0^{3/2}\varepsilon_{n,\delta}$. By Lemma~\ref{lem:integrated-error-interval}\textit{(c)}, $C_0\varepsilon_{n,\delta}/f_0(x) \leq \kappa_j$, so the infimum in~\eqref{eq:integrated-error-interval} is taken over a non-empty set of bandwidths $h$ for every $x \in [a_j,a_{j+1}]$. Moreover, by~\eqref{eq:partition-property-2} in Lemma~\ref{lem:partition-properties} and Lemma~\ref{lem:integrated-error-interval}\textit{(b)},
\begin{equation}
\label{eq:Vj-min-bandwidth}
\frac{C_0\varepsilon_{n,\delta}}{f_0(x)} \leq \frac{2C_0\varepsilon_{n,\delta}}{\bigl(f_0(a_{j+1})^2 f_0(a_j)\bigr)^{1/3}} \leq \Bigl(\frac{\varepsilon_{n,\delta}}{2^{2|j|+10}J_0(u_{j+1})^2 J_0(u_j)}\Bigr)^{1/3} \leq \Bigl(\frac{\varepsilon_{n,\delta}}{V_j^2}\Bigr)^{1/3}.
\end{equation}
We may therefore apply Proposition~\ref{prop:Holder-basic} to the function $g \equiv g_j$ in~\eqref{eq:gj}, with $\beta = 1$, $L_* = V \equiv V_j$, $\varepsilon \equiv \varepsilon_{n,\delta}$, $h_{\max} \equiv \kappa_j$, and deduce from Lemma~\ref{lem:integrated-error-interval}\textit{(a,b)} and~\eqref{eq:partition-property-4} in Lemma~\ref{lem:partition-properties} that
\begin{align}
&\sum_{\omega \in \{-1,1\}} \int_{a_j}^{a_{j+1}}\Delta_{f_0}^{(\omega)}(x)^2 f_0(x) \laplaced x \notag \\
&\hspace{1cm}\lesssim \sum_{\omega \in \{-1,1\}} \int_{a_j}^{a_{j+1}} \inf\Bigl\{\bigl(g_j(x) - g_j(x - \omega h)\bigr)^2 \vee \frac{\varepsilon_{n,\delta}}{h^3} : \Bigl(\frac{\varepsilon_{n,\delta}}{V_j^2}\Bigr)^{1/3} \wedge \kappa_j \leq h \leq \kappa_j\Bigr\} \laplaced x \notag \\
&\hspace{1cm}\lesssim V_j^{4/3}\varepsilon_{n,\delta}^{1/3} + \frac{(a_{j+1} - a_j)\varepsilon_{n,\delta}}{\kappa_j^3}  \notag \\
\label{eq:gamma-bulk-interval}
&\hspace{1cm}\lesssim \bigl(|J_0'(u_{j-1})| \vee |J_0'(u_{j+2})|\bigr)^{4/3}J_0(u_j)^{2/3}\varepsilon_{n,\delta}^{1/3} + 2^{2(|j|+3)}J_0(u_j)^2\varepsilon_{n,\delta}.
\end{align}
Next, since $J_0(0) = J_0(1) = 0$ by Lemma~\ref{lem:density-quantile-basics}\textit{(a)},
\[
|J_0(u)| \leq \int_0^u |J_0'| \wedge \int_u^1 |J_0'| \leq L\int_0^{u \wedge (1 - u)} v^{-(1 - \gamma)/2} \laplaced v = 2L\,\frac{u^{(1 + \gamma)/2} \wedge (1 - u)^{(1 + \gamma)/2}}{1 + \gamma}
\]
for $u \in [0,1]$, so
\begin{equation}
\label{eq:J0-gamma-bound}
f_0(a_j) = J_0(u_j) \leq 2L\,\frac{\{u_j \wedge (1 - u_j)\}^{(1 + \gamma)/2}}{1 + \gamma} = \frac{2^{1 - (|j|+1)(1 + \gamma)/2}L}{1 + \gamma}
\end{equation}
for all $j \in \Z$. We will now bound~\eqref{eq:gamma-bulk-interval} in two different ways depending on the value of $j$. First, recalling from~\eqref{eq:envelope-integral} that $i(f_0) \leq \min\{r,2^{1 - \gamma}L^2/\gamma\} =: \tilde{r}$, let $j_\star := \min\{j \in \N : 2^j \geq (L^2/\tilde{r})^{1/\gamma}\}$. If $j \in \{0,\dotsc,j_\star\}$, then by~\eqref{eq:density-quantile-alt},~\eqref{eq:J0-gamma-bound}, the Cauchy--Schwarz inequality and~\eqref{eq:fisher-information-sum} in Lemma~\ref{lem:J0-fisher-information-sums},
\begin{align}
2^{j/2}J_0(u_j) &\lesssim 2^{j/2}J_0(u_{j_\star}) + \sum_{\ell=j}^{j_\star} 2^{(j - 2\ell)/2}|J_0'(u_\ell)| \lesssim_\gamma 2^{j/2}2^{-\frac{j_\star(1 + \gamma)}{2}}L + \sum_{\ell=j}^{j_\star} 2^{-(\ell - j)/2} \cdot 2^{-\ell/2}|J_0'(u_\ell)| \notag \\
&\lesssim 2^{-j_\star\gamma/2}L + \biggl(\sum_{\ell=j}^{j_\star} 2^{-(\ell - j)}\biggr)^{1/2}\biggl(\sum_{\ell=j}^{j_\star} 2^{-\ell}J_0'(u_\ell)^2\biggr)^{1/2} \lesssim \frac{\tilde{r}^{1/2}}{L}L + i(f_0)^{1/2} \lesssim \tilde{r}^{1/2}.
\end{align}
Moreover, by H\"older's inequality and~\eqref{eq:fisher-information-sum},
\begin{align}
\sum_{j=0}^{j_\star} 2^{-j/3}\bigl(|J_0'(u_{j-1})| \vee |J_0'(u_{j+2})|\bigr)^{4/3} &= 2^{j_\star/3} \sum_{j=0}^{j_\star} 2^{-(j_\star - j)/3} \cdot 2^{-2j/3}\bigl(|J_0'(u_{j-1})| \vee |J_0'(u_{j+2})|\bigr)^{4/3} \notag \\
&\lesssim \Bigl(\frac{L^{2/3}}{\tilde{r}^{1/3}}\Bigr)^{1/\gamma} \biggl(\sum_{j=0}^{j_\star} 2^{-(j_\star - j)}\biggr)^{1/3} \biggl(\sum_{j=-1}^{j_\star + 2} 2^{-j}J_0'(u_j)^2\biggr)^{2/3} \notag \\
&\lesssim \Bigl(\frac{L^{2/3}}{\tilde{r}^{1/3}}\Bigr)^{1/\gamma}\tilde{r}^{2/3}.
\end{align}
Similarly, for $j \in \{-j_\star,\dotsc,0\}$, we have
\begin{equation}
\label{eq:gamma-bulk-1b}
2^{|j|/2}J_0(u_j) \lesssim \tilde{r}^{1/2}, \qquad \sum_{j=-j_\star}^0 2^{-|j|/3}\bigl(|J_0'(u_{j-1})| \vee |J_0'(u_{j+2})|\bigr)^{4/3} \lesssim \Bigl(\frac{L^{2/3}}{\tilde{r}^{1/3}}\Bigr)^{1/\gamma}\tilde{r}^{2/3}.
\end{equation}
Thus, letting $j' := j_\star \wedge j_{n,\delta}$, we deduce from~\eqref{eq:gamma-bulk-interval}--\eqref{eq:gamma-bulk-1b} that
\begin{align}
&\sum_{\omega \in \{-1,1\}} \int_{a_{-j'}}^{a_{j'}} \Delta_{f_0}^{(\omega)}(x)^2 f_0(x) \laplaced x \notag \\
&\hspace{1cm}\lesssim_\gamma \varepsilon_{n,\delta}^{1/3}\sum_{j=-j'}^{j'-1} \bigl(|J_0'(u_{j-1})| \vee |J_0'(u_{j+2})|\bigr)^{4/3}J_0(u_j)^{2/3} + L^2\varepsilon_{n,\delta}\sum_{j=-j'}^{j'-1} 2^{|j|(1 - \gamma)} \notag \\
&\hspace{1cm}\lesssim \tilde{r}^{1/3}\varepsilon_{n,\delta}^{1/3} \sum_{j=-j_\star}^{j_\star} 2^{-|j|/3}\bigl(|J_0'(u_{j-1})| \vee |J_0'(u_{j+2})|\bigr)^{4/3} + 2^{j_\star(1 - \gamma)}L^2\varepsilon_{n,\delta} \notag \\
\label{eq:gamma-bulk-1}
&\hspace{1cm}\lesssim \tilde{r}\Bigl(\frac{L^{2/3}}{\tilde{r}^{1/3}}\Bigr)^{1/\gamma}\varepsilon_{n,\delta}^{1/3} + \tilde{r}\Bigl(\frac{L^2}{\tilde{r}}\Bigr)^{1/\gamma}\varepsilon_{n,\delta}.
\end{align}
For our second bound, it follows from~\eqref{eq:gamma-bulk-interval}, Lemma~\ref{lem:integrated-error-interval}\textit{(b)} and~\eqref{eq:J0-gamma-bound} that if $j \in \{-j_{n,\delta},\dotsc,j_{n,\delta} - 1\}$, then
\begin{align*}
\sum_{\omega \in \{-1,1\}} \int_{a_j}^{a_{j+1}} \Delta_{f_0}^{(\omega)}(x)^2 f_0(x) \laplaced x &\lesssim_\gamma 2^{4|j|/3}J_0(u_j)^2\varepsilon_{n,\delta}^{1/3} + 2^{2|j|}J_0(u_j)^2\varepsilon_{n,\delta} \\
&\lesssim_\gamma L^2\bigl(2^{|j|(\frac{1}{3} - \gamma)}\varepsilon_{n,\delta}^{1/3} + 2^{|j|(1 - \gamma)}\varepsilon_{n,\delta}\bigr).
\end{align*}
Therefore, if $j_\star < j_{n,\delta}$, in which case $j_{n,\delta} \geq 1$ and $\varepsilon_{n,\delta} \asymp 2^{-j_{n,\delta}} < 2^{-j_\star} \asymp_\gamma (\tilde{r}/L^2)^{1/\gamma}$, then
\begin{align}
\label{eq:gamma-bulk-2}
&\sum_{\omega \in \{-1,1\}} \int_{a_{j_\star}}^{a_{j_{n,\delta}}} \Delta_{f_0}^{(\omega)}(x)^2 f_0(x) \laplaced x \lesssim_\gamma L^2\sum_{j=j_\star}^{j_{n,\delta}-1} \bigl(2^{j(\frac{1}{3} - \gamma)}\varepsilon_{n,\delta}^{1/3} + 2^{j(1 - \gamma)}\varepsilon_{n,\delta}\bigr) \\
&\hspace{1cm} \lesssim_\gamma M :=
\begin{cases}
L^2\bigl(2^{j_{n,\delta}(\frac{1}{3}-\gamma)}\varepsilon_{n,\delta}^{1/3} + \varepsilon_{n,\delta}^\gamma\bigr) \asymp L^2\varepsilon_{n,\delta}^\gamma \;\; &\text{if }\gamma \in (0,1/3) \\[8pt]
L^2\bigl(j_{n,\delta}\varepsilon_{n,\delta}^{1/3} + \varepsilon_{n,\delta}^{1/3}\bigr) \asymp L^2\varepsilon_{n,\delta}^{1/3}\log_+\Bigl(\dfrac{1}{\varepsilon_{n,\delta}}\Bigr) \;\; &\text{if }\gamma = 1/3 \\[10pt]
L^2\bigl(2^{-j_\star(\gamma - \frac{1}{3})}\varepsilon_{n,\delta}^{1/3} + \varepsilon_{n,\delta}^\gamma\bigr) \asymp_\gamma 
\tilde{r}\Bigl(\dfrac{L^{2/3}}{\tilde{r}^{1/3}}\Bigr)^{1/\gamma}\varepsilon_{n,\delta}^{1/3} \;\; &\text{if }\gamma \in (1/3,1].
\end{cases}
\notag
\end{align}
Similarly, $\sum_{\omega \in \{-1,1\}} \int_{a_{-j_{n,\delta}}}^{a_{-j_\star}} \Delta_{f_0}^{(\omega)}(x)^2 f_0(x) \laplaced x \lesssim_\gamma M$. 
Further, by~\eqref{eq:tail-growth} and Lemma~\ref{lem:density-quantile-basics}\textit{(b)},
\begin{align}
\int_{-\infty}^{a_{-j_{n,\delta}}} \psi_0^2 f_0 &= \int_0^{u_{-j_{n,\delta}}} (J_0')^2 \leq \int_0^{2^{-(j_{n,\delta}+1)}} \frac{L^2}{u^{1 - \gamma}} \laplaced u \leq \frac{L^2}{2^{\gamma(j_{n,\delta}+1)} \gamma} \lesssim_\gamma L^2\varepsilon_{n,\delta}^\gamma, \notag \\
\label{eq:gamma-tail}
\int_{a_{j_{n,\delta}}}^\infty \psi_0^2 f_0 &= \int_{u_{j_{n,\delta}}}^1 (J_0')^2 \leq \int_{1 - 2^{-(j_{n,\delta}+1)}}^1 \frac{L^2}{(1 - u)^{1 - \gamma}} \laplaced u \lesssim_\gamma L^2\varepsilon_{n,\delta}^\gamma.
\end{align}
Moreover,~\eqref{eq:envelope-integral} and Proposition~\ref{prop:scorePointwiseErr}\textit{(c)} ensure that on the event $\goodEvent$ of probability at least $1 - \delta$ defined in~\eqref{eq:goodEvent}, the multiscale estimator $\hat{\psi}_{n,\delta}$ satisfies
\[
\sup_{f_0 \in \mathcal{J}_{\gamma,L,r}} \int_{-\infty}^\infty (\hat{\psi}_{n,\delta} - \psi_0)^2 f_0 \leq \sup_{f_0 \in \mathcal{J}_{\gamma,L,r}} \int_{-\infty}^\infty \psi_0^2 f_0 = \sup_{f_0 \in \mathcal{J}_{\gamma,L,r}} i(f_0) \leq r \wedge \frac{2^{1 - \gamma}L^2}{\gamma} = \tilde{r},
\]
so $\mathcal{M}_n(\delta,\mathcal{J}_{\gamma,L,r}) \lesssim_\gamma \bar{r}$. In combination with~\eqref{eq:gamma-bulk-1}--\eqref{eq:gamma-tail}, we conclude from Proposition~\ref{prop:scorePointwiseErr}\textit{(b,c)} that on~$\goodEvent$, the estimator~$\scoreEstimateShapeConstrained$ satisfies
\begin{align}
\int_{-\infty}^\infty \bigl(\scoreEstimateShapeConstrained(x) &- \psi_0(x)\bigr)^2 f_0(x) \laplaced x \notag \\
&\lesssim \biggl\{\sum_{\omega \in \{-1,1\}} \int_{a_{-j_{n,\delta}}}^{a_{j_{n,\delta}}} \Delta_{f_0}^{(\omega)}(x)^2 f_0(x) \laplaced x + \int_{-\infty}^{a_{-j_{n,\delta}}} \psi_0^2 f_0 + \int_{a_{j_{n,\delta}}}^\infty \psi_0^2 f_0\biggr\} \wedge \tilde{r} \notag \\ 
&\lesssim_\gamma \biggl\{\tilde{r}\Bigl(\frac{L^{2/3}}{\tilde{r}^{1/3}}\Bigr)^{1/\gamma}\varepsilon_{n,\delta}^{1/3} + \tilde{r}\Bigl(\frac{L^2}{\tilde{r}}\Bigr)^{1/\gamma}\varepsilon_{n,\delta} + M + \tilde{r}\,\frac{L^2}{\tilde{r}}\varepsilon_{n,\delta}^\gamma\biggr\} \wedge \tilde{r} \notag \\
\label{eq:tail-growth-multiscale}
&\lesssim \biggl\{\tilde{r}\Bigl(\frac{L^{2/\gamma}}{\tilde{r}^{1/\gamma}}\varepsilon_{n,\delta}\Bigr)^{\gamma \wedge \frac{1}{3}} + M\biggr\} \wedge \tilde{r} \asymp_\gamma M \wedge \bar{r}.
\end{align}
This yields the upper bound~\eqref{eq:tail-growth-full} on the minimax $(1 - \delta)$th quantile $\mathcal{M}_n(\delta,\mathcal{J}_{\gamma,L,r})$ for every $\gamma \in (0,1]$. In particular, when $r = 2^{1 - \gamma}L^2/\gamma$, we have $\mathcal{J}_{\gamma,L,r} = \mathcal{J}_{\gamma,L}$, so~\eqref{eq:tail-growth-full} simplifies to the bound~\eqref{eq:tail-growth-simple} on $\mathcal{M}_n(\delta,\mathcal{J}_{\gamma,L})$.
\end{proof}

\begin{proof}[Proof of Corollary~\ref{cor:tail-growth-expectation}]
We use the notation and definitions in the statement and proof of Theorem~\ref{thm:tail-growth}. For $B \in (0,\infty]$, consider the projected estimator $\scoreEstimateShapeConstrained^{(B)} := (-B) \vee \scoreEstimateShapeConstrained \wedge B$. On the event $\goodEvent$, Proposition~\ref{prop:scorePointwiseErr}\textit{(c)} ensures that
\[
\bigl|\scoreEstimateShapeConstrained^{(B)}(x)\bigr| \leq |\scoreEstimateShapeConstrained(x)| \leq |\psi_0(x)|
\]
for all $x \in \R$; moreover, if $x \in [a_{-{j_{n,\delta}}},a_{{j_{n,\delta}}}]$, then by~\eqref{eq:tail-growth},
\begin{align*}
\bigl|\scoreEstimateShapeConstrained(x)\bigr| \leq \bigl|\psi_0(x)\bigr| &\leq |\psi_0(a_{-j_{n,\delta}})| \vee |\psi_0(a_{j_{n,\delta}})| = |J_0'(u_{-j_{n,\delta}})| \vee |J_0'(u_{j_{n,\delta}})| \\
&\leq 2^{\frac{(j_{n,\delta} + 1)(1 - \gamma)}{2}}L \leq \min\bigl(2^{11/2}C_0^{3/2}\varepsilon_{n,\delta}, 1\bigr)^{-\frac{1 - \gamma}{2}}L =: B_{n,\delta}.
\end{align*}
Thus, if $B \geq B_{n,\delta}$, then $\scoreEstimateShapeConstrained^{(B)}(x) = \scoreEstimateShapeConstrained(x)$ for all $x \in [a_{-{j_{n,\delta}}},a_{{j_{n,\delta}}}]$, so by~\eqref{eq:tail-growth-multiscale} and~\eqref{eq:gamma-tail},
\begin{align}
\label{eq:tail-growth-multiscale-proj}
\int_{-\infty}^{\infty}\bigl(\scoreEstimateShapeConstrained^{(B)} - \psi_0\bigr)^2 f_0 &\leq \biggl\{\int_{a_{-{j_{n,\delta}}}}^{a_{{j_{n,\delta}}}} (\scoreEstimateShapeConstrained - \psi_0)^2 f_0 + \int_{-\infty}^{a_{-{j_{n,\delta}}}}\psi_0^2 f_0 + \int_{a_{{j_{n,\delta}}}}^{\infty}\psi_0^2 f_0\biggr\} \lesssim_\gamma B \wedge \bar{r}
\end{align}
on $\goodEvent$. In other words, the upper bound on $\mathcal{M}_n(\delta,\mathcal{J}_{\gamma,L,r})$ is achieved by $\hat{\psi}_{n,\delta}^{(B)}$ for any $B \geq B_{n,\delta}$.

Next, by considering the trivial estimator $\tilde\psi \equiv 0$, we deduce from~\eqref{eq:envelope-integral} that
\[
\mathcal{M}_n(\mathcal{J}_{\gamma,L,r}) \leq \sup_{f_0 \in \mathcal{J}_{\gamma,L,r}} \int_{-\infty}^\infty \psi_0^2 f_0 \leq r \wedge \frac{2^{1 - \gamma}L^2}{\gamma} = \tilde{r}. 
\]
Next, taking $\delta = 1/n$ and $B := B_{n,\delta}$, we have $\Pr(\goodEvent^c) \leq \delta = 1/n$ by Proposition~\ref{prop:confidenceBandsAreValid},
and
\[
\varepsilon_{n,\delta} = \frac{2\log(n^2/\delta)}{9n} = \frac{2\log n}{3n}, \qquad |\scoreEstimateShapeConstrained^{(B)}| \leq B \asymp L(\varepsilon_{n,\delta} \wedge 1)^{-(1 - \gamma)/2} \asymp L\Bigl(\frac{n}{\log n}\Bigr)^{\frac{1 - \gamma}{2}}. 
\]
Thus, by~\eqref{eq:tail-growth-multiscale-proj} and~\eqref{eq:envelope-integral},
\begin{align*}
\mathcal{M}_n(\mathcal{J}_{\gamma,L,r}) &\leq \sup_{f_0 \in \mathcal{J}_{\gamma,L,r}} \E\biggl(\int_{-\infty}^{\infty} \bigl(\scoreEstimateShapeConstrained^{(B)} - \psi_0\bigr)^2 f_0\biggr) \wedge \tilde{r} \\
&\leq \sup_{f_0 \in \mathcal{J}_{\gamma,L,r}} \biggl\{\E\biggl(\Ind_{\goodEvent^c}\int_{-\infty}^{\infty} (B - \psi_0)^2 f_0\biggr) + \E\biggl(\Ind_{\goodEvent}\int_{-\infty}^{\infty} \bigl(\scoreEstimateShapeConstrained^{(B)} - \psi_0\bigr)^2 f_0\biggr)\biggr\} \wedge \tilde{r} \\
&\lesssim \min\biggl\{\delta\Bigl(B^2 + \sup_{f_0 \in \mathcal{J}_{\gamma,L,r}} \int_{-\infty}^\infty \psi_0^2 f_0\Bigr) + \bar{r}\Bigl(\frac{L^{2/\gamma}}{\bar{r}^{1/\gamma}} \varepsilon_{n,\delta}\Bigr)^{\gamma \wedge \frac{1}{3}} \log_+^{\Ind_{\{\gamma = 1/3\}}}\Bigl(\frac{1}{\varepsilon_{n,\delta}}\Bigr),\,\bar{r}\biggr\} \\
&\lesssim_\gamma \bar{r}\min\biggl\{\frac{L^2}{\bar{r}} \cdot \frac{1}{n}\Bigl(\frac{n}{\log n}\Bigr)^{1 - \gamma} + \frac{1}{n} + \Bigl(\frac{L^{2/\gamma}}{\bar{r}^{1/\gamma}} \cdot \frac{\log n}{n}\Bigr)^{\gamma \wedge \frac{1}{3}} \log^{\Ind_{\{\gamma = 1/3\}}}\Bigl(\frac{n}{\log n}\Bigr),\,1\biggr\} \\
&\lesssim_\gamma \bar{r}\min\Bigl\{\Bigl(\frac{L^{2/\gamma}}{\bar{r}^{1/\gamma}} \cdot \frac{1}{n}\Bigr)^{\gamma \wedge \frac{1}{3}}\log^{c_\gamma}n,\,1\Bigr\},
\end{align*}
as required.
\end{proof}

\subsubsection{Proofs for Section~\ref{subsec:holder}}

The proof strategy for Theorem~\ref{thm:holder-minimax-upper} is analogous to that for Theorem~\ref{thm:tail-growth}, in that we employ the same sequence $(a_j)_{j \in \Z}$ as in Lemma~\ref{lem:partition-properties} but use Lemma~\ref{lem:holder-density-quantile} instead of~\eqref{eq:tail-growth} to bound terms involving $J_0'$ and $J_0$. The H\"older condition restricts the range of indices $j$ for which Proposition~\ref{prop:Holder-basic} is applicable on $[a_j,a_{j+1}]$ via Lemma~\ref{lem:integrated-error-interval}; see~\eqref{eq:holder-density-indices} below. 

\begin{proof}[Proof of Theorem~\ref{thm:holder-minimax-upper}]
Consider an arbitrary $f_0 \in \mathcal{F}_{\beta,L,r}$. We first assemble some useful bounds that will be invoked repeatedly later on. First, since $C_\beta \leq 2/\sqrt{\pi}$ and $\log_+(\lambda x) \leq \lambda\log_+(x)$ for $\lambda \geq 1$ and $x > 0$, 
it follows from the definition of $u_j$ for $j \in \Z$ and~\eqref{eq:holder-density-quantile-deriv} and~\eqref{eq:holder-density-quantile} in Lemma~\ref{lem:holder-density-quantile} that
\begin{equation}
\label{eq:holder-geometric-quantiles}
\begin{split}
|J_0'(u_j)| &\leq 2L^{1/\beta}\log_+^{(\beta - 1)/\beta}(2^{|j|+1}C_\beta) \leq 3L^{1/\beta}(|j| \vee 1)^{(\beta - 1)/\beta}, \\
J_0(u_j) &\leq 2^{-|j|+1}L^{1/\beta}\log_+^{(\beta - 1)/\beta}(2^{|j|+1}C_\beta) \leq 2^{-|j|}L^{1/\beta}(|j| \vee 1)^{(\beta - 1)/\beta}.
\end{split}
\end{equation}
Fix $\omega \in \{-1,1\}$ and let $j_{n,\delta} := \big\lfloor\log_2(2^{-13/2}C_0^{-3/2}/\varepsilon_{n,\delta})\big\rfloor_+$. If $j \in \{-j_{n,\delta},\dotsc,j_{n,\delta} - 1\}$, then for every $x \in [a_j, a_{j+1}]$, we have $C_0\varepsilon_{n,\delta}/f_0(x) 
\leq \kappa_j$ by Lemma~\ref{lem:integrated-error-interval}\textit{(c)}, so the infimum in~\eqref{eq:integrated-error-interval} is taken over a non-empty set of bandwidths $h$. 
Since $f_0 \in \mathcal{F}_{\beta,L}$, the function $g_j \colon x \mapsto \psi_0(x)f_0(a_j)^{1/2}$ in Lemma~\ref{lem:integrated-error-interval} satisfies $|g_j(x) - g_j(y)| \leq Lf_0(a_j)^{1/2}|x - y|^{\beta - 1}$ for all $x,y \in \R$, so we seek to apply Proposition~\ref{prop:Holder-basic} to $g \equiv g_j$ on $[a_j - \kappa_j, a_{j+1} + \kappa_j]$ with $L_* \equiv Lf_0(a_j)^{1/2}$ and $V \equiv V_j = g_j(a_j - \kappa_j) - g_j(a_{j+1} + \kappa_j)$.
To this end, observe that whenever
\begin{equation}
\label{eq:holder-density-indices}
J_0(u_j) \geq L^{1/\beta}(2C_0)^{(2\beta + 1)/(2\beta)}\varepsilon_{n,\delta},
\end{equation}
it follows from \eqref{eq:partition-property-2} in Lemma~\ref{lem:partition-properties} that every $x \in [a_j,a_{j+1}]$ satisfies
\begin{equation}
\label{eq:C0epsilonbound}
\frac{C_0\varepsilon_{n,\delta}}{f_0(x)} \leq \frac{2C_0\varepsilon_{n,\delta}}{f_0(a_j)} \leq \Bigl(\frac{\varepsilon_{n,\delta}}{L^2 f_0(a_j)}\Bigr)^{1/(2\beta + 1)}. 
\end{equation}
If there exists $j \in \{-j_{n,\delta},\dotsc,j_{n,\delta} - 1\}$ for which~\eqref{eq:holder-density-indices} holds, then let $j'$ and $j''$ be the minimum and maximum such indices respectively. Otherwise, let $j' = j'' := 1$. 
Then by~\eqref{eq:holder-geometric-quantiles}, $j_* := \max(|j'|,|j''|)$ satisfies $(|j_*| \vee 1)^{(\beta - 1)/\beta}2^{-|j_*|} \gtrsim \varepsilon_{n,\delta}$ 
and hence
\begin{equation}
\label{eq:bulk-indices}
2^{|j_*|} \lesssim \frac{1}{\varepsilon_{n,\delta}}\log_+\Bigl(\frac{1}{\varepsilon_{n,\delta}}\Bigr).
\end{equation}
If $j' \leq j < j''$, 
then by Lemma~\ref{lem:integrated-error-interval}\textit{(a)},~\eqref{eq:C0epsilonbound}, Proposition~\ref{prop:Holder-basic}, Lemma~\ref{lem:integrated-error-interval}\textit{(b)} and~\eqref{eq:partition-property-4} in Lemma~\ref{lem:partition-properties},
\begin{align}
&\hspace{-0.4cm}\sum_{\omega \in \{-1,1\}} \int_{a_j}^{a_{j+1}} \Delta_{f_0}^{(\omega)}(x)^2 f_0(x) \laplaced x \notag \\
&\lesssim \sum_{\omega \in \{-1,1\}} \int_{a_j}^{a_{j+1}} \inf\Bigl\{\bigl(g_j(x) - g_j(x - \omega h)\bigr)^2 \vee \frac{\varepsilon_{n,\delta}}{h^3} : \Bigl(\frac{\varepsilon_{n,\delta}}{L^2 f_0(a_j)}\Bigr)^{1/(2\beta + 1)} \wedge \kappa_j \leq h \leq \kappa_j\Bigr\} \laplaced x \notag \\
&\lesssim V_j L^{1/(2\beta + 1)}f_0(a_j)^{1/(4\beta+2)}\varepsilon_{n,\delta}^{\beta/(2\beta + 1)} + \frac{(a_{j+1} - a_j)\varepsilon_{n,\delta}}{\kappa_j^3} \notag \\
\label{eq:holder-bulk-interval}
&\lesssim L^{1/(2\beta + 1)}\bigl(|J_0'(u_{j-1})| \vee |J_0'(u_{j+2})|\bigr)\bigl(J_0(u_j)^{\beta + 1}\varepsilon_{n,\delta}^\beta\bigr)^{1/(2\beta + 1)} + 2^{2(|j|+3)}J_0(u_j)^2\varepsilon_{n,\delta}.
\end{align}
Similarly to Theorem~\ref{thm:tail-growth}, we will bound~\eqref{eq:holder-bulk-interval} in two different ways depending on the value of $j$. First, recalling from~\eqref{eq:holder-fisher-information} that $f_0 \in \mathcal{F}_{\beta,L,r}$ satisfies $i(f_0) \leq r \wedge 8L^{2/\beta} =: s$, let $j_L := \min\bigl\{j \in \N : 2^j \geq (L^{2/\beta}/s)\log_+^{2(\beta-1)/\beta}(L^{2/\beta}/s)\bigr\}$. 
If $j \in \{0,\dotsc,j_L\}$, then by~\eqref{eq:density-quantile-alt},~\eqref{eq:holder-geometric-quantiles}, the Cauchy--Schwarz inequality and~\eqref{eq:fisher-information-sum},
\begin{align}
2^{j/2}J_0(u_j) &\lesssim 2^{j/2}J_0(u_{j_L}) + 2^{j/2}\sum_{\ell=j}^{j_L} 2^{-\ell}|J_0'(u_\ell)| \notag \\
&\lesssim L^{1/\beta}2^{-j_L/2}j_L^{(\beta - 1)/\beta} + \sum_{\ell=j}^{j_L} 2^{-(\ell - j)/2} \cdot 2^{-\ell/2}|J_0'(u_\ell)| \notag \\
&\lesssim s^{1/2} + \biggl(\sum_{\ell=j}^{j_L} 2^{-(\ell - j)}\biggr)^{1/2}\biggl(\sum_{\ell=j}^{j_L} 2^{-\ell}J_0'(u_\ell)^2\biggr)^{1/2} \lesssim s^{1/2} + i(f_0)^{1/2} \lesssim s^{1/2}.
\end{align}
By another application of the Cauchy--Schwarz inequality and~\eqref{eq:fisher-information-sum},
\begin{align}
\sum_{j=0}^{j_L} 2^{-\frac{j(\beta + 1)}{2(2\beta + 1)}}\bigl(|J_0'(u_{j-1})| \vee |J_0'(u_{j+2})|\bigr) &= 2^{\frac{j_L\beta}{2(2\beta + 1)}} \sum_{j=0}^{j_L} 2^{-\frac{(j_L - j)\beta}{2(2\beta + 1)}} \cdot 2^{-j/2}\bigl(|J_0'(u_{j-1})| \vee |J_0'(u_{j+2})|\bigr) \notag \\
&\lesssim 2^{\frac{j_L\beta}{2(2\beta + 1)}} \biggl(\sum_{j=0}^{j_L} 2^{-\frac{(j_L - j)\beta}{2\beta + 1}}\biggr)^{1/2}\biggl(\sum_{j=-1}^{j_L + 2} 2^{-j}J_0'(u_j)^2\biggr)^{1/2} \notag \\
\label{eq:holder-bulk-interval-2}
&\lesssim \Bigl(\frac{L}{s^{\beta/2}}\Bigr)^{\frac{1}{2\beta + 1}}\log_+^{\frac{\beta - 1}{2\beta + 1}}\Bigl(\frac{L^{2/\beta}}{s}\Bigr)s^{1/2}.
\end{align}
The same bounds hold for $j \in \{-j_L,\dotsc,-1\}$ by analogous reasoning, so we deduce from~\eqref{eq:holder-bulk-interval}--\eqref{eq:holder-bulk-interval-2} that 
\begin{align}
\sum_{j=j' \vee -j_L}^{(j'' \wedge j_L) - 1} &\sum_{\omega \in \{-1,1\}} \int_{a_j}^{a_{j+1}} \Delta_{f_0}^{(\omega)}(x)^2 f_0(x) \laplaced x \notag \\
&\lesssim (L\varepsilon_{n,\delta}^\beta)^{\frac{1}{2\beta + 1}}\sum_{j=-j_L}^{j_L} J_0(u_j)^{\frac{\beta + 1}{2\beta + 1}}\bigl(|J_0'(u_{j-1})| \vee |J_0'(u_{j+2})|\bigr) + s\sum_{j=-j_L}^{j_L} 2^{|j|}\varepsilon_{n,\delta} \notag \\
&\lesssim (L\varepsilon_{n,\delta}^\beta s^{\frac{\beta + 1}{2}})^{\frac{1}{2\beta + 1}} \sum_{j=-j_L}^{j_L} 2^{-\frac{|j|(\beta + 1)}{2(2\beta + 1)}}\bigl(|J_0'(u_{j-1})| \vee |J_0'(u_{j+2})|\bigr) + L^{2/\beta}\log_+^{\frac{2(\beta-1)}{\beta}}\Bigl(\frac{L^{2/\beta}}{s}\Bigr) \varepsilon_{n,\delta}  \notag \\
\label{eq:holder-bulk-sum-1}
&\lesssim s\Bigl(\frac{L^{2/\beta}}{s}\varepsilon_{n,\delta}\Bigr)^{\beta/(2\beta + 1)}\log_+^{\frac{\beta - 1}{2\beta + 1}}\Bigl(\frac{L^{2/\beta}}{s}\Bigr) + s\Bigl(\frac{L^{2/\beta}}{s}\varepsilon_{n,\delta}\Bigr)\log_+^{\frac{2(\beta-1)}{\beta}}\Bigl(\frac{L^{2/\beta}}{s}\Bigr).
\end{align}

Alternatively, for $j' \leq j < j''$, we can apply~\eqref{eq:holder-geometric-quantiles} directly to~\eqref{eq:holder-bulk-interval} to obtain
\begin{align}
\sum_{\omega \in \{-1,1\}} &\int_{a_j}^{a_{j+1}} \Delta_{f_0}^{(\omega)}(x)^2 f_0(x) \laplaced x \notag \\
&\lesssim L^{\frac{1}{2\beta + 1}}L^{1/\beta}(|j| \vee 1)^{1 - \frac{1}{\beta}}\bigl(2^{-|j|}(|j| \vee 1)^{1 - \frac{1}{\beta}}L^{1/\beta}\bigr)^{\frac{\beta + 1}{2\beta + 1}}\varepsilon_{n,\delta}^{\beta/(2\beta + 1)} + L^{2/\beta}(|j| \vee 1)^{2(1 - \frac{1}{\beta})}\varepsilon_{n,\delta} \notag \\
\label{eq:holder-bulk-interval-3}
&= L^{2/\beta}\bigl\{2^{-|j|\frac{\beta + 1}{2\beta + 1}}(|j| \vee 1)^{\frac{(3\beta+2)(\beta-1)}{\beta(2\beta+1)}}\,\varepsilon_{n,\delta}^{\beta/(2\beta + 1)} + (|j| \vee 1)^{2(\beta - 1)/\beta}\varepsilon_{n,\delta}\bigr\}.
\end{align}
If $(j_1,j_2) = (j',j' \vee -j_L)$ or $(j_1,j_2) = (j'' \wedge j_L,j'')$ and $j_1 < j_2$, then $j_* \geq 1$, and by~\eqref{eq:bulk-indices},
\begin{align*}
\sum_{j=j_1}^{j_2 - 1} L^{2/\beta}&\bigl\{2^{-|j|\frac{\beta + 1}{2\beta + 1}}(|j| \vee 1)^{\frac{(3\beta+2)(\beta-1)}{\beta(2\beta+1)}}\,\varepsilon_{n,\delta}^{\beta/(2\beta + 1)} + (|j| \vee 1)^{2(\beta - 1)\beta}\varepsilon_{n,\delta}\bigr\} \\
&\lesssim_\beta L^{2/\beta}\varepsilon_{n,\delta}^{\beta/(2\beta + 1)}2^{-j_L\frac{\beta + 1}{2\beta + 1}}j_L^{\frac{(3\beta+2)(\beta-1)}{\beta(2\beta+1)}} + L^{2/\beta}j_*^{(3\beta - 2)/\beta}\varepsilon_{n,\delta} \\
&\lesssim s\Bigl(\frac{L^{2/\beta}}{s}\varepsilon_{n,\delta}\Bigr)^{\beta/(2\beta + 1)}\log_+^{\frac{\beta - 1}{2\beta + 1}}\Bigl(\frac{L^{2/\beta}}{s}\Bigr) + s\Bigl(\frac{L^{2/\beta}}{s}\varepsilon_{n,\delta}\Bigr)\log_+^{(3\beta - 2)/\beta}\Bigl(\frac{1}{\varepsilon_{n,\delta}}\Bigr).
\end{align*}
Together with~\eqref{eq:holder-bulk-sum-1} and~\eqref{eq:holder-bulk-interval-3}, this yields
\begin{align}
\sum_{\omega \in \{-1,1\}} \int_{a_{j'}}^{a_{j''}} \Delta_{f_0}^{(\omega)}(x)^2 f_0(x) \laplaced x &\lesssim_\beta s\Bigl(\frac{L^{2/\beta}}{s}\varepsilon_{n,\delta}\Bigr)^{\beta/(2\beta + 1)}\log_+^{\frac{\beta - 1}{2\beta + 1}}\Bigl(\frac{L^{2/\beta}}{s}\Bigr) \notag \\
\label{eq:holder-bulk-new}
&\quad+ s\Bigl(\frac{L^{2/\beta}}{s}\varepsilon_{n,\delta}\Bigr)\Bigl\{\log_+^{\frac{2(\beta-1)}{\beta}}\Bigl(\frac{L^{2/\beta}}{s}\Bigr) + \log_+^{(3\beta - 2)/\beta}\Bigl(\frac{1}{\varepsilon_{n,\delta}}\Bigr)\Bigr\}.
\end{align}

Next, recall that $j' \geq -j_{n,\delta}$ by definition. If $j' > -j_{n,\delta}$, then in both of the cases $j' < j''$ and $j' = j''$,~\eqref{eq:holder-density-indices} does not hold for any $j \in \{-j_{n,\delta},\dotsc,j'-1\}$, so by~\eqref{eq:partition-property-2} in Lemma~\ref{lem:partition-properties},
\[
\sup_{u \in (u_{-j_{n,\delta}}, u_{j'}]}J_0(u) \leq 2 \max_{j \in \{-j_{n,\delta},\dotsc,j'-1\}}J_0(u_{j}) \lesssim L^{1/\beta}\varepsilon_{n,\delta}.
\]
Therefore, by Lemma~\ref{lem:density-quantile-basics}\textit{(b)},~\eqref{eq:holder-density-quantile-deriv}, a calculation similar to~\eqref{eq:holder-information-integral} in the proof of Lemma~\ref{lem:holder-density-quantile}, and finally~\eqref{eq:bulk-indices}, it follows in both cases that
\begin{align}
\int_{-\infty}^{a_{j'}} \psi_0^2 f_0 = \int_0^{u_{j'}} (J_0')^2 &\leq \int_0^{u_{-j_{n,\delta}}} (J_0')^2 + \sup_{u \in (u_{-j_{n,\delta}},u_{j'}]} |J_0'(u)| \int_{u_{-j_{n,\delta}}}^{u_{j'}} |J_0'| \notag \\
&\lesssim L^{2/\beta} \int_0^{u_{-j_{n,\delta}}} \log_+^{2(\beta - 1)/\beta}\Bigl(\frac{C_\beta}{u}\Bigr) \laplaced u + L^{1/\beta}j_{n,\delta}^{(\beta - 1)/\beta}\sup_{u \in (u_{-j_{n,\delta}}, u_{j'}]} J_0(u) \notag \\ 
\label{eq:holder-tail-bound}
&\lesssim L^{2/\beta}\varepsilon_{n,\delta}\log_+^{2(\beta - 1)/\beta}\Bigl(\frac{1}{\varepsilon_{n,\delta}}\Bigr).
\end{align}
By analogous reasoning, the same bound holds for $\int_{a_{j''}}^\infty \psi_0^2 f_0$. Moreover, by Proposition~\ref{prop:scorePointwiseErr}\textit{(c)} and~\eqref{eq:holder-fisher-information} in Lemma~\ref{lem:holder-density-quantile},
\[
\int_{-\infty}^\infty \bigl(\scoreEstimateShapeConstrained(x) -\psi_0(x)\bigr)^2 f_0(x) \laplaced x \leq i(f_0) \leq r \wedge 8L^{2/\beta} = s. 
\]
Combining this with~\eqref{eq:holder-bulk-new} and~\eqref{eq:holder-tail-bound}, we deduce from Proposition~\ref{prop:scorePointwiseErr}\textit{(b,c)} that on the event $\goodEvent$ of probability at least $1 - \delta$, the estimator $\scoreEstimateShapeConstrained$ satisfies
\begin{align}
&\hspace{-0.1cm}\int_{-\infty}^\infty \bigl(\scoreEstimateShapeConstrained(x) -\psi_0(x)\bigr)^2 f_0(x) \laplaced x \notag \\
&\lesssim \biggl\{\sum_{\omega \in \{-1,1\}} \int_{a_{j'}}^{a_{j''}} \Delta_{f_0}^{(\omega)}(x)^2 f_0(x) \laplaced x + \int_{-\infty}^{a_{j'}} \psi_0^2 f_0 + \int_{a_{j''}}^\infty \psi_0^2 f_0\biggr\} \wedge s \notag \\
&\lesssim_\beta \biggl\{s\Bigl(\frac{L^{2/\beta}}{s}\varepsilon_{n,\delta}\Bigr)^{\frac{\beta}{2\beta + 1}}\log_+^{\frac{\beta - 1}{2\beta + 1}}\Bigl(\frac{L^{2/\beta}}{s}\Bigr) + s\Bigl(\frac{L^{2/\beta}}{s}\varepsilon_{n,\delta}\Bigr)\biggl(\log_+^{\frac{2(\beta-1)}{\beta}\Bigl(\frac{L^{2/\beta}}{s}\Bigr) + \log_+^{3 - \frac{2}{\beta}}\Bigl(\frac{1}{\varepsilon_{n,\delta}}\Bigr)}\biggr)\biggr\} \wedge s \notag \\
\label{eq:multiscale-holder-bound}
&\lesssim_\beta r'\min\biggl\{\Bigl(\frac{L^{2/\beta}}{r'}\varepsilon_{n,\delta}\Bigr)^{\beta/(2\beta + 1)}\log_+^{\frac{\beta - 1}{2\beta + 1}}\Bigl(\frac{L^{2/\beta}}{r'}\Bigr) + \frac{L^{2/\beta}}{r'}\varepsilon_{n,\delta}\log_+^{3 - \frac{2}{\beta}}\Bigl(\frac{1}{\varepsilon_{n,\delta}}\Bigr),\,1\biggr\}.
\end{align}
This yields the required upper bound~\eqref{eq:holder-minimax-upper} on the minimax $(1 - \delta)$th quantile $\mathcal{M}_n(\delta,\mathcal{F}_{\beta,L,r})$. In particular, when $r = 8L^{2/\beta}$, we have $\mathcal{F}_{\beta,L,r} = \mathcal{F}_{\beta,L}$, so~\eqref{eq:multiscale-holder-bound} simplifies to a bound of order $L^{2/\beta}(\varepsilon_{n,\delta} \wedge 1)^{\beta/(2\beta + 1)}$ on $\mathcal{M}_n(\delta,\mathcal{F}_{\beta,L})$.
\end{proof}

\begin{proof}[Proof of Corollary~\ref{cor:holder-minimax-expec}]
We argue similarly to the proof of Corollary~\ref{cor:tail-growth-expectation}. Proposition~\ref{prop:scorePointwiseErr}\textit{(c)} ensures that on $\goodEvent$, the projected estimator $\scoreEstimateShapeConstrained^{(B)} := (-B) \vee \scoreEstimateShapeConstrained \wedge B$ with $B \in (0,\infty]$ satisfies
\[
\bigl|\scoreEstimateShapeConstrained^{(B)}(x)\bigr| \leq |\scoreEstimateShapeConstrained(x)| \leq |\psi_0(x)|
\]
for all $x \in \R$. Moreover, if $x \in [a_{-{j_{n,\delta}}},a_{{j_{n,\delta}}}]$, then by~\eqref{eq:holder-geometric-quantiles},
\begin{align*}
\bigl|\scoreEstimateShapeConstrained(x)\bigr| \leq \bigl|\psi_0(x)\bigr| \leq |J_0'(u_{-j_{n,\delta}})| \vee |J_0'(u_{j_{n,\delta}})| \leq 3L^{1/\beta}(j_{n,\delta} \vee 1)^{(\beta - 1)/\beta} =: B_{n,\delta}'
\end{align*}
on $\goodEvent$. Thus, if $B \geq B_{n,\delta}'$, then $\scoreEstimateShapeConstrained^{(B)}(x) = \scoreEstimateShapeConstrained(x)$ for all $x \in [a_{-{j_{n,\delta}}},a_{{j_{n,\delta}}}]$, so by~\eqref{eq:multiscale-holder-bound} and~\eqref{eq:holder-tail-bound},
\begin{align}
&\int_{-\infty}^{\infty}\bigl(\scoreEstimateShapeConstrained^{(B)} - \psi_0\bigr)^2 f_0 \leq \biggl\{\int_{a_{-{j_{n,\delta}}}}^{a_{{j_{n,\delta}}}} (\scoreEstimateShapeConstrained - \psi_0)^2 f_0 + \int_{-\infty}^{a_{-{j_{n,\delta}}}}\psi_0^2 f_0 + \int_{a_{{j_{n,\delta}}}}^{\infty}\psi_0^2 f_0\biggr\} \notag \\
\label{eq:holder-multiscale-proj}
&\hspace{1.5cm}\lesssim_\beta r'\min\biggl\{\Bigl(\frac{L^{2/\beta}}{r'}\varepsilon_{n,\delta}\Bigr)^{\beta/(2\beta + 1)}\log_+^{\frac{\beta - 1}{2\beta + 1}}\Bigl(\frac{L^{2/\beta}}{r'}\Bigr) + \frac{L^{2/\beta}}{r'}\varepsilon_{n,\delta}\log_+^{3 - \frac{2}{\beta}}\Bigl(\frac{1}{\varepsilon_{n,\delta}}\Bigr),\,1\biggr\}
\end{align}
on $\goodEvent$. In other words, the upper bound on $\mathcal{M}_n(\delta,\mathcal{F}_{\beta,L,r})$ is achieved by $\hat{\psi}_{n,\delta}^{(B)}$ for any $B \geq B_{n,\delta}'$.
Moreover, by considering the estimator $\tilde\psi \equiv 0$, we deduce from~\eqref{eq:holder-fisher-information} that $\mathcal{M}_n(\mathcal{F}_{\beta,L,r}) \leq \sup_{f_0 \in \mathcal{F}_{\beta,L,r}} i(f_0) \leq s$. Now let $\delta = 1/n$ 
and $B = B_{n,\delta}' \asymp L^{1/\beta}\log_+^{1 - \frac{1}{\beta}}(1/\varepsilon_{n,\delta}) \asymp L^{1/\beta}\log^{1 - \frac{1}{\beta}}n$. 
Then by Proposition~\ref{prop:confidenceBandsAreValid} and the fact that $|\scoreEstimateShapeConstrained^{(B)}| \leq B$, together with~\eqref{eq:holder-fisher-information} and~\eqref{eq:holder-multiscale-proj},
\begin{align*}
&\hspace{-0.3cm}\mathcal{M}_n(\mathcal{F}_{\beta,L,r}) \\
&\leq \sup_{f_0 \in \mathcal{F}_{\beta,L,r}} \E\biggl(\int_{-\infty}^{\infty} \bigl(\scoreEstimateShapeConstrained^{(B)} - \psi_0\bigr)^2 f_0\biggr) \wedge s \\
&\lesssim_\beta \min\biggl\{(B^2 + r')\,\delta + r'\Bigl(\frac{L^{2/\beta}}{r'}\varepsilon_{n,\delta}\Bigr)^{\beta/(2\beta + 1)}\log_+^{\frac{\beta - 1}{2\beta + 1}}\Bigl(\frac{L^{2/\beta}}{r'}\Bigr) + r'\frac{L^{2/\beta}}{r'}\varepsilon_{n,\delta}\log_+^{3 - \frac{2}{\beta}}\Bigl(\frac{1}{\varepsilon_{n,\delta}}\Bigr),\,r'\biggr\} \\
&\lesssim_\beta r'\min\biggl\{\frac{L^{2/\beta}}{r'} \cdot \frac{\log^{2(1 - \frac{1}{\beta})}n}{n} + \Bigl(\frac{L^{2/\beta}}{r'} \cdot \frac{\log n}{n}\Bigr)^{\beta/(2\beta + 1)}\log_+^{\frac{\beta - 1}{2\beta + 1}}n + \frac{L^{2/\beta}}{r'} \cdot \frac{\log n}{n} \log^{3 - \frac{2}{\beta}}n,\,1\biggr\} \\
&\lesssim_\beta r'\min\biggl\{\Bigl(\frac{L^{2/\beta}}{r'} \cdot \frac{\log^{(4\beta - 2)/\beta}n}{n}\Bigr)^{\beta/(2\beta + 1)},\,1\biggr\},
\end{align*}
as required.
\end{proof}

\subsection{Proofs of minimax lower bounds}
\label{subsec:proofs-lower-bds}

\subsubsection{Proofs for Sections~\ref{subsec:main-results-prelim} and~\ref{subsec:tail-growth}}

We begin with a lemma based on the construction discussed in Section~\ref{subsec:main-results-prelim}, which is used to prove the lower bounds over both the class $\mathcal{F}$ in Proposition~\ref{prop:risk-inf} and $\mathcal{J}_{\gamma,L,r}$ in Theorem~\ref{thm:tail-growth-lower-bd}. 

\begin{lemma}
\label{lem:fa-2pt}
Given $a \geq \sqrt{2}$ and $h \in [0,1/a]$, define $x_a := a - \frac{1}{a}$ and $\psi_{a,h},\phi_{a,h} \colon \R \to \R$ by
\[
\psi_{a,h}(x) := -\sgn(x)\Bigl(a\Ind_{\{|x| \geq x_a + h\}} + \frac{a}{2}\Ind_{\{|x| \in (x_a - h, x_a + h)\}} \Bigr), \qquad \phi_{a,h}(x) := \int_0^x \psi_{a,h}.
\]
Then $f_{a,h} := e^{\phi_{a,h}}/\int_{-\infty}^\infty e^{\phi_{a,h}}$ is a symmetric density with antisymmetric score function $\psi_{a,h}$, and $f_{a,h} \in \mathcal{J}_{\gamma,a^\gamma,2}$ for any $\gamma \in (0,1]$. Moreover, the corresponding distribution $P_{a,h}$ satisfies
\begin{equation}
\label{eq:fa-divergences}
\chi^2(P_{a,h},P_{a,0}) := \int_{-\infty}^\infty \Bigl(\frac{f_{a,h}}{f_{a,0}} - 1\Bigr)^2 f_{a,0} < 2ah^3 \quad\text{and}\quad \int_{-\infty}^\infty (\psi_{a,0} - \psi_{a,h})^2 (f_{a,0} \wedge f_{a,h}) \geq \frac{ah}{2e}.
\end{equation}
Consequently, for $\delta \in (0,1/4)$, we have
\begin{align*}
\sup_{h \in (0,1/a]} \mathcal{M}_n(\{f_{a,0},f_{a,h}\}) \gtrsim \Bigl(\frac{a^2}{n}\Bigr)^{1/3} \wedge 1, \quad\;\; \sup_{h \in (0,1/a]} \mathcal{M}_n(\delta,\{f_{a,0},f_{a,h}\}) \gtrsim \Bigl(\frac{a^2\log(1/\delta)}{n}\Bigr)^{1/3} \wedge 1.
\end{align*}
\end{lemma}


The functions $\phi_{a,h}$ and $\psi_{a,h}$ are illustrated in Figure~\ref{fig:lower-bd-tails-c1} in the main text. 

\begin{proof}
Since $h \in [0,1/a]$, we have $x_a - h  \geq a - \frac{2}{a} \geq 0$. For $x \in \R$, we have $\phi_{a,0}(x) = -a(|x| - x_a)_+$ and
\[
\phi_{a,h}(x) = \int_0^x \psi_{a,h} = \phi_{a,0}(x) - \frac{a}{2}\bigl(h - \bigl||x| - x_a\bigr|\bigr)_+ \leq \phi_{a,0}(x),
\]
with equality if and only if $|x| \notin (x_a - h, x_a + h)$. Therefore,
\[
C_h := \int_{-\infty}^\infty e^{\phi_{a,h}} \leq \int_{-\infty}^\infty e^{\phi_{a,0}} = 2x_a + 2\int_{x_a}^\infty e^{-a(x - x_a)} = 2a =: C_0,
\]
so because $\psi_{a,h}$ is decreasing, $f_{a,h}$ is a continuous and symmetric log-concave density with antisymmetric score function $\psi_{a,h}$. Since $e^{\phi_{a,h}} \leq e^{\phi_{a,0}} \leq 1$, we have
\begin{align}
\label{eq:fa-normalisation}
C_0 - C_h = \int_{-\infty}^\infty (e^{\phi_{a,0}} - e^{\phi_{a,h}}) \leq \int_{-\infty}^\infty (\phi_{a,0} - \phi_{a,h}) = a\int_{-h}^h (h - |x|) \laplaced x = ah^2 \leq \frac{1}{a} \leq \frac{C_0}{4}.
\end{align}

For $\gamma \in (0,1)$, we claim that $f_{a,h} \in \mathcal{J}_{\gamma,a^\gamma,2}$ for all $h \in [0,1/a]$. Indeed, since $f_{a,h}$ is symmetric, the corresponding distribution function $F_{a,h}$ satisfies $F_{a,h}(x) = 1 - F_{a,h}(-x)$ for $x \in \R$. First consider $h = 0$. If $x \geq x_a$, then $F_{a,0}(-x) \leq 1/2 \leq F_{a,0}(x)$ and
\begin{align*}
F_{a,0}(-x)^{\frac{1 - \gamma}{2}} = \bigl(1 - F_{a,0}(x)\bigr)^{\frac{1 - \gamma}{2}} &= \biggl(\frac{1}{2a} \int_x^\infty e^{-a(z - x_a)} \laplaced z\biggr)^{\frac{1 - \gamma}{2}} \\
&\leq \Bigl(\frac{1}{2a^2}\Bigr)^{\frac{1 - \gamma}{2}} = \frac{2^{-\frac{1 - \gamma}{2}}a^\gamma}{a} = \frac{2^{-\frac{1 - \gamma}{2}}a^\gamma}{|\psi_{a,0}(x)|} \leq \frac{2^{-\frac{1 - \gamma}{2}}a^\gamma}{|\psi_{a,0}(-x)|}.
\end{align*}
Moreover, if $x \in [0,x_a)$, then $|\psi_{a,0}(x)| = 0 \leq a^\gamma/\min\bigl(F_{a,0}(x), 1 - F_{a,0}(x)\bigr)^{\frac{1 - \gamma}{2}}$. Since $i(f_{a,0}) = \int_{-\infty}^\infty \psi_{a,0}^2 f_{a,0} = 2a^2\int_{x_a}^\infty f_{a,0} = 1$, we have $f_{a,0} \in \mathcal{J}_{\gamma,a^\gamma,1} \subseteq \mathcal{J}_{\gamma,a^\gamma,2}$.

Now consider $h \in (0,1/a]$. If $x \leq 0$, then $F_{a,h}(x) \leq 1/2 \leq F_{a,h}(-x)$, and by~\eqref{eq:fa-normalisation},
\begin{align*}
1 - F_{a,h}(-x) = F_{a,h}(x) = \frac{\int_{-\infty}^x e^{\phi_{a,h}}}{C_h} &< \frac{2\int_{-\infty}^x e^{\phi_{a,0}}}{C_0} \\
&= 2F_{a,0}(x) \leq \Bigl(\frac{a^\gamma}{|\psi_{a,h}(x)|}\Bigr)^{\frac{2}{1 - \gamma}} = \Bigl(\frac{a^\gamma}{|\psi_{a,h}(-x)|}\Bigr)^{\frac{2}{1 - \gamma}}.
\end{align*}
For $x \in [-x_a + h,0]$, we have $\frac{1}{2} - F_{a,h}(x) = \int_x^0 f_{a,h} = \int_x^0 e^{\phi_{a,0}}/C_h \geq \int_x^0 e^{\phi_{a,0}}/C_0 = \frac{1}{2} - F_{a,0}(x)$, so because $|\psi_{a,h}(x)| \leq a \Ind_{\{|x| \geq x_a - h\}}$, we have
\begin{align*}
i(f_{a,h}) = 2a^2\int_{-\infty}^{-x_a + h} f_{a,h} &= 2a^2 F_{a,h}(-x_a + h) \leq 2a^2 F_{a,0}(-x_a + h) \\
&= 2a^2\Bigl(\int_{-\infty}^{-x_a} \frac{e^{a(x_a + x)}}{2a} \laplaced x + \int_{-x_a}^{-x_a + h} \frac{1}{2a}\Bigr) = 1 + ah \leq 2.
\end{align*}
Thus, $f_{a,h} \in \mathcal{J}_{\gamma,a^\gamma,2}$, as claimed. 

Next, by~\eqref{eq:fa-normalisation}, $c_h := C_0/C_h \in [1,4/3]$ and $\log f_{a,h} - \log f_{a,0} = \phi_{a,h} - \phi_{a,0} + \log c_h \leq \log(4/3)$. Moreover,
\[
c_h - 1 \leq \frac{4(C_0 - C_h)}{3C_0} \leq \frac{4ah^2}{3C_0} = \frac{2h^2}{3}. 
\]
Since $f_{a,0} \leq 1/C_0 = 1/(2a)$ on $\R$ and $|e^x - 1| \leq 4|x|/3$ for $x \in (-\infty,\log(4/3)]$, it follows that
\begin{align*}
\chi^2(P_{a,h},P_{a,0}) &= \int_{-\infty}^\infty \Bigl(\frac{f_{a,h}}{f_{a,0}} - 1\Bigr)^2 f_{a,0} \leq 2\int_{x_a - h}^{x_a + h} (e^{\phi_{a,h} - \phi_{a,0} + \log c_h} - 1)^2\,\frac{1}{2a} + \int_{-\infty}^\infty (c_h - 1)^2 f_{a,0} \\
&\leq \frac{16}{9a} \int_{x_a - h}^{x_a + h} (\phi_{a,h} - \phi_{a,0} + \log c_h)^2 + (c_h - 1)^2 \\
&\leq \frac{16}{9a} \biggl(\frac{a^2}{2}\int_{-h}^h (h - |x|)^2 \laplaced x + 4h\log^2 c_h\biggr) + \frac{4h^4}{9} \\
&\leq \frac{16}{9a}\Bigl(\frac{a^2 h^3}{3} + \frac{16h^5}{9}\Bigr) + \frac{4h^4}{9} = ah^3\Bigl(\frac{16}{27} + \frac{256h^2}{81a^2} + \frac{4h}{9a}\Bigr) < 2ah^3.
\end{align*}
Moreover, $|\psi_{a,0}(x) - \psi_{a,h}(x)| = a\Ind_{\|x| \in (x_a - h, x_a + h\}}/2$ for all $x \in \R$. If $|x| \in [x_a - h, x_a + h]$, then
\[
f_{a,h}(x) = \frac{e^{\phi_{a,h}(x)}}{2a} \geq \frac{e^{\phi_{a,0}(x_a + h)}}{2a} = \frac{e^{-ah}}{2a} \geq \frac{1}{2ea}
\]
for all $h \in [0,1/a]$. Therefore,
\[
\int_{-\infty}^\infty (\psi_{a,0} - \psi_{a,h})^2 (f_{a,0} \wedge f_{a,h}) \geq \frac{a^2}{2} \int_{x_a - h}^{x_a + h} f_{a,0} \wedge f_{a,h} \geq \frac{ah}{2e},
\]
which establishes~\eqref{eq:fa-divergences}. Finally, taking
\[
h' := \Bigl(\frac{1}{2an}\Bigr)^{1/3} \wedge \frac{1}{a},
\]
we have $h' \in (0,1/a]$ and $\KL(P_{a,h'}^{\otimes n},P_{a,0}^{\otimes n}) = n\KL(P_{a,h'},P_{a,0}) \leq n\chi^2(P_{a,h'},P_{a,0}) < 1$, so by Le Cam's two-point lemma \citep[e.g.][Lemma~8.4]{samworth24modern},
\begin{align*}
\mathcal{M}_n(\{f_{a,0},f_{a,h'}\}) &\geq \inf_{\hat{\psi}_n \in \hat{\Psi}_n} \max_{h \in \{0,h'\}} \E_{f_{a,h}}\biggl(\int_{-\infty}^\infty (\widehat{\psi}_n - \psi_{a,h})^2 (f_{a,0} \wedge f_{a,h'})\biggr) \\
&\geq \frac{1}{16}\int_{-\infty}^\infty (\psi_{a,0} - \psi_{a,h'})^2 (f_{a,0} \wedge f_{a,h'}) \geq \frac{ah'}{32e} = \frac{1}{32e}\min\Bigl\{\Bigl(\frac{a^2}{2n}\Bigr)^{1/3},\,1\Bigr\}.
\end{align*}
On the other hand, given $\delta \in (0,1/4]$, let
\[
h'' := \Bigl\{\frac{1}{2an}\log\Bigl(\frac{1}{4\delta(1 - \delta)}\Bigr)\Bigr\}^{1/3} \wedge \frac{1}{a}.
\]
Then $\KL(P_{a,h''}^{\otimes n},P_{a,0}^{\otimes n}) \leq n\chi^2(P_{a,h''},P_{a,0}) < \log\bigl(\frac{1}{4\delta(1 - \delta)}\bigr)$, so by \citet[Corollary~6]{ma2025high},
\begin{align*}
\mathcal{M}_n(\delta,\{f_{a,0},f_{a,h''}\}) &\geq \inf_{\hat{\psi}_n \in \hat{\Psi}_n} \max_{h \in \{0,h''\}}\mathrm{Quantile}_{f_{a,h}}\biggl(1 - \delta, \int_{-\infty}^\infty (\widehat{\psi}_n - \psi_{a,h})^2 (f_{a,0} \wedge f_{a,h''})\biggr) \\
&\geq \frac{1}{4}\int_{-\infty}^\infty (\psi_{a,0} - \psi_{a,h''})^2 (f_{a,0} \wedge f_{a,h''}) \\
&\geq \frac{ah''}{8e} = \frac{1}{8e}\min\biggl(\Bigl\{\frac{a^2}{2n}\log\Bigl(\frac{1}{4\delta(1 - \delta)}\Bigr)\Bigr\}^{1/3},\,1\biggr).
\end{align*}
This concludes the proof.
\end{proof}

\begin{proof}[Proof of Proposition~\ref{prop:risk-inf}]
By Lemma~\ref{lem:fa-2pt} with $a = (2 \vee n)^{1/2}$, we have $\mathcal{M}_n(\mathcal{F}_2) \wedge \mathcal{M}_n(\delta,\mathcal{F}_2) \gtrsim 1$. For general $r \in (0,\infty)$, we have $\mathcal{F}_r = \{(r/2)^{1/2}f_0\bigl((r/2)^{1/2}\cdot\bigr) : f_0 \in \mathcal{F}_2\}$, so
\[
\mathcal{M}_n(\mathcal{F}_r) = \frac{r}{2} \mathcal{M}_n(\mathcal{F}_2) \gtrsim r, \qquad \mathcal{M}_n(\delta,\mathcal{F}_r) = \frac{r}{2} \mathcal{M}_n(\delta,\mathcal{F}_2) \gtrsim r.
\]
On the other hand, by considering the estimator $\hat{\psi}_n \equiv 0$, we obtain $\mathcal{M}_n(\delta,\mathcal{F}_r) \vee \mathcal{M}_n(\mathcal{F}_r) \leq r$.
\end{proof}

\begin{proof}[Proof of Theorem~\ref{thm:tail-growth-lower-bd}]
For both the minimax risk and minimax quantile, we obtain two separate lower bounds and take their maximum at the end of the proof to establish the desired conclusion. The first bound follows directly from Lemma~\ref{lem:fa-2pt}, while the second is derived by applying Le Cam's two-point lemma and its high-probability analogue \citep[Theorem~4 and Corollary~6]{ma2025high} to a perturbation of a different base density defined below.

\begin{figure}
\begin{center}
\subfigure{\includegraphics[width=0.45\textwidth]{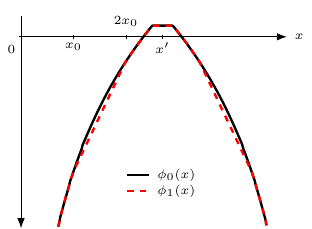}}
\hfill
\subfigure{\includegraphics[width=0.45\textwidth]{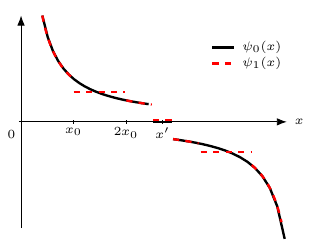}}
\end{center}
\caption{The log-densities $\phi_0$ and $\phi_1$ (left) and corresponding score functions $\psi_0$ and $\psi_1$ (right) used in the proof of Theorem~\ref{thm:tail-growth-lower-bd}.}
\label{fig:lower-bd-tails-c2}
\end{figure}

Given $\tilde{L} \geq 1$, let $b := (1 + \gamma)/(1 - \gamma) > 1$ and define $g_0 \colon \R \to \R$ by
\[
g_0(x) :=
\begin{cases}
\tilde{L}^{b+1}x_+^b \;&\text{if }x \leq \tilde{L}^{-1/\gamma} \\
\tilde{L}^{-1/\gamma} \;&\text{if }x \in (\tilde{L}^{-1/\gamma},x'], \\
g_0(2x' - x) \;&\text{if }x > x',
\end{cases}
\]
where $x' := \frac{1}{2}\tilde{L}^{1/\gamma} + (1 - \frac{1 - \gamma}{2})\tilde{L}^{-1/\gamma} \geq \tilde{L}^{-1/\gamma}$. Then $g_0$ is a continuous density on $\R$ whose score function~$\psi_0$ satisfies
\[
\psi_0(x) = \frac{g_0'(x)}{g_0(x)} =
\begin{cases}
\,\dfrac{b}{x} \;&\text{if }x \in (0,\tilde{L}^{-1/\gamma}) \\[6pt]
\,\,0 \;&\text{if }x \in (\tilde{L}^{-1/\gamma}, x'] \\[3pt]
-\psi_0(2x' - x) \;&\text{if } x \in (x',2x') \setminus \{2x' - \tilde{L}^{-1/\gamma}\};
\end{cases}
\]
see Figure~\ref{fig:lower-bd-tails-c2}.
Thus, $\psi_0$ is decreasing on $(0,2x')$, so $g_0 \in \mathcal{F}$.  Next, for some $x_0 \in (0,\tilde{L}^{-1/\gamma}/2]$ that we will specify later on, let $\phi_0 := \log g_0$ and define $\phi_1 \colon (0,\infty) \to \R$ by
\[
\phi_1(x) := 
\begin{cases}
\dfrac{2x_0 - x}{x_0}\phi_0(x_0) + \dfrac{x - x_0}{x_0}\phi_0(2x_0) = \phi_0(x_0) + \dfrac{x - x_0}{x_0}\,b\log 2 \;&\text{if }x \in [x_0,2x_0] \\[5pt]
\,\phi_0(x) \;&\text{if }x \in (0,x'] \setminus (x_0,2x_0) \\[3pt]
\phi_1(2x' - x) \;&\text{if } x > x'.
\end{cases}
\]
Since $\phi_0$ is concave and continuous on $(0,2x')$, the same is true of $\phi_1$, and $\phi_1 \leq \phi_0$ on $(0,2x')$. Thus, $C_1 := \int_0^\infty e^{\phi_1} \leq \int_0^\infty e^{\phi_0} = 1$, and $g_1 := e^{\phi_1}/C_1 \in \mathcal{F}$ has an antisymmetric score function $\psi_1$ satisfying
\[
\psi_1(x) =
\begin{cases}
\dfrac{b\log 2}{x_0} \;&\text{if }x \in (x_0,2x_0) \\[6pt]
\,\psi_0(x) \;&\text{if }x \in (0,x'] \setminus [x_0,2x_0] \\[3pt]
-\psi_1(2x' - x) \;&\text{if } x \in (x',2x').
\end{cases}
\]
It follows that
\begin{align*}
1 - C_1 = \int_{-\infty}^\infty (e^{\phi_0} - e^{\phi_1}) &= 2\int_{x_0}^{2x_0} \biggl(\tilde{L}^{b+1}x^b - \exp\Bigl\{\log(\tilde{L}^{b+1}x_0^b) + \frac{x - x_0}{x_0}\,b\log 2\Bigr\}\biggr) \laplaced x \notag \\
&= 2\tilde{L}^{b+1}x_0^{b + 1} \int_1^2 (y^b - 2^{b(y - 1)}) \laplaced y \leq \frac{c_b}{2}(\tilde{L}x_0)^{2/(1 - \gamma)},
\end{align*}
where $c_b := 1 \vee \bigl(4\int_1^2 (y^b - 2^{b(y - 1)}) \laplaced y\bigr) > 0$ since $y > 2^{y-1}$ for $y \in (1,2)$. Therefore, if $\tilde{L} \geq c_b^{\gamma/2}$, then $C_1 = 1 - \frac{1}{2}c_b(\tilde{L}x_0)^{2/(1 - \gamma)} \geq 1 - \frac{1}{2}c_b\tilde{L}^{-2/\gamma} \geq 1/2$. 

Now given any $L \geq 2c_b^{\gamma/2}(1 - \gamma)^{\frac{1 - \gamma}{2}}b\log 2 =: L_\gamma$, define
\[
\tilde{L} := \frac{L}{2(1 - \gamma)^{\frac{1 - \gamma}{2}}b\log 2} \geq c_b^{\gamma/2}, \qquad r_\gamma := \frac{(4b\log 2)^2}{b - 1}.
\]
We claim that $g_0,g_1 \in \mathcal{J}_{\gamma,L,r_\gamma}$. Indeed, $g_0,g_1$ are symmetric about $x'$ by construction, so the corresponding distribution functions $G_0,G_1$ satisfy $G_k(x) = 1 - G_k(2x' - x) \leq 1/2$ for $x \in [0,x']$ and $k \in \{0,1\}$. If $x \in (0,\tilde{L}^{-1/\gamma})$, then
\begin{align*}
G_0(x)^{\frac{1 - \gamma}{2}} = \biggl(\int_0^x \tilde{L}^{b+1}z^b \laplaced z\biggr)^{\frac{1 - \gamma}{2}} = \tilde{L}x\Bigl(\frac{1 - \gamma}{2}\Bigr)^{\frac{1 - \gamma}{2}} = \frac{\tilde{L}b}{|\psi_0(x)|}\Bigl(\frac{1 - \gamma}{2}\Bigr)^{\frac{1 - \gamma}{2}}.
\end{align*}
For such $x$, we also have $0 < \psi_1(x) \leq (2\log 2)\psi_0(x)$, so
\[
G_1(x) = \frac{\int_{-\infty}^x e^{\phi_1}}{C_1} \leq \frac{\int_{-\infty}^x e^{\phi_0}}{C_1} \leq 2G_0(x) = (1 - \gamma)\Bigl(\frac{\tilde{L}b}{|\psi_0(x)|}\Bigr)^{\frac{2}{1 - \gamma}} \leq \Bigl(\frac{L}{|\psi_1(x)|}\Bigr)^{\frac{2}{1 - \gamma}}.
\] 
Moreover, $\psi_0(x) = \psi_1(x) = 0$ for $x \in (\tilde{L}^{-1/\gamma}, x']$. On the other hand, if $x \in (x',2x')$ and $k \in \{0,1\}$, then $G_k(x) = 1 - G_k(2x' - x) \geq 1/2$ and
\begin{align*}
|\psi_k(x)| &= |\psi_k(2x' - x)| \leq \frac{L}{G_k(2x' - x)^{(1 - \gamma)/2}} =
\frac{L}{\bigl(1 - G_k(x)\bigr)^{(1 - \gamma)/2}}.
\end{align*}
Therefore, for any $u \in (0,1)$ and $k \in \{0,1\}$, we deduce by taking $x = G_k^{-1}(u) \in (0,2x')$ that
\[
|J_k'(u)| = |\psi_k(x)| \leq \frac{L}{G_k(x)^{(1 - \gamma)/2} \wedge \bigl(1 - G_k(x)\bigr)^{(1 - \gamma)/2}} = \frac{L}{u^{(1 - \gamma)/2} \wedge (1 - u)^{(1 - \gamma)/2}}.
\]
Since 
$g_1 = e^{\phi_1}/C_1 \leq 2e^{\phi_0} = 2g_0$ on $(0,\tilde{L}^{-1/\gamma})$ and since $\psi_1 = \psi_0 = 0$ on $(\tilde{L}^{-1/\gamma},x']$, we also have
\begin{align*}
i(g_1) = 2\int_0^{\tilde{L}^{-1/\gamma}} \psi_1^2\,g_1 &\leq (4\log 2)^2 \int_0^{\tilde{L}^{-1/\gamma}} \psi_0^2\,g_0 = (8\log^2 2)\,i(g_0) \\
&= (4\log 2)^2 \int_0^{\tilde{L}^{-1/\gamma}} \tilde{L}^{b+1}b^2 x^{b - 2} \laplaced x = \frac{(4b\log 2)^2}{b - 1}.
\end{align*}
Thus, $g_0,g_1 \in \mathcal{J}_{\gamma,L,r_\gamma}$, as required. 

Next, $\phi_1 - \phi_0 - \log C_1 \leq \log 2$ on $(0,2x')$ and
\[
\int_{x_0}^{2x_0} (\phi_0 - \phi_1)^2 = \int_{x_0}^{2x_0} \biggl(b\log\frac{x}{x_0} - \Bigl(\frac{x}{x_0} - 1\Bigr)b\log 2\biggr)^2 \laplaced x 
< \frac{b^2 x_0}{527}.
\]
Hence, denoting by $Q_k$ the probability measure with density $g_k$ for $k \in \{0,1\}$, we have
\begin{align*}
\chi^2(Q_1,Q_0) = \int_{-\infty}^\infty \Bigl(\frac{g_1}{g_0} - 1\Bigr)^2 g_0 &\leq 2\int_{x_0}^{2x_0} (e^{\phi_1 - \phi_0 - \log C_1} - 1)^2\,\tilde{L}^{b+1}x^b \laplaced x + \int_{-\infty}^\infty \Bigl(\frac{1}{C_1} - 1\Bigr)^2 g_0 \\
&\leq 8\tilde{L}^{b+1}(2x_0)^b \int_{x_0}^{2x_0} (\phi_1 - \phi_0 - \log C_1)^2 + \Bigl(\frac{1}{C_1} - 1\Bigr)^2 \\
&\leq 16\tilde{L}^{b+1}(2x_0)^b\biggl(x_0\log^2 C_1 + \int_{x_0}^{2x_0} (\phi_1 - \phi_0)^2\biggr) + 4(1 - C_1)^2 \\
&< 16\tilde{L}^{b+1}(2x_0)^b\Bigl(x_0(4\log^2 2)(1 - C_1)^2 + \frac{b^2 x_0}{527}\Bigr) + c_b^2(\tilde{L}x_0)^{\frac{4}{1 - \gamma}} \\
&\leq 2^{b + 4}(\tilde{L}x_0)^{\frac{2}{1 - \gamma}}\Bigl(\log^2 2 + \frac{b^2}{527} + \frac{c_b^2}{2^{b + 4}}(\tilde{L}x_0)^{\frac{2}{1 - \gamma}}\Bigr) \leq c_b'(\tilde{L}x_0)^{{\frac{2}{1 - \gamma}}},
\end{align*}
where we can take $c_b' := 2^{b + 4}(\log^2 2 + b^2/527) + c_b$ since $(\tilde{L}x_0)^{2/(1 - \gamma)} \leq \tilde{L}^{-2/\gamma} \leq 1/c_b$. 

Thus, letting
\[
x_0 := \frac{1}{\tilde{L}}\Bigl\{\frac{1}{nc_b'}\log\Bigl(\frac{1}{4\delta(1 - \delta)}\Bigr)\Bigr\}^{(1 - \gamma)/2} \wedge \frac{\tilde{L}^{-1/\gamma}}{2},
\]
we have $2x_0 \leq \tilde{L}^{-1/\gamma}$, and 
\begin{align*}
\KL(Q_1^{\otimes n},Q_0^{\otimes n}) = n\KL(Q_1,Q_0) \leq n\chi^2(Q_1,Q_0) < \log\Bigl(\frac{1}{4\delta(1 - \delta)}\Bigr).
\end{align*}
Since $g_0(x) \wedge g_1(x) \geq g_0(x_0) \geq \tilde{L}^{b+1}x_0^b$  for $x \in [x_0,2x_0]$, it follows from \citet[Theorem~4 and Corollary~6]{ma2025high} that for any $\delta \in (0,1/4]$ and $L \geq L_\gamma$, we have
\begin{align*}
\mathcal{M}_n(\delta,\mathcal{J}_{\gamma,L,r_\gamma}) &\geq 
\inf_{\hat{\psi}_n \in \hat{\Psi}_n}\max_{j \in \{0,1\}}\mathrm{Quantile}_{g_j}\biggl(1 - \delta, \int_{-\infty}^\infty (\widehat{\psi}_n - \psi_j)^2 (g_0 \wedge g_1)\biggr) \\
&\geq \frac{1}{4}\int_{-\infty}^\infty (\psi_0 - \psi_1)^2 (g_0 \wedge g_1) \\
&\geq \frac{1}{2}\int_{x_0}^{2x_0} \tilde{L}^{b+1}x_0^b\,\Bigl(\frac{b}{x} - \frac{b\log 2}{x_0} \Bigr)^2\,dx = \frac{b^2 (1 - 2\log^2 2)}{4}\,\tilde{L}^{b+1}x_0^{b - 1} \\
&\gtrsim_\gamma \min\biggl\{L^2\Bigl(\frac{\log(1/\delta)}{n}\Bigr)^\gamma,\,1\biggr\}.
\end{align*}
On the other hand, taking
\[
x_0 := \frac{1}{\tilde{L}}\Bigl(\frac{1}{nc_b'}\Bigr)^{(1 - \gamma)/2} \wedge \frac{\tilde{L}^{-1/\gamma}}{2},
\]
we have $\KL(Q_1^{\otimes n},Q_0^{\otimes n}) \leq 1$, so by Le Cam's two-point lemma \citep[e.g.][Lemma~8.4]{samworth24modern} and a similar calculation to that above,
\begin{align*}
\mathcal{M}_n(\mathcal{J}_{\gamma,L,r_\gamma}) \geq \inf_{\hat{\psi}_n \in \hat{\Psi}_n} \max_{j \in \{0,1\}} \E\biggl(\int_{-\infty}^\infty (\widehat{\psi}_n - \psi_j)^2 (g_0 \wedge g_1)\biggr) &\geq \frac{1}{16}\int_{-\infty}^\infty (\psi_0 - \psi_1)^2 (g_0 \wedge g_1) \\
&\gtrsim_\gamma \min\Bigl(\frac{L^2}{n^\gamma},\,1\Bigr).
\end{align*}

For general $L,r > 0$, let $r'' := r \wedge (r_\gamma L^2/L_\gamma^2)$ and $\lambda := (r''/r_\gamma)^{1/2}$. Then $L/\lambda \geq L_\gamma$ and $\mathcal{J}_{\gamma,L,r} \supseteq \mathcal{J}_{\gamma,L,r''} = \bigl\{x \mapsto \lambda f_0(\lambda x) : f_0 \in \mathcal{J}_{\gamma,\lambda^{-1}L,r_\gamma}\bigr\}$, so
\begin{align}
\mathcal{M}_n(\delta,\mathcal{J}_{\gamma,L,r}) = \lambda^2\mathcal{M}_n(\delta,\mathcal{J}_{\gamma,\lambda^{-1}L,r_\gamma}) &\gtrsim_\gamma \lambda^2 \min\biggl\{\frac{L^2}{\lambda^2}\biggl(\frac{\log(1/\delta)}{n}\biggr)^\gamma,\,1\biggr\} \notag \\
\label{eq:gamma-lower-bd-2}
&\gtrsim_\gamma \min\biggl\{L^2\biggl(\frac{\log(1/\delta)}{n}\biggr)^\gamma,\,\bar{r}\biggr\},
\end{align}
and
\begin{align}
\label{eq:gamma-lower-bd-expectation-2}
\mathcal{M}_n(\mathcal{J}_{\gamma,L,r}) = \lambda^2\mathcal{M}_n(\mathcal{J}_{\gamma,\lambda^{-1}L,r_\gamma}) &\gtrsim_\gamma \min\Bigl\{\frac{L^2}{n^\gamma},\,\bar{r}\Bigr\}.
\end{align}

On the other hand, let $\breve{r} := r \wedge (2^{1 - \gamma}L^2)$ and $\lambda := (\breve{r}/2)^{1/2}$. Then $a := (L/\lambda)^{1/\gamma} \geq \sqrt{2}$ and $\mathcal{J}_{\gamma,L,r} \supseteq \mathcal{J}_{\gamma,L,\breve{r}} = \bigl\{x \mapsto \lambda f_0(\lambda x) : f_0 \in \mathcal{J}_{\gamma,a^\gamma,2}\bigr\}$, so by Lemma~\ref{lem:fa-2pt},
\begin{align}
\mathcal{M}_n(\delta,\mathcal{J}_{\gamma,L,r}) \geq \lambda^2\mathcal{M}_n(\delta,\mathcal{J}_{\gamma,a^\gamma,2}) &\geq \frac{\lambda^2}{8e}\min\biggl(\Bigl\{\frac{L^{2/\gamma}}{2\lambda^{2/\gamma}n}\log\Bigl(\frac{1}{4\delta(1 - \delta)}\Bigr)\Bigr\}^{1/3},\,1\biggr) \notag \\
\label{eq:gamma-lower-bd-1}
&\gtrsim_\gamma \bar{r}\min\biggl\{\Bigl(\frac{L^{2/\gamma}}{\bar{r}^{1/\gamma}} \cdot \frac{\log(1/\delta)}{n}\Bigr)^{1/3},\,1\biggr\}.
\end{align}
Similarly,
\begin{align}
\label{eq:gamma-lower-bd-expectation-1}
\mathcal{M}_n(\mathcal{J}_{\gamma,L,r}) \geq \lambda^2\mathcal{M}_n(\mathcal{J}_{\gamma,a^\gamma,2}) &\gtrsim_\gamma \bar{r}\min\Bigl\{\Bigl(\frac{L^{2/\gamma}}{\bar{r}^{1/\gamma}} \cdot \frac{1}{n}\Bigr)^{1/3},\,1\Bigr\}.
\end{align}

Finally, taking the maximum of the bounds in~\eqref{eq:gamma-lower-bd-2} and~\eqref{eq:gamma-lower-bd-1}, we obtain
\begin{align*}
\mathcal{M}_n(\delta,\mathcal{J}_{\gamma,L,r}) &\gtrsim_\gamma \max\biggl\{\bar{r}\biggl(\frac{L^{2/\gamma}}{\bar{r}^{1/\gamma}} \cdot \frac{\log(1/\delta)}{n}\biggr)^{1/3},\,\bar{r}\biggl(\frac{L^{2/\gamma}}{\bar{r}^{1/\gamma}} \cdot \frac{\log(1/\delta)}{n}\biggr)^\gamma\biggr\} \wedge \bar{r} \\
&= \bar{r}\biggl\{\biggl(\frac{L^{2/\gamma}}{\bar{r}^{1/\gamma}} \cdot \frac{\log(1/\delta)}{n}\biggr)^{\gamma \wedge \frac{1}{3}},\,1\biggr\},
\end{align*}
and similarly combining~\eqref{eq:gamma-lower-bd-expectation-2} and~\eqref{eq:gamma-lower-bd-expectation-1} yields
\[
\mathcal{M}_n(\mathcal{J}_{\gamma,L,r}) \gtrsim_\gamma \max\biggl\{\bar{r}\Bigl(\frac{L^{2/\gamma}}{\bar{r}^{1/\gamma}} \cdot \frac{1}{n}\Bigr)^{1/3},\,\frac{L^2}{n^\gamma}\biggr\} \wedge \bar{r} = \bar{r}\min\biggl\{\Bigl(\frac{L^{2/\gamma}}{\bar{r}^{1/\gamma}} \cdot \frac{1}{n}\Bigr)^{\gamma \wedge \frac{1}{3}},\,1\biggr\}. \qedhere
\]
\end{proof}

\subsubsection{Proofs for Section~\ref{subsec:holder}}

Turning now to the proof of Theorem~\ref{thm:holder-lower-bd}, we remark that the construction in the case $\beta = 1$ case is very similar to that in Lemma~\ref{lem:fa-2pt}.

\begin{proof}[Proof of Theorem~\ref{thm:holder-lower-bd}]
Let $\tilde{L} \geq 1$. Given $h \in (0,\tilde{L}^{-1/\beta})$, let 
\[
m := \Bigl\lceil\Bigl(\frac{1}{\tilde{L}^{1/\beta}h}\Bigr)^{\beta - 1}\Bigr\rceil,
\]
so that 
\begin{equation}
\label{eq:lower-bd-holder-height}
\tilde{L}^{1/\beta} = \tilde{L}h^{\beta - 1}\Bigl(\frac{1}{\tilde{L}^{1/\beta}h}\Bigr)^{\beta - 1} \leq \tilde{L}h^{\beta - 1}m \leq \tilde{L}h^{\beta - 1}\biggl(\Bigl(\frac{1}{\tilde{L}^{1/\beta}h}\Bigr)^{\beta - 1} + 1\biggr) \leq 2\tilde{L}^{1/\beta}.
\end{equation}
Noting that $m = 1$ if and only if $\beta = 1$, define the intervals
\[
I_j := 
\begin{cases}
(\tilde{L}^{1/\beta}, \tilde{L}^{1/\beta} + h) &\;\;\text{if }j = 1\\
\biggl(\tilde{L}^{1/\beta} + \Bigl\{\Bigl(\dfrac{j}{2}\Bigr)^{1/(\beta - 1)} -\dfrac{1}{2}\Bigr\}2h,\;\tilde{L}^{1/\beta} + \Bigl(\dfrac{j}{2}\Bigr)^{1/(\beta - 1)}2h\biggr) &\;\;\text{if }j \in \{2,\dotsc,m\},
\end{cases}
\]
which are of length $h$ and disjoint because $\max\bigl\{1,\bigl((j - 1)/2\bigr)^{1/(\beta-1)} + 1/2\bigr\} \leq (j/2)^{1/(\beta-1)}$ for $j \geq 2$ and $\beta \in (1,2]$. Moreover, when $m =1$, we have $\sup I_m - \tilde{L}^{1/\beta} = h \leq \tilde{L}^{-1/\beta}$. If instead $m \geq 2$, then $\beta > 1$ and by~\eqref{eq:lower-bd-holder-height},
\begin{equation}
\label{eq:lower-bd-holder-width}
\sup I_m - \tilde{L}^{1/\beta} = \Bigl(\frac{m}{2}\Bigr)^{1/(\beta - 1)}2h \leq \frac{2}{\tilde{L}^{1/\beta}}.
\end{equation}

\begin{figure}
\begin{center}
\subfigure{\includegraphics[width=0.32\textwidth]{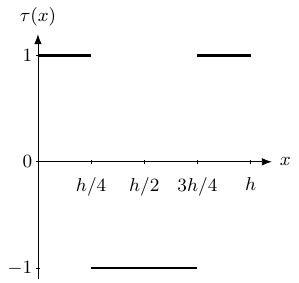}}
\hfill
\subfigure{\includegraphics[width=0.32\textwidth]{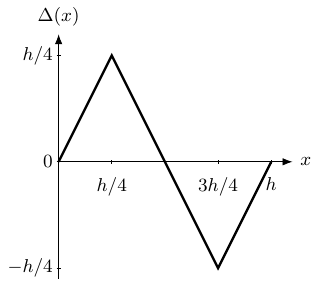}}
\hfill
\subfigure{\includegraphics[width=0.32\textwidth]{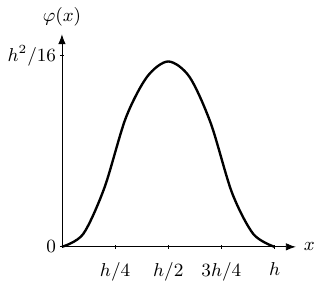}}
\end{center}
\caption{Plots of the functions $\tau$, $\Delta$ and $\varphi$ on the interval $[0,h]$.}
\label{fig:tauDeltaPhi}
\end{figure}

Next, let
\[
\tau := \Ind_{(0,h/4]} - \Ind_{(h/4,3h/4]} + \Ind_{(3h/4,h]} = \Ind_{(0,h]} - 2\Ind_{(h/4,3h/4]}.
\]
Defining $\Delta,\varphi \colon \R \to \R$ by $\Delta(x) := \int_0^x \tau$ and $\varphi(x) := \int_0^x \Delta$ (see Figure~\ref{fig:tauDeltaPhi}), we have $\Ind_{(0,h]}(x) + \zeta\tau(x) \geq 0$ for $x \in [0,h]$ and $\zeta \in \{0,1\}$, and $\Delta(x) = \varphi(x) = 0$ for $x \in \R \setminus [0,h]$. For $\zeta = (\zeta_1,\dotsc,\zeta_m) \in \{0,1\}^m$, define $\psi_\zeta \colon \R \to \R$ to be the decreasing continuous function with $\psi_\zeta(0) := 0$ and right derivative given by
\[
\psi_\zeta^{(\mathrm{R})}(x) := 
\begin{cases}
-\tilde{L}h^{\beta - 2}\sum_{j=1}^m (\Ind_{(0,h]} + \zeta_j\tau)(x - \inf I_j) &\text{for }x \geq 0 \\
\psi_\zeta^{(\mathrm{R})}(-x) &\text{for }x < 0.
\end{cases}
\]
Then $\psi_\zeta^{(\mathrm{R})}(x) = 0$ whenever $|x| \notin  \bigcup_{j=1}^m\,I_j =: E_m$, and
\[
\psi_{(0,\ldots,0)}(x) = -\tilde{L}h^{\beta - 2}\int_0^x \Ind_{E_m}
\]
for $x \geq 0$, as illustrated in Figure~\ref{fig:perturb}. Moreover, every $\psi_\zeta$ is an odd decreasing function with
\[
\psi_\zeta(x) = 
\begin{cases}
\psi_{(0,\ldots,0)}(x) -  \tilde{L}h^{\beta - 2}\sum_{j=1}^m \zeta_j \Delta(x - \inf I_j) \;\;&\text{for }x \in [0,\sup I_m] \\[3pt]
-\tilde{L}h^{\beta-1}m \;\;&\text{for }x > \sup I_m
\end{cases}
\]
and antiderivative $\phi_\zeta \colon \R \to (-\infty,0]$ given by $\phi_\zeta(x) := \int_0^x \psi_\zeta$, which is therefore symmetric and concave. Since $\psi_\zeta(x) \leq 0$ for $x \geq 0$ with equality when $|x| \leq \tilde{L}^{1/\beta}$, and $\psi_\zeta(x) \geq -\tilde{L}h^{\beta - 1}m$ for $x \geq 0$ with equality when $x \geq \sup I_m$, we have
\begin{align}
\label{eq:phi-bound-exp}
-\tilde{L}h^{\beta-1}m(|x| - \tilde{L}^{1/\beta})_+ \leq \phi_\zeta(x) \leq -\tilde{L}h^{\beta-1}m(|x| - \sup I_m)_+
\end{align}
for all $x \in \R$. Therefore, by~\eqref{eq:lower-bd-holder-height} and \eqref{eq:lower-bd-holder-width}, $C_\zeta := \log\bigl(\int_{-\infty}^\infty e^{\phi_\zeta}\bigr)$ satisfies
\begin{align}
2\tilde{L}^{1/\beta} + \frac{2}{\tilde{L}h^{\beta-1} m} &= \int_{-\infty}^\infty e^{-\tilde{L}h^{\beta-1}m(|x| - \tilde{L}^{1/\beta})_+} \laplaced x \leq e^{C_\zeta} \leq \int_{-\infty}^\infty e^{-\tilde{L}h^{\beta-1} m(|x| - \sup I_m)_+} \laplaced x \notag \\ 
\label{eq:expC-bd}
&\leq 2\Bigl(\tilde{L}^{1/\beta} + \frac{2}{\tilde{L}^{1/\beta}} + \frac{1}{\tilde{L}h^{\beta-1}m}\Bigr) \leq 2\Bigl(\tilde{L}^{1/\beta} + \frac{3}{\tilde{L}^{1/\beta}}\Bigr) \leq 8\tilde{L}^{1/\beta}.
\end{align}
Together with~\eqref{eq:lower-bd-holder-width}, this implies that $f_\zeta := e^{\phi_\zeta - C_\zeta}$ is a symmetric log-concave density with 
\begin{align}
i(f_\zeta) = \int_{-\infty}^\infty \psi_\zeta^2 f_\zeta \leq \frac{2(\tilde{L}h^{\beta - 1}m)^2}{e^{C_\zeta}}\int_{\tilde{L}^{1/\beta}}^\infty e^{\phi_\zeta} &\leq \frac{(\tilde{L}h^{\beta - 1}m)^2}{\tilde{L}^{1/\beta}} \int_{\tilde{L}^{1/\beta}}^\infty e^{-\tilde{L}h^{\beta - 1}m(x - \sup I_m)_+} \laplaced x \notag \\
\label{eq:holder-fisher-upper-bd}
&\leq \frac{(\tilde{L}h^{\beta - 1}m)^2}{\tilde{L}^{1/\beta}}\Bigl(\frac{2}{\tilde{L}^{1/\beta}} + \frac{1}{\tilde{L}h^{\beta-1}m}\Bigr) \leq 10.
\end{align}
Moreover, if $|x| \leq \sup I_m$, then by~\eqref{eq:phi-bound-exp},
\begin{align}
\frac{e^{-4}}{8\tilde{L}^{1/\beta}} \leq \frac{\exp\bigl(-\tilde{L}h^{\beta-1}m(\sup I_m - \tilde{L}^{1/\beta})\bigr)}{8\tilde{L}^{1/\beta}} &\leq \exp\bigl(\phi_\zeta(\sup I_m) - C_\zeta\bigr) \notag \\
\label{eq:fZeta-bd}
&\leq f_\zeta(x) \leq \frac{1}{e^{C_\zeta}} \leq \frac{1}{2\tilde{L}^{1/\beta}}.
\end{align}

\begin{figure}
\begin{center}
\includegraphics[scale=0.8]{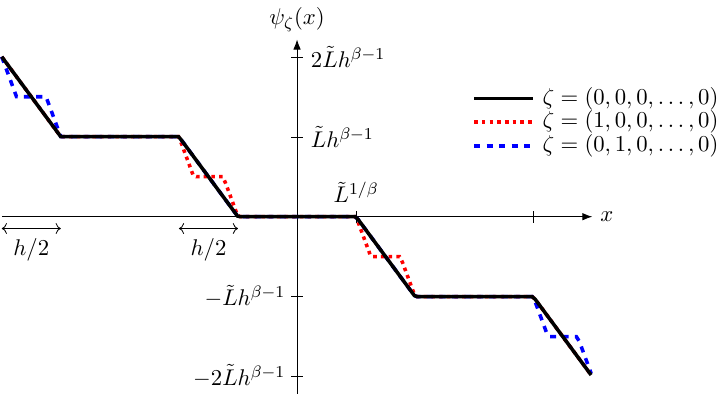}
\end{center}
\caption{Illustration of the functions $\psi_\zeta$ defined in the proof of Theorem~\ref{thm:holder-lower-bd} for various values of $\zeta \in \{0,1\}^m$. }
\label{fig:perturb}
\end{figure}

We now verify that $f_\zeta \in \mathcal{F}_{\beta,6\tilde{L},r}$ with $r = 10$ for every $\zeta \in \{0,1\}^m$. Let $x \in \R$. First, if $\beta = 1$ and $t \in \mathbb{R}$, then $|\psi_\zeta(x + t) - \psi_\zeta(x)| \leq 2\tilde{L}$. On the other hand, if $\beta \in (1,2]$ and $t \in [0,h]$, then
\[
|\psi_\zeta(x + t) - \psi_\zeta(x)| \leq 2\tilde{L}h^{-(2-\beta)}t \leq 2\tilde{L}t^{\beta-1}.
\]
Further, for $j \in [m]$, we have
\[
\sup_{x \in I_j}\psi_\zeta(x)-\inf_{x \in I_j}\psi_\zeta(x) = \psi_\zeta(\inf I_j) - \psi_\zeta(\sup I_j) =\tilde{L}h^{\beta-1}
\]
and $\inf I_{j+1}-\sup I_j$ is increasing in $j \in [m-1]$. Since $\psi_\zeta^{(\mathrm{R})} = 0$ on $\bigl(E_m \cup(-E_m)\bigr)^c$, it follows that for $t > h$, we have 
\begin{align*}
|\psi_\zeta(x + t) - \psi_\zeta(x)| 
&\leq \tilde{L}h^{\beta - 1}\sum_{j=1}^m \bigl(\mathbbm{1}_{\{I_j \cap (x,x+t) \neq \emptyset\}} + \mathbbm{1}_{\{(-I_j) \cap (x,x+t) \neq \emptyset\}}\bigr) \\
&\leq \tilde{L}h^{\beta - 1}\biggl\{\sum_{j=1}^m \bigl(\mathbbm{1}_{\{I_j \cap (\tilde{L}^{1/\beta},\tilde{L}^{1/\beta} + t)\neq\emptyset\}} + \mathbbm{1}_{\{(-I_j) \cap (-\tilde{L}^{1/\beta} - t,-\tilde{L}^{1/\beta}) \neq \emptyset\}}\bigr) + 2\biggr\} \\
&\leq \tilde{L}h^{\beta - 1}\Bigl(\frac{4(t+h)^{\beta-1}}{(2h)^{\beta-1}} + 2\Bigr) \leq \tilde{L}t^{\beta-1}(4 + 2) = 6\tilde{L}t^{\beta-1}.
\end{align*}
Since $f_\zeta$ is symmetric and~\eqref{eq:holder-fisher-upper-bd} holds, we have $f_\zeta \in \mathcal{F}_{\beta,6\tilde{L},10}$, as claimed.

For $\zeta = (\zeta_1,\dotsc,\zeta_m)$ and $\xi = (\xi_1,\dotsc,\xi_m) \in \{0,1\}^m$, let $S := \{j \in [m]: \zeta_j \neq \xi_j\}$. Since $\Delta(x) = 0$ for $x \geq h$ and $\int_0^h \Delta = 0$, we have $\psi_\zeta(x) = \psi_\xi(x)$ and $\phi_\zeta(x) = \phi_\xi(x)$ for all $x \in [0,\infty) \setminus \bigcup_{j \in S} I_j$. For $j \in S$, we have
\begin{align}
\label{eq:bulk-assouad-1}
\int_{I_j} (\psi_\zeta - \psi_\xi)^2 &= (\tilde{L}h^{\beta - 2})^2\int_0^h \Delta^2 = \frac{\tilde{L}^2 h^{2\beta-1}}{48}, \\
\label{eq:bulk-assouad-2}
\int_{I_j} (\phi_\zeta - \phi_\xi)^2 &= (\tilde{L}h^{\beta - 2})^2\int_0^h \varphi^2 = \frac{23\tilde{L}^2 h^{2\beta + 1}}{15360}.
\end{align}
Moreover, since $|\log(x) - \log(x')| \leq |x - x'|$ for $x,x' \geq 1$ and $|e^x - e^{x'}| \leq |x - x'|$ for $x,x' \leq 0$,~\eqref{eq:expC-bd} and~\eqref{eq:lower-bd-holder-height} imply that
\begin{align}
|C_\zeta - C_\xi| &= \bigl|\log\bigl(e^{C_\zeta}/(2\tilde{L}^{1/\beta})\bigr) - \log\bigl(e^{C_{\xi}}/(2\tilde{L}^{1/\beta})\bigr)\bigr| \leq \frac{1}{2\tilde{L}^{1/\beta}}\int_{-\infty}^\infty |e^{\phi_\zeta} - e^{\phi_\xi}| \notag \\
\label{eq:C-xi-bd}
&\leq \frac{1}{\tilde{L}^{1/\beta}}\sum_{j \in S} \int_{I_j} |\phi_\zeta - \phi_\xi| = \frac{|S|\tilde{L}h^{\beta - 2}}{\tilde{L}^{1/\beta}}\int_0^h \varphi 
= \frac{|S|\tilde{L}^{1 - 1/\beta}h^{\beta + 1}}{32} \\
&\leq \frac{m\tilde{L}^{1 - 1/\beta}h^{\beta + 1}}{32} \leq \frac{h^2}{16} \leq \frac{\tilde{L}^{-2/\beta}}{16} \leq \frac{1}{16}. \notag
\end{align}
Since $x \mapsto (e^x - 1)/x$ is increasing on $(0,\infty)$, we have 
\begin{align}
\int_{(0,\infty) \setminus \cup_{j \in S} I_j} (e^{\phi_\zeta - \phi_\xi + C_\xi - C_\zeta} - 1)^2\,f_\xi &= \int_{(0,\infty) \setminus \cup_{j \in S}I_j} (e^{C_\xi - C_\zeta} - 1)^2\,f_\xi \notag \\
&\leq 16^2(e^{1/16} - 1)^2(C_\xi - C_\zeta)^2\int_{(0,\infty) \setminus \cup_{j \in S}I_j} f_\xi. \nonumber
\end{align}
By~\eqref{eq:phi-bound-exp}, $|\phi_\zeta - \phi_\xi + C_\xi - C_\zeta| \leq \tilde{L}h^{\beta-1}m\,(\sup I_m - L^{1/\beta}) + 1/16 \leq 65/16$ on $E_m$, so for $j \in S$, it follows from~\eqref{eq:fZeta-bd} and~\eqref{eq:bulk-assouad-2} that
\begin{align}
\int_{I_j} (e^{\phi_\zeta - \phi_\xi + C_\xi - C_\zeta} - 1)^2\,f_\xi &\leq \biggl(\frac{16(e^{65/16} - 1)}{65} \biggr)^2\int_{I_j} (\phi_\zeta - \phi_\xi + C_\xi - C_\zeta)^2 \,f_\xi \notag \\
&< 396\int_{I_j}\bigl((\phi_\xi - \phi_\zeta)^2 + (C_\xi - C_\zeta)^2\bigr)\,f_\xi \notag \\
\label{eq:chi-squared-1}
& \leq 396\,\biggl(\frac{23\tilde{L}^2 h^{2\beta + 1}}{15360 \cdot 2\tilde{L}^{1/\beta}} + (C_\xi - C_\zeta)^2\int_{I_j}f_{\xi}\biggr).
\end{align}
Let $P_\zeta,P_\xi$ denote the probability measures with densities $f_\zeta,f_\xi$ respectively. Then by~\eqref{eq:C-xi-bd} and~\eqref{eq:chi-squared-1},
\begin{align}
\KL(P_\zeta,P_\xi) \leq \chi^2(P_\zeta,P_\xi) = \int_{-\infty}^\infty \Bigl(\frac{f_\zeta}{f_\xi} - 1\Bigr)^2 f_\xi &= 2\int_0^\infty (e^{\phi_\zeta - \phi_\xi + C_\xi - C_\zeta} - 1)^2\,f_\xi \notag \\ 
\label{eq:minimax-chi-squared}
&< \frac{396}{2^8}\bigl(|S|\tilde{L}^{2 - 1/\beta}h^{2\beta + 1} + |S|^2\tilde{L}^{2 - 2/\beta}h^{2\beta + 2}\bigr).
\end{align}
In particular, if $|S| = 1$, in which case we write $\zeta \sim \xi$ (and $\zeta \sim_j \xi$ if $S = \{j\}$ for $j \in [m]$), then by Pinsker's inequality,
\begin{align}
\label{eq:minimax-chi-squared-1}
\TV^2(P_\zeta^{\otimes n},P_\xi^{\otimes n}) \leq \frac{n\KL(P_\zeta,P_\xi)}{2} < \frac{396}{2^9}\,n\tilde{L}^{2 - 1/\beta}h^{2\beta + 1}(1 + \tilde{L}^{-1/\beta}h) \leq \frac{396}{2^8}\,n\tilde{L}^{2 - 1/\beta}h^{2\beta + 1}.
\end{align}

We now establish two separate lower bounds on the minimax $(1 -\delta)$th quantile via different choices of~$h$ and techniques from~\citet{ma2025high}. First, let
\[
h := \frac{1}{2\tilde{L}^{1/\beta}}\biggl(\Bigl(\frac{\tilde{L}^{2/\beta}}{n}\Bigr)^{1/(2\beta + 1)} \wedge 1\biggr),
\]
so that $\TV(P_\zeta^{\otimes n},P_\xi^{\otimes n}) \leq 1/2$ by~\eqref{eq:minimax-chi-squared-1}. Since $\bigl\{f_\zeta : \zeta \in \{0,1\}^m\bigr\} \subseteq \mathcal{F}_{\beta,6\tilde{L},10}$, it follows from~\eqref{eq:fZeta-bd}, Assouad's lemma~\citep[Lemma~9.7]{samworth24modern},~\eqref{eq:bulk-assouad-1},~\eqref{eq:minimax-chi-squared-1} and~\eqref{eq:lower-bd-holder-height} that
\begin{align}
\mathcal{M}_n(\mathcal{F}_{\beta,6\tilde{L},10}) &\geq \inf_{\hat{\psi}_n \in \hat{\Psi}_n} \max_{\zeta \in \{0,1\}^m} \E_{f_\zeta}\biggl(\sum_{j=1}^m \int_{(-I_j) \cup I_j} \frac{(\widehat{\psi}_n - \psi_\zeta)^2}{8e^4\tilde{L}^{1/\beta}}\biggr) \notag \\
&\geq \frac{1}{4} \sum_{j=1}^m \min_{\zeta,\xi \in \{0,1\}^m:\zeta \sim_j \xi}\,\int_{(-I_j) \cup I_j} \frac{(\psi_\zeta - \psi_\xi)^2}{8e^4\tilde{L}^{1/\beta}} \Bigl(1 - \max_{\zeta,\xi \in \{0,1\}^m : \zeta \sim \xi} \TV(P_\zeta^{\otimes n},P_\xi^{\otimes n})\Bigr) \notag \\
&\geq \frac{m}{4 \times 4e^4\tilde{L}^{1/\beta}}\frac{\tilde{L}^2 h^{2\beta-1}}{48} \Bigl(1 - \frac{1}{2}\Bigr) \notag \\
\label{eq:lower-bd-holder-unweighted}
&\geq \frac{\tilde{L}h^\beta}{16e^4 \times 48} \cdot \frac{1}{2} \geq \frac{1}{6144e^4}\Bigl\{\Bigl(\frac{\tilde{L}^{2/\beta}}{n}\Bigr)^{\beta/(2\beta + 1)} \wedge 1\Bigr\}.
\end{align}
Moreover, by~\eqref{eq:bulk-assouad-1} and~\eqref{eq:lower-bd-holder-height},
\begin{align}
\label{eq:lower-bd-holder-diameter}
d := \max_{\zeta,\xi \in \{0,1\}^m} \sum_{j=1}^m \int_{(-I_j) \cup I_j} \frac{(\psi_\zeta - \psi_\xi)^2}{8e^4\tilde{L}^{1/\beta}} = \frac{m\tilde{L}^2 h^{2\beta-1}}{192e^4\tilde{L}^{1/\beta}} \in \Bigl[\frac{\tilde{L}h^\beta}{192e^4},\frac{\tilde{L}h^\beta}{96e^4}\Bigr].
\end{align}
Hence, taking $\epsilon = 1/12$ and
\[
0 < \delta \leq \frac{1}{136} < \frac{\frac{1}{64} - \epsilon^2}{(1 + \epsilon)^2} \leq 
\frac{\frac{\tilde{L}h^\beta}{6144e^4} - \epsilon^2 d}{(1 + \epsilon)^2 d}
\]
in \citet[Theorem~8]{ma2025high}, we deduce from~\eqref{eq:fZeta-bd} that
\begin{align*}
\mathcal{M}_n(\delta,\mathcal{F}_{\beta,6\tilde{L},10}) &\geq \inf_{\hat{\psi}_n \in \hat{\Psi}_n} \max_{\zeta \in \{0,1\}^m} \mathrm{Quantile}_{f_\zeta}\biggl(1 - \delta,\,\int_{-\infty}^\infty (\widehat{\psi}_n - \psi_\zeta)^2 f_\zeta\biggr) \\
&\geq \inf_{\hat{\psi}_n \in \hat{\Psi}_n} \max_{\zeta \in \{0,1\}^m} \mathrm{Quantile}_{f_\zeta}\biggl(1 - \delta,\,\sum_{j=1}^m \int_{(-I_j) \cup I_j} \frac{(\widehat{\psi}_n - \psi_\zeta)^2}{8e^4\tilde{L}^{1/\beta}}\biggr) \\
&\geq  \epsilon^2 d \geq \frac{\tilde{L}h^\beta}{144 \times 192e^4} \geq \frac{1}{576 \times 192e^4}\biggl\{\Bigl(\frac{\tilde{L}^{2/\beta}}{n}\Bigr)^{\beta/(2\beta + 1)} \wedge 1\biggr\}.
\end{align*}
Since the function $x \mapsto x^3 + 3x^2(1 - x) =: H(x)$ satisfies $(H \circ H \circ H \circ H)(1/4) < 1/136$, taking $k = 4$ in \citet[Proposition~9]{ma2025high} shows that for every $\delta \in (0,1/4]$, we have
\begin{align}
\label{eq:lower-bd-holder-quantile-1}
\mathcal{M}_n(\delta,\mathcal{F}_{\beta,6\tilde{L},10}) \geq \frac{1}{4^4 \times 3^2 \times 576 \times 192 e^4}\Bigl\{\Bigl(\frac{\tilde{L}^{2/\beta}}{n}\Bigr)^{\beta/(2\beta + 1)} \wedge 1\Bigr\}.
\end{align}

Now for $\delta \in (0,1/4]$, we instead define
\[
h := \Bigl\{\frac{2^7}{1188n\tilde{L}}\log\Bigl(\frac{1}{4\delta(1 - \delta)}\Bigr)\Bigr\}^{1/(\beta + 2)} \wedge \frac{1}{2\tilde{L}^{1/\beta}}. 
\]
Fix $\zeta,\xi \in \{0,1\}^m$ with $S = [m]$. Then by~\eqref{eq:minimax-chi-squared} and~\eqref{eq:lower-bd-holder-height},
\begin{align*}
\KL(P_\zeta^{\otimes n},P_\xi^{\otimes n}) &< \frac{396}{2^8}\,n\bigl(m\tilde{L}^{2 - 1/\beta}h^{2\beta + 1} + m^2\tilde{L}^{2 - 2/\beta}h^{2\beta + 2}\bigr) \\
&\leq \frac{396n}{2^8}(2\tilde{L}h^{\beta + 2} + 4h^4) 
\leq \frac{396n\tilde{L}h^{\beta + 2}}{2^7}(1 + 2\tilde{L}^{-2/\beta}) \\
&\leq \frac{1188n\tilde{L}h^{\beta + 2}}{2^7} \leq \log\Bigl(\frac{1}{4\delta(1 - \delta)}\Bigr).
\end{align*}
Arguing as in~\eqref{eq:lower-bd-holder-unweighted}, we deduce from \citet[Theorem~4 and Corollary~6]{ma2025high} and~\eqref{eq:lower-bd-holder-diameter} that
\begin{align}
\mathcal{M}_n(\delta,\mathcal{F}_{\beta,6\tilde{L},10}) &\geq \inf_{\hat{\psi}_n \in \hat{\Psi}_n}\max_{\theta \in \{\zeta,\xi\}}\mathrm{Quantile}_{f_\theta}\biggl(1 - \delta,\,\sum_{j=1}^m \int_{(-I_j) \cup I_j} \frac{(\widehat{\psi}_n - \psi_\theta)^2}{8e^4\tilde{L}^{1/\beta}}\biggr) \notag \\
\label{eq:lower-bd-holder-quantile-2}
&\geq \frac{1}{4} \times \frac{\tilde{L}h^\beta}{192e^4} \geq \frac{1}{768e^4}\biggl(\biggl\{\frac{2^7\tilde{L}^{2/\beta}}{1188n}\log\Bigl(\frac{1}{4\delta(1 - \delta)}\Bigr)\biggr\}^{\beta/(\beta + 2)} \wedge 1\biggr).
\end{align}
Finally, for general $L,r > 0$ and $s := r \wedge 10(L/6)^{2/\beta}$, we have $\tilde{L} := L(s/10)^{-\beta/2}/6 \geq 1$. Since $\mathcal{F}_{\beta,L,r} \supseteq \mathcal{F}_{\beta,L,s} = \bigl\{x \mapsto (s/10)^{1/2}f_0\bigl((s/10)^{1/2}x\bigr) : f_0 \in \mathcal{F}_{\beta,6\tilde{L},10}\bigr\}$, it follows from~\eqref{eq:lower-bd-holder-unweighted} that
\begin{align}
\mathcal{M}_n(\mathcal{F}_{\beta,L,r}) = \frac{s}{10}\mathcal{M}_n(\mathcal{F}_{\beta,6\tilde{L},10}) &\geq \frac{s}{10 \times 6144e^4}\Bigl\{\Bigl(\frac{(L/6)^{2/\beta}}{ns/10}\Bigr)^{\beta/(2\beta + 1)} \wedge 1\Bigr\} \notag \\
\label{eq:holder-lower-bd-expectation}
&\gtrsim r'\min\biggl\{\biggl(\frac{L^{2/\beta}}{r'} \cdot \frac{1}{n}\biggr)^{\beta/(2\beta + 1)},\,1\biggr\}.
\end{align}
Moreover, from~\eqref{eq:lower-bd-holder-quantile-1}, for $\delta \in (0,1/4]$,
\begin{align}
\label{eq:holder-lower-bd-asssouad}
\mathcal{M}_n(\delta,\mathcal{F}_{\beta,L,r}) = \frac{s}{10}\mathcal{M}_n(\delta,\mathcal{F}_{\beta,6\tilde{L},10})
&\gtrsim r'\min\biggl\{\biggl(\frac{L^{2/\beta}}{r'} \cdot \frac{1}{n}\biggr)^{\beta/(2\beta + 1)},\,1\biggr\}.
\end{align}
Similarly, it follows from~\eqref{eq:lower-bd-holder-quantile-2} that
\begin{align}
\label{eq:holder-lower-bd-le-cam}
\mathcal{M}_n(\delta,\mathcal{F}_{\beta,L,r}) \gtrsim_\beta r'\min\biggl\{\biggl(\frac{L^{2/\beta}}{r'} \cdot \frac{\log(1/\delta)}{n}\biggr)^{\beta/(\beta + 2)},\,1\biggr\}.
\end{align}
The desired lower bounds on $\mathcal{M}_n(\mathcal{F}_{\beta,L,r})$ and $\mathcal{M}_n(\delta,\mathcal{F}_{\beta,L,r})$ follow respectively from~\eqref{eq:holder-lower-bd-expectation}, and the maximum of the bounds in~\eqref{eq:holder-lower-bd-asssouad} and~\eqref{eq:holder-lower-bd-le-cam}.
\end{proof}

We conclude this section with a minimax lower bound over the larger class of densities with H\"older scores that are not necessarily decreasing. For $\beta \in [1,2]$ and $L > 0$, denote by $\tilde{\mathcal{F}}_{\beta,L}$ 
the set of all locally absolutely continuous densities $f_0$ on $\R$ whose score functions $\psi_0 = (\log f_0)'$ (defined Lebesgue almost everywhere when $\beta = 1$) satisfy
\[
|\psi_0(x) - \psi_0(y)| \leq L|x - y|^{\beta-1}
\]
for all $x,y \in \R$. For $r > 0$, define $\tilde{\mathcal{F}}_{\beta,L,r} := \{f_0 \in \tilde{\mathcal{F}}_{\beta,L}: i(f_0) \leq r\}$.

\begin{prop}
\label{prop:lower-bd-no-shape-c}
For $\beta \in [1,2]$, $L,r > 0$, $n \in \N$ and $\delta \in (0,1/4)$, we have
\begin{align*}
\mathcal{M}_n(\delta,\tilde{\mathcal{F}}_{\beta,L,r}) &\gtrsim_\beta r'\min\biggl\{\biggl(\frac{L^{2/\beta}}{r'} \cdot \frac{1}{n}\biggr)^{\frac{2(\beta - 1)}{2\beta + 1}} + \biggl(\frac{L^{2/\beta}}{r'} \cdot \frac{\log(1/\delta)}{n}\biggr)^{\frac{\beta - 1}{\beta}},\,1\biggr\}, \\
\mathcal{M}_n(\tilde{\mathcal{F}}_{\beta,L,r}) &\gtrsim_\beta r'\min\biggl\{\biggl(\frac{L^{2/\beta}}{r'} \cdot \frac{1}{n}\biggr)^{\frac{2(\beta-1)}{2\beta+1}},\,1\biggr\},
\end{align*}
where $r' := r \wedge L^{2/\beta}$.
\end{prop}


\begin{proof}
Suppose first that $\tilde{L} := 2L/5 \geq 1$. Let $a := \tilde{L}^{1/\beta}$ and define $\psi_0 \colon \R \to \R$ by
\[
\psi_0(x) := -\tilde{L}^{2/\beta} \int_0^x \Ind_{\{0 \leq |z| - a \leq \tilde{L}^{-1/\beta}\}} \laplaced z = -\tilde{L}^{1/\beta}\sgn(x)\min\bigl\{\tilde{L}^{1/\beta}(|x| - a)_+,1\bigr\}.
\]
Then for all $x,y \in \R$, we have
\begin{equation}
\label{eq:psi-0-holder-easy}
|\psi_0(x) - \psi_0(y)| \leq \tilde{L}^{2/\beta}(|x - y| \wedge 2\tilde{L}^{-1/\beta}) \leq 2^{2 - \beta}\tilde{L}|x - y|^{\beta - 1}.
\end{equation}
Given $h \in (0,\tilde{L}^{-1/\beta}/2]$, let
\[
\tau := \Ind_{(0,h/4]} - \Ind_{(h/4,3h/4]} + \Ind_{(3h/4,h]} = \Ind_{(0,h]} - 2\Ind_{(h/4,3h/4]},
\]
and as in the proof of Theorem~\ref{thm:holder-lower-bd}, define $\Delta,\varphi \colon \R \to \R$ by $\Delta(x) := \int_0^x \tau$ and $\varphi(x) := \int_0^x \Delta$ (see Figure~\ref{fig:tauDeltaPhi}). Then $\Delta(x) = \varphi(x) = 0$ for $x \in \R \setminus [0,h]$. Let $m := \floor{\tilde{L}^{-1/\beta}/(2h)} \geq 1$, and for $\zeta = (\zeta_1,\dotsc,\zeta_m) \in \{0,1\}^m$, define $\Delta_\zeta \colon \R \to \R$ by
\[
\Delta_\zeta(x) := \tilde{L}h^{\beta - 2}\sum_{j=1}^m \zeta_j \Delta\Bigl(x - a - \frac{\tilde{L}^{-1/\beta}}{2} - (j - 1)h\Bigr) =: -\Delta_\zeta(-x)
\]
for $x \geq 0$. Since $mh \leq \tilde{L}^{-1/\beta}/2$, we have $\Delta_\zeta(x) = 0$ whenever $|x| \leq \tilde{L}^{1/\beta}$ or $|x| \geq a + \tilde{L}^{-1/\beta}$. Moreover, $|\Delta_\zeta(x)| \leq \tilde{L}h^{\beta - 1}/4$ for all $x \in \R$, so
\begin{equation}
\label{eq:Delta-zeta-holder}
|\Delta_\zeta(x) - \Delta_\zeta(y)| \leq \bigl(\tilde{L}h^{\beta - 2}|x - y|\bigr) \wedge \frac{\tilde{L}h^{\beta - 1}}{2} \leq 2^{\beta - 2}\tilde{L}|x - y|^{\beta - 1}
\end{equation}
for all $x,y \in \R$. Now let 
\begin{equation*}
\psi_\zeta := \psi_0 + \Delta_\zeta. 
\end{equation*}
Then by~\eqref{eq:psi-0-holder-easy} and~\eqref{eq:Delta-zeta-holder},
\begin{equation}\label{eq:psi-zeta-holder}
|\psi_\zeta(x) - \psi_\zeta(y)| \leq (2^{2 - \beta} + 2^{\beta - 2})\tilde{L}|x - y|^{\beta - 1} \leq L|x - y|^{\beta - 1}
\end{equation}
for all $x,y \in \R$.  Moreover, $\psi_\zeta(x) = \psi_0(x) = 0$ when $|x| \leq a$ and $\psi_\zeta(x) = \psi_0(x) = -\tilde{L}^{1/\beta}\sgn(x)$ when $|x| \geq a + \tilde{L}^{-1/\beta}$. Since $h \in (0,\tilde{L}^{-1/\beta}/2]$, we have $\psi_\zeta(x) \leq (-\tilde{L}^{2/\beta} \cdot \frac{1}{2}\tilde{L}^{-1/\beta} + \frac{1}{4}\tilde{L}h^{\beta - 1})_+ = 0$ for all $x \geq 0$. 
Therefore, the functions $\phi_0 \colon \R \to (-\infty,0]$ and $\phi_\zeta \colon \R \to (-\infty,0]$ given by $\phi_0(x) := \int_0^x \psi_\zeta$ and $\phi_\zeta(x) := \int_0^x \psi_\zeta$ respectively are symmetric, and since $\int_0^x \Delta_\zeta \geq 0$ for all $x \geq 0$, we have 
\begin{align}
\label{eq:phi-bound-exp-no-shape-c}
-\tilde{L}^{1/\beta}(|x| - a)_+ \leq \phi_0(x) \leq \phi_\zeta(x) \leq -\tilde{L}^{1/\beta}(|x| - a - \tilde{L}^{-1/\beta})_+ \leq 0
\end{align}
for all $x \in \R$. Since $a = \tilde{L}^{1/\beta} \geq \tilde{L}^{-1/\beta}$, this means that $C_\zeta := \log\bigl(\int_{-\infty}^\infty e^{\phi_\zeta}\bigr)$ satisfies
\begin{align} 
2\tilde{L}^{1/\beta} < 2(a + \tilde{L}^{-1/\beta}) = \int_{-\infty}^\infty e^{-\tilde{L}^{1/\beta}(|x| - a)_+} \laplaced x \leq e^{C_\zeta} &\leq \int_{-\infty}^\infty e^{-\tilde{L}^{1/\beta}(|x| - a - \tilde{L}^{-1/\beta})_+} \laplaced x \notag \\ 
\label{eq:expCBound-no-shape-c}
&= 2(a + 2\tilde{L}^{-1/\beta}) \leq 6\tilde{L}^{1/\beta}.
\end{align}
Thus, $f_\zeta := e^{\phi_\zeta - C_\zeta}$ is a symmetric, absolutely continuous density, and for all $x \in [-(a + \tilde{L}^{-1/\beta}),a +\tilde{L}^{-1/\beta}]$, we have
\begin{align}
\label{eq:fZetaBound-no-shape-c}
\frac{1}{6e\tilde{L}^{1/\beta}} &\leq \exp\bigl(-\tilde{L}^{1/\beta}(|x| - a)_+ - C_\zeta\bigr) \leq f_\zeta(x) \leq \frac{1}{2\tilde{L}^{1/\beta}}.
\end{align}
Further, by~\eqref{eq:phi-bound-exp-no-shape-c} and~\eqref{eq:expCBound-no-shape-c}, 
\begin{align*}
i(f_\zeta) =  \int_{-\infty}^\infty \psi_\zeta^2 f_\zeta \leq \frac{2\tilde{L}^{2/\beta}}{e^{C_\zeta}} \int_a^\infty e^{\phi_\zeta}
&\leq \tilde{L}^{1/\beta} \int_a^\infty e^{-\tilde{L}^{1/\beta}(x - a - \tilde{L}^{-1/\beta})_+} \laplaced x = 2.
\end{align*}
Together with~\eqref{eq:psi-zeta-holder}, this implies that $f_\zeta \in \tilde{\mathcal{F}}_{\beta,L,2}$ for all $\zeta \in \{0,1\}^m$. 

We now proceed with calculations similar to those in the proof of Theorem~\ref{thm:holder-lower-bd}. For $\zeta = (\zeta_1,\dotsc,\zeta_m)$ and $\xi = (\xi_1,\dotsc,\xi_m) \in \{0,1\}^m$, let $S := \{j \in [m]: \zeta_j \neq \xi_j\}$. Define $I_j := \bigl[a + \frac{1}{2}\tilde{L}^{-1/\beta} + (j - 1)h, a + \frac{1}{2}\tilde{L}^{-1/\beta} + jh\bigr]$ for $j \in [m].$ Since $\Delta(x) = 0$ for $x \geq h$ and $\int_0^h \Delta = 0$, we have $\psi_\zeta(x) = \psi_\xi(x)$ and $\phi_\zeta(x) = \phi_\xi(x)$ for all $x \in [0,\infty) \setminus \bigcup_{j \in S} I_j$. For $j \in S$, we have 
\begin{align}
\label{eq:bulk-assouad-1-no-shape-c}
\int_{I_j} (\psi_\zeta - \psi_\xi)^2 &= (\tilde{L}h^{\beta - 2})^2 \int_0^h \Delta^2 = \frac{\tilde{L}^2 h^{2\beta-1}}{48}, \\
\int_{I_j} (\phi_\zeta - \phi_\xi)^2 &= (\tilde{L}h^{\beta - 2})^2 \int_0^h \varphi^2 = \frac{23\tilde{L}^2 h^{2\beta + 1}}{15360}. \notag
\end{align}
In view of~\eqref{eq:expCBound-no-shape-c},~\eqref{eq:fZetaBound-no-shape-c} and the facts that $h \leq \tilde{L}^{-1/\beta}/2 \leq 1/2$, $m \leq \tilde{L}^{-1/\beta}/(2h)$ and $|\phi_\zeta - \phi_\xi| \leq \tilde{L}^{1/\beta} \cdot \tilde{L}^{-1/\beta}$ = 1 on $\R$, virtually identical arguments to those in~\eqref{eq:C-xi-bd}--\eqref{eq:chi-squared-1} yield the bounds
\begin{align*}
|C_\zeta - C_\xi| &\leq \frac{|S|\tilde{L}^{1 - 1/\beta}h^{\beta + 1}}{32} \leq\frac{\tilde{L}^{1 - 2/\beta}h^\beta}{64} \leq \frac{1}{16}, \\
\int_{(0,\infty) \setminus \cup_{j \in S} I_j} (e^{\phi_\zeta - \phi_\xi + C_\xi - C_\zeta} - 1)^2\,f_\xi &\leq 16^2(e^{1/16} - 1)^2(C_\xi - C_\zeta)^2\int_{(0,\infty) \setminus \cup_{j \in S}I_j} f_\xi, \\
\int_{I_j} (e^{\phi_\zeta - \phi_\xi + C_\xi - C_\zeta} - 1)^2\,f_\xi &< 396\,\biggl(\frac{23\tilde{L}^2 h^{2\beta + 1}}{15360 \cdot 2\tilde{L}^{1/\beta}} + (C_\xi - C_\zeta)^2\int_{I_j}f_{\xi}\biggr)
\end{align*}
for $j \in S$. Thus, denoting by $P_\zeta,P_\xi$ the probability measures with densities $f_\zeta,f_\xi$ respectively, 
\begin{align}
\KL(P_\zeta,P_\xi) \leq \chi^2(P_\zeta,P_\xi) &= \int_{-\infty}^\infty \Bigl(\frac{f_\zeta}{f_\xi} - 1\Bigr)^2 f_\xi = 2\int_0^\infty (e^{\phi_\zeta - \phi_\xi + C_\xi - C_\zeta} - 1)^2\,f_\xi \notag \\
&< \frac{396}{2^8}\bigl(|S|\tilde{L}^{2 - 1/\beta}h^{2\beta + 1} + |S|^2\tilde{L}^{2 - 2/\beta}h^{2\beta + 2}\bigr) \notag \\
\label{eq:minimax-kl-no-shape-c}
&\leq \frac{396m\tilde{L}^{2 - 1/\beta}h^{2\beta + 1}}{2^8}\Bigl(1 + \frac{\tilde{L}^{-2/\beta}}{2}\Bigr) \leq \frac{1188m\tilde{L}^{2 - 1/\beta}h^{2\beta + 1}}{2^9},
\end{align}
where the final two inequalities hold because $m \leq \tilde{L}^{-1/\beta}/(2h)$ and $\tilde{L} \geq 1$. Now let 
\begin{align*}
h := \frac{1}{\tilde{L}^{1/\beta}}\biggl(\Bigl(\frac{\tilde{L}^{2/\beta}}{113n}\Bigr)^{1/(2\beta + 1)} \wedge \frac{1}{2}\biggr),
\qquad M := \ceil{e^{m/8}} \geq 2,
\end{align*}
where we recall that $m = \floor{\tilde{L}^{-1/\beta}/(2h)}$. By the Gilbert--Varshamov lemma~\citep[Exercise~8.10]{samworth24modern}, there exist distinct $\zeta^1,\dots,\zeta^M \in \{0,1\}^m$ such that $\sum_{j=1}^m \Ind_{\{\zeta_j^p \neq \zeta_j^q\}} > m/4$ for any distinct $p,q \in [M]$ 
By~\eqref{eq:minimax-kl-no-shape-c},
\begin{align*}
\frac{M^{-1}\sum_{j=1}^M \KL(P_{\zeta^j}^{\otimes n},P_{\zeta^1}^{\otimes n}) + \log(2 - M^{-1})}{\log M} &\leq \frac{1}{m/8} \cdot \frac{1188mn\tilde{L}^{2 - 1/\beta}h^{2\beta + 1}}{2^9} + \frac{\log(3/2)}{\log 2} < \frac{3}{4}.
\end{align*}
We conclude that if $L \geq 5/2$, then by~\eqref{eq:fZetaBound-no-shape-c}, Fano's lemma \citep[Corollary~8.12]{samworth24modern} and~\eqref{eq:bulk-assouad-1-no-shape-c},
\begin{align}
\mathcal{M}_n(\tilde{\mathcal{F}}_{\beta,L,2}) &\geq 
\inf_{\hat{\psi}_n \in \hat{\Psi}_n} \max_{\ell \in [M]}\,\E_{f_{\zeta^\ell}}\biggl(\sum_{j=1}^m \int_{(-I_j) \cup I_j} \frac{(\widehat{\psi}_n - \psi_{\zeta^\ell})^2}{6e\tilde{L}^{1/\beta}}\biggr) \geq \frac{2m}{4} \cdot \frac{\tilde{L}^2 h^{2\beta - 1}}{2^2 \times 6e\tilde{L}^{1/\beta} \times 48} \Bigl(1 - \frac{3}{4}\Bigr) \notag \\
\label{eq:holder-no-shape-lower-bd-1}
&\geq \frac{\tilde{L}^{2 - \frac{2}{\beta}}h^{2(\beta - 1)}}{2^6 \times 6e \times 48} 
\gtrsim \min\biggl\{\Bigl(\frac{\tilde{L}^{2/\beta}}{n}\Bigr)^{\frac{2(\beta - 1)}{2\beta + 1}},\,1\biggr\}.
\end{align}
Similarly, for $\delta \in (0,1/4)$, \citet[Lemma~7]{ma2025high} yields
\begin{align}
\mathcal{M}_n(\delta,\tilde{\mathcal{F}}_{\beta,L,2}) 
&\geq \inf_{\hat{\psi}_n \in \hat{\Psi}_n} \max_{\ell \in [M]}\,\mathrm{Quantile}_{f_{\zeta^\ell}}\biggl(1 - \delta,\,\sum_{j=1}^m \int_{(-I_j) \cup I_j} \frac{(\widehat{\psi}_n - \psi_{\zeta^\ell})^2}{6e\tilde{L}^{1/\beta}}\biggr) \notag \\
\label{eq:holder-no-shape-lower-bd-1a}
&\geq \frac{2m}{4} \cdot \frac{\tilde{L}^2 h^{2\beta-1}}{2^2 \times 6e\tilde{L}^{1/\beta} \times 48} \gtrsim \min\biggl\{\Bigl(\frac{\tilde{L}^{2/\beta}}{n}\Bigr)^{\frac{2(\beta - 1)}{2\beta + 1}},\,1\biggr\}.
\end{align}

Now suppose instead that we define
\[
h := \biggl\{\frac{2^{10}}{1188n\tilde{L}^{2(1 - \frac{1}{\beta})}}\log\Bigl(\frac{1}{4\delta(1 - \delta)}\Bigr)\biggr\}^{1/(2\beta)} \wedge \frac{1}{2\tilde{L}^{1/\beta}}.
\]
Fix $\zeta,\xi \in \{0,1\}^m$ with $S = [m]$. Then by~\eqref{eq:minimax-kl-no-shape-c},
\begin{align*}
\KL(P_\zeta^{\otimes n},P_\xi^{\otimes n}) < \frac{1188mn\tilde{L}^{2 - 1/\beta}h^{2\beta + 1}}{2^9} \leq \frac{1188n\tilde{L}^{2 - \frac{2}{\beta}}h^{2\beta}}{2^{10}} \leq \log\Bigl(\frac{1}{4\delta(1 - \delta)}\Bigr),
\end{align*}
so by~\eqref{eq:fZetaBound-no-shape-c} and \citet[Corollary~6]{ma2025high},
\begin{align}
\mathcal{M}_n(\delta,\tilde{\mathcal{F}}_{\beta,L,2}) &\geq \inf_{\hat{\psi}_n \in \hat{\Psi}_n}\max_{\theta \in \{\zeta,\xi\}}\mathrm{Quantile}_{f_\theta}\biggl(1 - \delta,\,\sum_{j=1}^m \int_{(-I_j) \cup I_j} \frac{(\widehat{\psi}_n - \psi_\theta)^2}{6e\tilde{L}^{1/\beta}}\biggr) \notag \\
\label{eq:holder-no-shape-lower-bd-1b}
&\geq 2m\,\frac{\tilde{L}^2 h^{2\beta-1}}{2^2 \times 6e\tilde{L}^{1/\beta} \times 48}
\gtrsim \min\biggl\{\biggl(\frac{\tilde{L}^{2/\beta}\log(1/\delta)}{n}\biggr)^{\frac{\beta - 1}{\beta}},\,1\biggr\}.
\end{align}

Finally, consider arbitrary $L,r > 0$. Let $s := \min\bigl\{r,2(2L/5)^{2/\beta}\bigr\}$. Then $L(s/2)^{-\beta/2} \geq 5/2$ and 
\[
\tilde{\mathcal{F}}_{\beta,L,r} \supseteq \tilde{\mathcal{F}}_{\beta,L,s} = \bigl\{(s/2)^{1/2}f_0\bigl((s/2)^{1/2}\cdot\bigr) : f_0 \in \tilde{\mathcal{F}}_{\beta,L(s/2)^{-\beta/2},2}\bigr\},
\]
so by~\eqref{eq:holder-no-shape-lower-bd-1},
\begin{align*}
\mathcal{M}_n(\tilde{\mathcal{F}}_{\beta,L,r}) &\geq \frac{s\mathcal{M}_n(\tilde{\mathcal{F}}_{\beta,L(s/2)^{-\beta/2},2})}{2}\gtrsim s\min\biggl\{\biggl(\frac{L^{2/\beta}}{s} \cdot \frac{1}{n}\biggr)^{\frac{2(\beta - 1)}{2\beta + 1}},\,1\biggr\}.
\end{align*}
Similarly, by~\eqref{eq:holder-no-shape-lower-bd-1a} and~\eqref{eq:holder-no-shape-lower-bd-1b},
\begin{align*}
\mathcal{M}_n(\delta,\tilde{\mathcal{F}}_{\beta,L,r}) &\geq \frac{s\mathcal{M}_n(\delta,\tilde{\mathcal{F}}_{\beta,L(s/2)^{-\beta/2},2})}{2} \\
&\gtrsim s\min\biggl\{\biggl(\frac{L^{2/\beta}}{s} \cdot \frac{1}{n}\biggr)^{\frac{2(\beta - 1)}{2\beta + 1}} + \biggl(\frac{L^{2/\beta}}{s} \cdot \frac{\log(1/\delta)}{n}\biggr)^{\frac{\beta - 1}{\beta}},\,1\biggr\},
\end{align*}
which completes the proof.
\end{proof}

\section{Proofs for Section~\ref{sec:multiscale}}
\label{sec:proofs-multiscale}

\begin{proof}[Proof of Lemma~\ref{lem:biasControl}]
By Fubini's theorem, $\int_{-\infty}^\infty f_h(x) \laplaced x = \int_{-\infty}^\infty \int_{-\infty}^\infty \kernelFunction[h](x - y) \laplaced x\,f_0(y) \laplaced y = 1$. Moreover, for every $x_0 \in \R$, we have $0 \leq f_h(x_0) = \int_{-\infty}^\infty K_h(x_0 - y)f_0(y) \laplaced y \leq 1/h < \infty$, so
\begin{align*}
f_h(x_0) 
= \int_{-\infty}^\infty \int_{-\infty}^{x_0} \kernelFunction[h]'(x - y) \laplaced x\,f_0(y) \laplaced y = \int_{-\infty}^{x_0} \int_{-\infty}^\infty \kernelFunction[h]'(x - y)f_0(y) \laplaced y \laplaced x.
\end{align*}
This shows that $f_h$ is an absolutely continuous function on $\R$. Since $|K_h'| \leq h^{-2}$ almost everywhere on $\R$, it follows from the dominated convergence theorem that
\[
(K_h' * f_0)(x') = \int_{-\infty}^\infty \kernelFunction[h]'(x' - y)f_0(y) \laplaced y \to (K_h' * f_0)(x_0)
\]
as $x' \to x_0$, so $K_h' * f_0$ is continuous. Hence, by the fundamental theorem of calculus,~$f_h$ is in fact continuously differentiable on $\R$ with derivative $f_h' = K_h' * f_0$. 

Moreover, $f_0$ is locally absolutely continuous by assumption, so $f_0(x) - f_0(0) = \int_0^x f_0'$ for all $x \in \R$. Since $\int_{-\infty}^\infty K_h' = 0$, it follows again by Fubini's theorem that
\begin{align*}
f_h'(x) &= \int_{-\infty}^\infty K_h'(x - y)\bigl(f_0(y) - f_0(0)\bigr) \laplaced y \\
&= \int_{-\infty}^\infty \int_{-\infty}^\infty \kernelFunction[h]'(x - y) f_0'(z)(\one_{\{0 \leq z \leq y\}} - \one_{\{y < z \leq 0\}}) \laplaced z \laplaced y \\
&= \int_{-\infty}^\infty \kernelFunction[h](x - z) f_0'(z)(\one_{\{z \geq 0\}} + \one_{\{z < 0\}}) \laplaced z = \int_{-\infty}^\infty \kernelFunction[h](x - z) f_0'(z) \laplaced z.
\end{align*}
Thus, $f_h' = K_h * f_0'$.

If in addition $f_0$ is log-concave, then $\psi_0 \colon \R \to [-\infty,\infty]$ is decreasing. Since $\kernelFunction[h]$ is supported on $[-h,h]$, it follows that for $x \in \R$ and $\omega \in \{-1,1\}$, we have
\begin{align*}
\omega f_h'(x) = \omega \int_{x-h}^{x+h} \kernelFunction[h](x - z)\psi_0(z)f_0(z)\laplaced z &\geq \omega \int_{x-h}^{x+h} \kernelFunction[h](x - z)\psi_0(x + \omega h)f_0(z)\laplaced z \\
&= \omega\psi_0(x + \omega h)f_h(x).
\end{align*}
This implies that $\omega\psi_h(x) \geq \omega\psi_0(x + \omega h)$ when $f_h(x) > 0$, i.e.~$[x - h, x + h] \cap \support \neq \emptyset$. Otherwise, either $x + h \leq \inf\support$, in which case $\omega\psi_0(x + \omega h) = \psi_h(x) = \infty$, or $x - h \geq \sup\support$ and $\omega\psi_0(x + \omega h) = \psi_h(x) = -\infty$. Therefore,~\eqref{eq:psi_h-psi_0} holds in all cases.
\end{proof}

The next two results apply to any Borel probability distribution  $\probDistribution$ on $\R$ and $X_1,\dotsc,X_n \iid \probDistribution$ with empirical distribution $\mathbb{P}_n$. Some of the remaining lemmas require $\probDistribution$ to have an absolutely continuous density, but the log-concavity assumption will only be needed again in the proofs of Propositions~\ref{prop:confidenceBandsAreValid} and~\ref{prop:scorePointwiseErr}.

First, we establish a simultaneous concentration bound for indicator functions of intervals.  It is convenient to define $\kullbackLeibler_+(p,q) := \one_{\{p > q\}}\,\kullbackLeibler(p,q)$ for $p,q \in [0,1]$, where $\kullbackLeibler$ is as in~\eqref{eq:kl-bernoulli}. For a related result, see \citet[Proposition~1]{ryter2024tails}.

\begin{lemma}
\label{lem:uniformOverIntervals} 
For any $n \in \N$ and any $\delta \in (0,1)$, 
we have $\Prob(\goodEvent) \geq 1 - \delta$. 
\end{lemma}

\begin{proof} 
For each interval $A \subseteq \R$, define $\mathbb{D}_n(A) := (n/2)\,\empiricalProb(A) - (n/2 - 1)\probDistribution(A)$. If $\mathbb{D}_n(A) \in [0,1]$ for some interval $A$, then~\eqref{eq:goodEvent} holds with the infimal value of zero attained at $t = \mathbb{D}_n(A)$. In particular, for $n \leq 2$, the claim~\eqref{eq:goodEvent} holds with probability 1 for all intervals $A \subseteq \R$. Therefore, we suppose henceforth that $n \geq 3$. For $x_0,x_1 \in \R$, define
\begin{align*}
\mathcal{I}(x_0,x_1) := 
\begin{cases} [x_0,x_1] &\text{ if }x_0 \leq x_1\\
\R \setminus (x_1,x_0) & \text{ if }x_1 < x_0.
\end{cases}
\end{align*}
Letting $\empiricalProb[n-2](B) := (n-2)^{-1}\,\sum_{i=1}^{n-2}\one_{\{X_i \in B\}}$ for Borel sets $B \subseteq \R$, we apply Lemma~\ref{lem:kl-binomial} to the
$n-2$ conditionally independent and identically distributed Bernoulli random variables $(\one_{\{X_i \in \mathcal{I}(X_{n-1},X_n)\}})_{i \in [n-2]}$ to obtain
\begin{align*}
\Prob\Bigl\{\kullbackLeibler_+\Bigl(\empiricalProb[n-2]\{\mathcal{I}(X_
{n-1},X_n)\},\probDistribution\{\mathcal{I}(X_
{n-1},X_n)\}\Bigr) \geq \frac{\log(n^2/\delta)}{n-2} \Bigm| X_{n-1},X_n\Bigr\} \leq \frac{\delta}{n^2}.   
\end{align*}
Moreover, by Lemma~\ref{lem:kl-joint-convexity}, 
\begin{align*}
&\kullbackLeibler_+\biggl(\empiricalProb[n]\{\mathcal{I}(X_
{n-1},X_n)\},\Bigl(1 - \frac{2}{n}\Bigr)\probDistribution\{\mathcal{I}(X_
{n-1},X_n)\} + \frac{2}{n}\biggr) \\
&\quad=\kullbackLeibler_+\biggl(\Bigl(1 - \frac{2}{n}\Bigr)\empiricalProb[n-2]\{\mathcal{I}(X_
{n-1},X_n)\} + \frac{2}{n},\Bigl(1 - \frac{2}{n}\Bigr)\probDistribution\{\mathcal{I}(X_
{n-1},X_n)\} + \frac{2}{n}\biggr) \\
&\quad \leq \frac{n-2}{n}\,\kullbackLeibler_+\Bigl(\empiricalProb[n-2]\,\{\mathcal{I}(X_
{n-1},X_n)\},\probDistribution\{\mathcal{I}(X_
{n-1},X_n)\}\Bigr).
\end{align*}
Hence, by the law of total expectation, 
\begin{align*}
\Prob\biggl\{\kullbackLeibler_+\biggl(\empiricalProb\{\mathcal{I}(X_
{n-1},X_n)\},\Bigl(1 - \frac{2}{n}\Bigr)\probDistribution\{\mathcal{I}(X_
{n-1},X_n)\} + \frac{2}{n}\biggr) \geq \frac{\log(n^2/\delta)}{n}\biggr\} \leq \frac{\delta}{n^2},
\end{align*}
so by symmetry and a union bound, the event
\begin{align*}
\mathcal{E}_{n,\delta} := \bigcap_{i=1}^n \bigcap_{j=1}^n\biggl\{\kullbackLeibler_+\biggl(\empiricalProb\{\mathcal{I}(X_i,X_j)\},\Bigl(1 - \frac{2}{n}\Bigr)\probDistribution\{\mathcal{I}(X_i,X_j)\} + \frac{2}{n}\biggr) < \frac{\log(n^2/\delta)}{n}\biggr\}
\end{align*}
has probability at least $1 - \delta$. 

Now suppose that $\mathcal{E}_{n,\delta}$ holds and let $A \subseteq \R$ be an interval. As noted, if $\mathbb{D}_n(A) \in [0,1]$, then~\eqref{eq:goodEvent} follows. Now suppose that $\mathbb{D}_n(A) > 1$. Then since $\empiricalProb(A) > 2/n$, we may choose a random tuple $(i_A,j_A) \in [n]^2$ such that $i_A < j_A$, with $\mathcal{I}(X_{i_A},X_{j_A})\subseteq A$ and $\empiricalProb(A) = \empiricalProb\{\mathcal{I}(X_{i_A},X_{j_A})\}$, so that on the event $\mathcal{E}_{n,\delta}$ we have
\begin{align*}
\inf_{t \in [0,1]}\,& \kullbackLeibler\biggl(\empiricalProb(A),\Bigl(1 - \frac{2}{n}\Bigr)\,\probDistribution(A) + \frac{2t}{n}\biggr) \\
&= \kullbackLeibler\biggl(\empiricalProb(A),\Bigl(1 - \frac{2}{n}\Bigr)\probDistribution(A) + \frac{2}{n}\biggr) \\
&= \kullbackLeibler_+\biggl(\empiricalProb\{\mathcal{I}(X_{i_A},X_{j_A})\},\Bigl(1 - \frac{2}{n}\Bigr)\probDistribution(A) + \frac{2}{n}\biggr) \\
& \leq \kullbackLeibler_+\biggl(\empiricalProb\{\mathcal{I}(X_{i_A},X_{j_A})\},\Bigl(1 - \frac{2}{n}\Bigr)\probDistribution\{\mathcal{I}(X_{i_A},X_{j_A})\} + \frac{2}{n}\biggr) < \frac{\log(n^2/\delta)}{n},
\end{align*}
where we have used the fact that $q \mapsto \kullbackLeibler_+(p,q)$ is decreasing on $[0,1]$ for each $p \in [0,1]$. Therefore,~\eqref{eq:goodEvent} holds when $\mathbb{D}_n(A) > 1$. Finally, if $\mathbb{D}_n(A) < 0$, then we may proceed similarly with $\R \setminus A$ in place of $A$. Indeed, since $\empiricalProb(A) < 1$, we may choose a random tuple $(i_A^\circ,j_A^\circ) \in [n]^2$ with $i_A^\circ > j_A^\circ$ such that $\mathcal{I}(X_{i_A^\circ},X_{j_A^\circ}) \subseteq \R \setminus A$ and $\empiricalProb\{\mathcal{I}(X_{i_A^\circ},X_{j_A^\circ})\} = \empiricalProb(\R \setminus A)$. Then on the event $\mathcal{E}_{n,\delta}$, we have
\begin{align*}
\inf_{t \in [0,1]}\,& \kullbackLeibler\biggl(\empiricalProb(A),\Bigl(1 - \frac{2}{n}\Bigr)\,\probDistribution(A) + \frac{2t}{n}\biggr)\\
&= \kullbackLeibler\biggl(\empiricalProb(A),\Bigl(1 - \frac{2}{n}\Bigr)\probDistribution(A)\biggr)\\
&=
\kullbackLeibler_+\biggl(1-\empiricalProb(A),\Bigl(1 - \frac{2}{n}\Bigr)\{1 - \probDistribution(A)\} + \frac{2}{n}\biggr)\\
& \leq
\kullbackLeibler_+\biggl(\empiricalProb\{\mathcal{I}(X_{i_A^\circ},X_{j_A^\circ})\},\Bigl(1 - \frac{2}{n}\Bigr)\probDistribution\{\mathcal{I}(X_{i_A^\circ},X_{j_A^\circ})\} + \frac{2}{n}\biggr)<\frac{\log(n^2/\delta)}{n},
\end{align*}
so~\eqref{eq:goodEvent} again holds. Thus, $\mathcal{E}_{n,\delta}\subseteq \goodEvent$, so $\Prob(\goodEvent) \geq \Prob(\mathcal{E}_{n,\delta}) \geq 1 - \delta$.
\end{proof}

Recall that $\varphi:\mathbb{R} \to \mathbb{R}$ is \emph{unimodal} if there exists $x_0 \in \mathbb{R}$ such that $\varphi(x) \leq \varphi(x')$ for $x \leq x' \leq x_0$ and $\varphi(x) \geq \varphi(x')$ for $x_0 \leq x \leq x'$.  We now extend Lemma~\ref{lem:uniformOverIntervals} from the class of indicator functions of intervals to the class $\functionClassUnimodal$ of all unimodal functions $\varphi: \R \to [0,1]$. For $n \geq 3$ and $\omega \in \{-1,1\}$, let
\begin{align*}
\tilde{P}^{\omega}(\varphi) &:=  \Bigl(1 - \frac{2}{n}\Bigr){\probDistribution}(\varphi) + \frac{1+\omega}{n}\\
\tilde{\mathbb{J}}_{n,\delta,\omega}(\varphi) &:=  \frac{2(1 -  2\epsilonByNDeltaMod)\tilde{P}^{\omega}(\varphi) + 3\epsilonByNDeltaMod + 3\omega\sqrt{4\epsilonByNDeltaMod \tilde{P}^{1}(\varphi)\{1- \tilde{P}^{-1}(\varphi)\}+ \epsilonByNDeltaMod^2}}{2(1+\epsilonByNDeltaMod)}\\
\hat{\mathbb{J}}_{n,\delta,\omega}(\varphi) &:=  \frac{(1 - 2\epsilonByNDeltaMod)\empiricalProb(\varphi)+ 3\epsilonByNDeltaMod + 3\omega\sqrt{\epsilonByNDeltaMod\empiricalProb(\varphi)\{1 - \empiricalProb(\varphi)\} + \epsilonByNDeltaMod^2}}{(1 + 4\epsilonByNDeltaMod)(1 - 2/n)} + \frac{\omega-1}{n-2}.
\end{align*}

\begin{lemma}
\label{lem:uniformBoundedFunctions} 
Suppose that $n \in \mathbb{N}$, $\delta \in (0,1)$ and the event $\goodEvent$ holds. Then
\begin{align*}
\inf_{t \in [0,1]} \kullbackLeibler\biggl(\empiricalProb(\varphi),\Bigl(1 - \frac{2}{n}\Bigr)\,\probDistribution(\varphi) + \frac{2t}{n}\biggr) < \frac{\log(n^2/\delta)}{n}
\end{align*}
for all $\varphi \in \functionClassUnimodal$. Further, if $n \geq 3$, then for all $\varphi \in \functionClassUnimodal$ we have $\empiricalProb(\varphi) \leq 2\probDistribution(\varphi) + 21\epsilonByNDeltaMod$, $\empiricalProb(\varphi) \in [\tilde{\mathbb{J}}_{n,\delta,-1}(\varphi),\tilde{\mathbb{J}}_{n,\delta,1}(\varphi)]$ and $\probDistribution(\varphi) \in [\hat{\mathbb{J}}_{n,\delta,-1}(\varphi),\hat{\mathbb{J}}_{n,\delta,1}(\varphi)]$. Moreover,
\begin{align*}
\max_{\omega \in \{-1,1\}} \bigl|\hat{\mathbb{J}}_{n,\delta,\omega}(\varphi) - \probDistribution(\varphi)\bigr| \leq 
26\sqrt{\epsilonByNDeltaMod\probDistribution(\varphi)} + 112\epsilonByNDeltaMod
\end{align*}
for all $\varphi \in \functionClassUnimodal$.
\end{lemma}

\begin{proof} 
Let $\varphi \in \functionClassUnimodal$. For each $\gamma \in [0,1]$, the set $A_\gamma := \{x \in \R : \varphi(x) \geq \gamma\}$ is an interval, and moreover for $x \in \mathbb{R}$, we have
\[
\varphi(x) = \int_0^{\varphi(x)}\laplaced\gamma = \int_0^1 \one_{A_\gamma}(x)\laplaced\gamma.
\]
In addition, let 
\[
\tau_\gamma(\varphi) := \argmin_{t \in [0,1]} \kullbackLeibler\biggl(\empiricalProb(A_\gamma),\Bigl(1 - \frac{2}{n}\Bigr)\,\probDistribution(A_\gamma) + \frac{2t}{n}\biggr).
\]
Then by Fubini's theorem,
\begin{align*}
\inf_{t \in [0,1]} \kullbackLeibler\biggl(\empiricalProb(\varphi) &,\Bigl(1 - \frac{2}{n}\Bigr)\,\probDistribution(\varphi) + \frac{2t}{n}\biggr) \\
&\leq \kullbackLeibler\biggl(\empiricalProb\Bigl(\int_0^1\one_{A_\gamma}\laplaced \gamma\Bigr),\Bigl(1 - \frac{2}{n}\Bigr)\,\probDistribution\Bigl(\int_0^1\one_{A_\gamma}\laplaced \gamma\Bigr) + \frac{2}{n}\Bigl(\int_0^1\tau_\gamma(\varphi)\laplaced\gamma\Bigr)\biggr)\\
&=\kullbackLeibler\biggl(\int_0^1\empiricalProb(A_\gamma)\laplaced \gamma,\int_0^1\Bigl\{\Bigl(1 - \frac{2}{n}\Bigr)\,\probDistribution(A_\gamma) + \frac{2\tau_\gamma(\varphi)}{n}\Bigr\} \laplaced \gamma\biggr)\\
&\leq \int_0^1\kullbackLeibler\biggl(\empiricalProb(A_\gamma),\Bigl(1 - \frac{2}{n}\Bigr)\,\probDistribution(A_\gamma) + \frac{2\tau_\gamma(\varphi)}{n}\biggr)\laplaced \gamma< \frac{\log(n^2/\delta)}{n} = \frac{9\epsilonByNDeltaMod}{2},
\end{align*}
where the penultimate inequality follows from the joint convexity of the Kullback--Leibler divergence (Lemma~\ref{lem:klJointConvex}) and the final inequality follows from the definition of $\goodEvent$ in~\eqref{eq:goodEvent}.

For the second part of the lemma, 
taking 
\[
p = \empiricalProb(\varphi) \quad \text{and} \quad q = \Bigl(1 - \frac{2}{n}\Bigr)\probDistribution(\varphi) + \frac{2}{n}\int_0^1\tau_\gamma(\varphi)\laplaced \gamma \in [\tilde{P}^{-1}(\varphi),\tilde{P}^1(\varphi)]
\]
in the second part of Lemma~\ref{lem:klLowerBound} yields 
$\empiricalProb(\varphi) \in [\tilde{\mathbb{J}}_{n,\delta,-1}(\varphi),\tilde{\mathbb{J}}_{n,\delta,1}(\varphi)]$. Since $n \geq 3$ and $\delta \in (0,1)$, we have $\log(n^2/\delta) > 1$, so $9\epsilonByNDeltaMod/2 > 1/n$ and hence
\begin{align*}
\empiricalProb(\varphi) \leq \tilde{\mathbb{J}}_{n,\delta,1}(\varphi) &\leq \probDistribution(\varphi) + \frac{2}{n} + \frac{3\epsilonByNDeltaMod}{2} + \frac{3}{2}\biggl(4\epsilonByNDeltaMod\Bigl(\probDistribution(\varphi) + \frac{2}{n}\Bigr) + \epsilonByNDeltaMod^2\biggr)^{1/2} \\
&\leq \probDistribution(\varphi) + \frac{21\epsilonByNDeltaMod}{2} + 3\epsilonByNDeltaMod^{1/2}\Bigl(\probDistribution(\varphi) + \frac{21\epsilonByNDeltaMod}{2}\Bigr)^{1/2} \leq 2\probDistribution(\varphi) + 21\epsilonByNDeltaMod.
\end{align*}
By instead applying the third part of Lemma~\ref{lem:klLowerBound} with the same values of $p$ and $q$, we obtain
$\probDistribution(\varphi) \in [\hat{\mathbb{J}}_{n,\delta,-1}(\varphi),\hat{\mathbb{J}}_{n,\delta,1}(\varphi)]$. We have
\begin{align*}
\max_{\omega \in \{-1,1\}} \bigl|\hat{\mathbb{J}}_{n,\delta,\omega}(\varphi) - \probDistribution(\varphi)\bigr| &\leq \bigl|\hat{\mathbb{J}}_{n,\delta,1}(\varphi) - \hat{\mathbb{J}}_{n,\delta,-1}(\varphi)\bigr| \leq \frac{2n}{n - 2}\Biggl(\frac{3\sqrt{\empiricalProb(\varphi)\epsilonByNDeltaMod + \epsilonByNDeltaMod^2}}{1 + 4\epsilonByNDeltaMod} + \frac{1}{n}\Biggr) \\
&\leq 18\epsilonByNDeltaMod^{1/2}\bigl(2\probDistribution(\varphi) + 22\epsilonByNDeltaMod\bigr)^{1/2} + 27\epsilonByNDeltaMod \leq 26\sqrt{\epsilonByNDeltaMod\probDistribution(\varphi)} + 112\epsilonByNDeltaMod,
\end{align*}
as required.
\end{proof}

We now establish the simultaneous validity of the confidence intervals $\bigl[\scoreConfidenceBoundByBandwidthCentre[-1](z),\scoreConfidenceBoundByBandwidthCentre[1](z)\bigr]$ for $\psi_h(z)$, across all $z \in \R$ and bandwidths $h > 0$. When $\psi_0$ is decreasing, the multiscale aggregation in~\eqref{eq:scoreEstimateShapeConstrainedInfimum} is designed to exploit the relationship between $\psi_h$ and $\psi_0$ in Lemma~\ref{lem:biasControl}, thereby guaranteeing the coverage of the confidence band $\bigl[\scoreConfidenceBoundShapeConstrained[-1](\cdot), \scoreConfidenceBoundShapeConstrained[1](\cdot)\bigr]$ for $\psi_0(\cdot)$ in Proposition~\ref{prop:confidenceBandsAreValid}.

\begin{lemma}
\label{lem:confidenceBandStatisticalErrorControl}
Suppose that $n \geq 3$, $\delta \in (0,1)$ and $\probDistribution$ has an absolutely continuous density $f_0$. Then on the event $\goodEvent$, we have
\begin{align*}
0 \leq \omega\biggl\{\frac{\scoreConfidenceBoundByBandwidthCentreDenominator(z)}{(1 + 4\epsilonByNDeltaMod)(1 - 2/n)}-h^2\, f_h(z)\biggr\} \quad\text{and}\quad 0 \leq 
\omega \biggl\{\frac{\avgDerivSign{n}{h}(z)\,\scoreConfidenceBoundByBandwidthCentreMagnitudeNumerator[\omega \avgDerivSign{n}{h}(z)](z)}{(1 + 4\epsilonByNDeltaMod)(1 - 2/n)}-h^2f_h'(z)\biggr\}
\end{align*}
for all $z \in \mathbb{R}$, $h > 0$ and $\omega \in \{-1,1\}$. It follows that
\begin{align*}
\goodEvent\subseteq  \bigcap_{z \in \R}\bigcap_{h > 0}\bigl\{\scoreConfidenceBoundByBandwidthCentre[-1](z) \leq \psi_h(z) \leq \scoreConfidenceBoundByBandwidthCentre[1](z) \bigr\}.
\end{align*}
\end{lemma}

\begin{proof}
Fix $z \in \R$ and $h > 0$.  Recall that $K$ is absolutely continuous on $\R$ with a weak derivative~$K'$ given by $K'(x) := -\sgn(x)\mathbbm{1}_{\{|x| < 1\}}$.  Writing $(K')_+$ and $(K')_-$ for the positive and negative parts of $K'$, observe that the functions $\varphi_0,\varphi_1^\pm : \R \to \R$ defined by
\begin{align*}
\varphi_0(x) := \kernelFunction\Bigl(\frac{z - x}{h}\Bigr) = h\,\kernelFunction[h](z - x), \qquad
\varphi_1^\pm(x) := (\kernelFunction')_\pm\Bigl(\frac{z - x}{h}\Bigr) =  \mathbbm{1}_{\{x \in (z,z\pm h)\}}
\end{align*}
are all unimodal, take values in $[0,1]$ and are supported within $[z - h, z + h]$. Letting $\varphi_1 := \varphi_1^+ - \varphi_1^-$, so that $\varphi_1(x) = \sgn(x-z)\mathbbm{1}_{\{|x-z|<h\}}$, we deduce from Lemma~\ref{lem:biasControl} that
\[
\empiricalProb(\varphi_0) = h\hat{f}_{n,h}(z), \quad \probDistribution(\varphi_0) = h f_h(z), \quad \empiricalProb(\varphi_1) = h^2\hat{f}_{n,h}'(z), \quad \probDistribution(\varphi_1) = 
h^2 f_h'(z)
\]
and $\empiricalProb(\varphi_0) \vee \empiricalProb(\varphi_1^-) \vee \empiricalProb(\varphi_1^+) \leq \empiricalProb([z - h, z + h]) =: \hat{p}_{z,h}$. Henceforth, we work on the event~$\goodEvent$. By the second part of Lemma~\ref{lem:uniformBoundedFunctions} applied to $\varphi \in \{\varphi_0,\varphi_1^+,\varphi_1^-\}$ and $\omega \in \{-1,1\}$, we have
\begin{align*}
\omega\probDistribution(\varphi) &\leq \omega \cdot \frac{(1 - 2\epsilonByNDeltaMod)\empiricalProb(\varphi) + 3\epsilonByNDeltaMod + 3\omega\sqrt{\epsilonByNDeltaMod\, \bigl(\hat{p}_{z,h} + \epsilonByNDeltaMod \bigr)} - (1 + 4\epsilonByNDeltaMod)(1-\omega)/n}{(1 + 4\epsilonByNDeltaMod)(1 - 2/n)} \\
&=  \omega \cdot \frac{(1 - 2\epsilonByNDeltaMod)\empiricalProb(\varphi) +  \omega\scoreConfidenceBoundByBandwidthWidth(z)+(3-4/n)\epsilonByNDeltaMod-1/n}{(1 + 4\epsilonByNDeltaMod)(1 - 2/n)}.
\end{align*}
Hence, taking $\varphi=\varphi_0$, it follows that $\omega\bigl\{h^2\,f_h(z)\,(1 + 4\epsilonByNDeltaMod)(1 - 2/n) - \scoreConfidenceBoundByBandwidthCentreDenominator(z)\bigr\} \leq 0$ for $\omega \in \{-1,1\}$.  Moreover, writing $\avgDerivSignNoArg := \avgDerivSign{n}{h}(z)$, we have
\begin{align*}
\omega h^2f'_h(z) &= \omega\, \probDistribution(\varphi_1) = \omega\{\probDistribution(\varphi_1^+) - \probDistribution(\varphi_1^-)\} \notag \\
& \leq \omega \cdot \frac{(1 - 2\epsilonByNDeltaMod)\{\empiricalProb(\varphi_1^+) - \empiricalProb(\varphi_1^-)\} +  2\omega\scoreConfidenceBoundByBandwidthWidth(z)}{(1 + 4\epsilonByNDeltaMod)(1 - 2/n)} \\
&= \omega\avgDerivSignNoArg \cdot \frac{(1 - 2\epsilonByNDeltaMod)h^2\,|\hat{f}_{n,h}'(z)|\, +  2\omega\avgDerivSignNoArg\scoreConfidenceBoundByBandwidthWidth(z)}{(1 + 4\epsilonByNDeltaMod)(1 - 2/n)} = \frac{\omega\avgDerivSignNoArg\,\scoreConfidenceBoundByBandwidthCentreMagnitudeNumerator[\omega \avgDerivSignNoArg](z)}{(1 + 4\epsilonByNDeltaMod)(1 - 2/n)}. \notag
\end{align*}
If $\scoreConfidenceBoundByBandwidthCentreMagnitudeNumerator[-1](z) > 0$,  then $\scoreConfidenceBoundByBandwidthCentreMagnitudeNumerator[\omega \avgDerivSignNoArg](z) > 0$. Also, if $\scoreConfidenceBoundByBandwidthCentreDenominator[-1](z) > 0$, then $\hat{f}_{n,h}(z) > 0$ and hence there must exist $X_i \in [z - h, z + h]$ for some $i \in [n]$, so $f_h(z) > 0$. Thus, when $\scoreConfidenceBoundByBandwidthCentreMagnitudeNumerator[-1](z) > 0$ and $\scoreConfidenceBoundByBandwidthCentreDenominator[-1](z) > 0$, we have
\begin{align*}
\omega\bigl\{\psi_h(z) - \scoreConfidenceBoundByBandwidthCentre(z)\bigr\}
&= \frac{\omega f'_h(z)}{f_h(z)} - \frac{\omega\avgDerivSignNoArg\,\scoreConfidenceBoundByBandwidthCentreMagnitudeNumerator[\omega \avgDerivSignNoArg](z)}{\scoreConfidenceBoundByBandwidthCentreDenominator[-\omega \avgDerivSignNoArg](z)}\\
& \leq \frac{\omega h^2 f'_h(z)}{h^2 f_h(z)}-\frac{\omega\avgDerivSignNoArg\,\scoreConfidenceBoundByBandwidthCentreMagnitudeNumerator[\omega \avgDerivSignNoArg](z)}{h^2\, f_h(z)\,(1 + 4\epsilonByNDeltaMod)(1 - 2/n)} \leq 0,
\end{align*}
since $n>2$. Otherwise, if $\scoreConfidenceBoundByBandwidthCentreMagnitudeNumerator[-1](z) \leq 0$, then since $\scoreConfidenceBoundByBandwidthCentreMagnitudeNumerator[1](z) > 0$ we have $\omega\avgDerivSignNoArg\,\scoreConfidenceBoundByBandwidthCentreMagnitudeNumerator[\omega \avgDerivSignNoArg](z) \geq 0$. Hence, if in addition $\scoreConfidenceBoundByBandwidthCentreDenominator[-1](z) > 0$, then
\begin{align*}
\omega\bigl\{\psi_h(z) - \scoreConfidenceBoundByBandwidthCentre(z)\bigr\}
&= \frac{\omega\,f'_h(z)}{f_h(z)}-\frac{\omega\avgDerivSignNoArg\,\scoreConfidenceBoundByBandwidthCentreMagnitudeNumerator[\omega \avgDerivSignNoArg](z)}{\scoreConfidenceBoundByBandwidthCentreDenominator[-1](z)}\\
& \leq \frac{\omega\,h^2\,f'_h(z)}{h^2\, f_h(z)}-\frac{\omega\avgDerivSignNoArg\,\scoreConfidenceBoundByBandwidthCentreMagnitudeNumerator[\omega \avgDerivSignNoArg](z)}{h^2\, f_h(z)\,(1 + 4\epsilonByNDeltaMod)(1 - 2/n)} \leq 0.
\end{align*}
Finally, if $\scoreConfidenceBoundByBandwidthCentreDenominator[-1](z) \leq 0$, then $\scoreConfidenceBoundByBandwidthCentre(z) = \omega \times \infty$ so again $\omega\bigl\{\psi_h(z) - \scoreConfidenceBoundByBandwidthCentre(z)\bigr\} \leq 0$, which completes the proof of the lemma.
\end{proof}

\begin{proof}[Proof of Proposition~\ref{prop:confidenceBandsAreValid}]
The first assertion follows from Lemma~\ref{lem:uniformOverIntervals}. Next, on the event $\goodEvent$, Lemma~\ref{lem:confidenceBandStatisticalErrorControl} ensures that $0 \leq \omega\bigl\{\scoreConfidenceBoundByBandwidthCentre[\omega](z) - \psi_h(z)\bigr\}$ for all $z \in \R$, $\omega \in \{-1,1\}$ and $h > 0$. When $f_0 \in \mathcal{F}$, this implies that
\begin{align*}
\omega \scoreConfidenceBoundShapeConstrained[\omega](x) = \inf\bigl\{\omega \scoreConfidenceBoundByBandwidthCentre[\omega](z) : \omega(x - z) \geq h > 0\bigr\} &\geq \inf\{\omega\psi_h(z) : \omega(x - z) \geq h > 0\} \\
&\geq \inf\{\omega \psi_0(z + \omega h) : \omega(x - z) \geq h > 0\} \\
&\geq \omega \psi_0(x)
\end{align*}
for all $x \in \R$ and $\omega \in \{-1,1\}$, where the second inequality follows from the log-concavity of $f_0$ and Lemma~\ref{lem:biasControl}. Therefore,
\begin{align}
\label{eq:confidence-interval-valid}
\goodEvent \subseteq \bigcap_{x \in \R} \bigl\{\scoreConfidenceBoundShapeConstrained[-1](x) \leq \psi_0(x) \leq \scoreConfidenceBoundShapeConstrained[1](x) \bigr\} = \bigcap_{x \in \R} \bigl\{\psi_0(x) \in \hat{\mathcal{I}}_{n,\delta}(x)\bigr\},
\end{align}
as required.
\end{proof}

Next, on $\goodEvent$, we seek to bound the widths of the confidence intervals $\bigl[\scoreConfidenceBoundByBandwidthCentre[-1](z),\scoreConfidenceBoundByBandwidthCentre[1](z)\bigr]$ for $\psi_h(z)$, so as to obtain via~\eqref{eq:confidence-interval-valid} the pointwise bounds in Proposition~\ref{prop:scorePointwiseErr} on the error of the final multiscale estimator $\scoreEstimateShapeConstrained$. To this end, in Lemma~\ref{lem:scoreErrNumDenom} below, we control the widths of the confidence intervals $\hat{\mathsf{M}}_{n,\delta,h}(z),\hat{\mathsf{D}}_{n,\delta,h}(z)$ for $h^2 f_h'(z), h^2f_h(z)$ respectively, from which the above interval for $\psi_h(z) = f_h'(z)/f_h(z)$ is derived. In particular,~\eqref{eq:scoreErrDenom} ultimately bounds the multiplicative error of the lower estimate for $f_h(z)$; see~\eqref{eq:denom-relative-lower-bd} and its application in~\eqref{eq:scoreErr}.

\begin{lemma}
\label{lem:scoreErrPrelim}
Suppose that $\goodEvent$ holds with $n \geq 3$ and $\delta \in (0,1)$. Then for any $z \in \mathbb{R}$ and $h > 0$, we have
\[
\frac{\empiricalProb([z - h, z + h]) + \epsilonByNDeltaMod}{(1 + 4\epsilonByNDeltaMod)^2(1 - 2/n)^2} \leq 129\bigl\{\probDistribution([z - h, z + h]) \vee \epsilonByNDeltaMod\bigr\}.
\]
\end{lemma}

\begin{proof}
For $z \in \mathbb{R}$ and $h > 0$, define $\varphi_2:\mathbb{R}\to\{0,1\}$ by $\varphi_2 := \mathbbm{1}_{[z - h, z + h]}$. Since $\varphi_2 \in \functionClassUnimodal$, we have $\empiricalProb(\varphi_{2}) \leq 2\probDistribution(\varphi_2)+21\epsilonByNDeltaMod$ by Lemma~\ref{lem:uniformBoundedFunctions}.  Thus, using the fact that $(1 + 4\epsilonByNDeltaMod)(1 - 2/n) > \{1+8/(9n)\}(1 - 2/n) \geq 35/81$ when $n \geq 3$ and $\delta \in (0,1)$, we obtain

\[
\frac{\empiricalProb(\varphi_2) + \epsilonByNDeltaMod}{(1 + 4\epsilonByNDeltaMod)^2(1 - 2/n)^2} \leq 11\probDistribution(\varphi_{2})+118\epsilonByNDeltaMod\leq 129\bigl\{\probDistribution(\varphi_{2}) \vee \epsilonByNDeltaMod\bigr\}.
\qedhere
\]
\end{proof}

\begin{lemma}
\label{lem:scoreErrNumDenom}
Suppose that $\goodEvent$ holds with $n \geq 3$ and $\delta \in (0,1)$. If $\probDistribution$ has an absolutely continuous density $f_0$, then for any $z \in \mathbb{R}$, $h > 0$ and $\omega \in \{-1,1\}$, we have
\begin{align}
\label{eq:scoreErrDenom}
0 &\leq \omega \biggl\{\frac{\scoreConfidenceBoundByBandwidthCentreDenominator(z)}{(1 + 4\epsilonByNDeltaMod)(1 - 2/n)}-h^2 f_h(z)\biggr\}
\leq \COneZH \cdot h \epsilonByNDeltaMod^{1/2}\bigl\{hf_h(z)\vee \epsilonByNDeltaMod\bigr\}^{1/2},
\\
\label{eq:scoreErrNum}
0 &\leq \omega \biggl\{\frac{\avgDerivSign{n}{h}(z) \scoreConfidenceBoundByBandwidthCentreMagnitudeNumerator[\omega\avgDerivSign{n}{h}(z)](z)}{(1 + 4\epsilonByNDeltaMod)(1 - 2/n)} - h^2 f_h'(z)\biggr\}
\leq \CTwoZH \cdot \epsilonByNDeltaMod^{1/2}\bigl\{hf_h(z)\vee \epsilonByNDeltaMod\bigr\}^{1/2}
\end{align}
where $\COneZH$ and $\CTwoZH$ are defined in~\eqref{eq:C-zh}.
\end{lemma}

\begin{proof}
For $z \in \mathbb{R}$ and $h > 0$, define the functions $\varphi_0,\varphi_1^\pm, \varphi_1, \varphi_{2} \in \functionClassUnimodal$ as in Lemmas~\ref{lem:confidenceBandStatisticalErrorControl} and~\ref{lem:scoreErrPrelim}. 
By Lemma~\ref{lem:uniformBoundedFunctions}, for any $\varphi \in \{\varphi_0, \varphi_1^+, \varphi_1^-\}$, we have
\begin{align}
\label{eq:varphiIneq}
\max_{\omega' \in \{-1,1\}} \bigl|\hat{\mathbb{J}}_{n,\delta,\omega'}(\varphi) - \probDistribution(\varphi)\bigr| \leq 138\epsilonByNDeltaMod^{1/2}\bigl\{\probDistribution(\varphi)\vee \epsilonByNDeltaMod\bigr\}^{1/2}.
\end{align}
For $\omega \in \{-1,1\}$, we have
\begin{align*}
&\omega \biggl\{\frac{\scoreConfidenceBoundByBandwidthCentreDenominator(z)}{(1 + 4\epsilonByNDeltaMod)(1 - 2/n)} - h\hat{\mathbb{J}}_{n,\delta,\omega}(\varphi_0)\biggr\} \\
&\hspace{1.5cm}= 3h \epsilonByNDeltaMod^{1/2} \cdot \frac{\{\empiricalProb(\varphi_2) + \epsilonByNDeltaMod\}^{1/2}-\bigl\{\empiricalProb(\varphi_0)\bigl(1-\empiricalProb(\varphi_0)\bigr) +\epsilonByNDeltaMod \bigr\}^{1/2}}{(1 + 4\epsilonByNDeltaMod)(1 - 2/n)}\\
&\hspace{1.5cm} \leq \frac{3 h\epsilonByNDeltaMod^{1/2}\bigl\{\empiricalProb(\varphi_{2}) + \epsilonByNDeltaMod\bigr\}^{1/2}} {(1 + 4\epsilonByNDeltaMod)(1 - 2/n)} \leq 35h\epsilonByNDeltaMod^{1/2}\bigl\{\probDistribution(\varphi_2)\vee \epsilonByNDeltaMod\bigr\}^{1/2},
\end{align*}
where the final inequality follows from Lemma~\ref{lem:scoreErrPrelim}.  Hence, by~\eqref{eq:varphiIneq} and the first part of Lemma~\ref{lem:confidenceBandStatisticalErrorControl},
\begin{align*}
0 &\leq 
\omega \biggl\{\frac{\scoreConfidenceBoundByBandwidthCentreDenominator(z)}{(1 + 4\epsilonByNDeltaMod)(1 - 2/n)}-h\probDistribution(\varphi_0)\biggr\} \\
&\leq h\epsilonByNDeltaMod ^{1/2}\bigl\{\probDistribution(\varphi_0)\vee \epsilonByNDeltaMod\bigr\}^{1/2}\biggl\{\frac{35\bigl(\probDistribution(\varphi_{2}) \vee \epsilonByNDeltaMod\bigr)^{1/2} }{\bigl(\probDistribution(\varphi_0)\vee \epsilonByNDeltaMod\bigr)^{1/2}}+138\biggr\} \\
&= \COneZH \cdot h \epsilonByNDeltaMod^{1/2}\bigl\{hf_h(z)\vee \epsilonByNDeltaMod\bigr\}^{1/2},
\end{align*}
which establishes~\eqref{eq:scoreErrDenom}. Similarly, with $\avgDerivSignNoArg \equiv \avgDerivSign{n}{h}(z)$,
\begin{align}
\label{eq:MHatLength}
\omega\,\biggl\{&\frac{\avgDerivSignNoArg\scoreConfidenceBoundByBandwidthCentreMagnitudeNumerator[\omega \avgDerivSignNoArg](z)}{(1 + 4\epsilonByNDeltaMod)(1 - 2/n)}-\hat{\mathbb{J}}_{n,\delta,\omega}(\varphi_1^+) + \hat{\mathbb{J}}_{n,\delta,-\omega}(\varphi_1^-)\biggr\} \nonumber \\
&=\frac{3 \epsilonByNDeltaMod^{1/2}}{(1 + 4\epsilonByNDeltaMod)(1 - 2/n)}\sum_{\varphi \in \{\varphi_1^+, \varphi_1^-\}}\Bigl(\bigl\{\empiricalProb(\varphi_2) + \epsilonByNDeltaMod\bigr\}^{1/2}-\bigl\{\empiricalProb(\varphi)\bigl(1 - \empiricalProb(\varphi)\bigr) +\epsilonByNDeltaMod \bigr\}^{1/2}\Bigr) \nonumber \\
&\leq \frac{6\epsilonByNDeltaMod^{1/2}\bigl\{\empiricalProb(\varphi_{2}) + \epsilonByNDeltaMod\bigr\}^{1/2}} {(1 + 4\epsilonByNDeltaMod)(1 - 2/n)} \leq 69\epsilonByNDeltaMod^{1/2}\bigl\{\probDistribution(\varphi_2)\vee \epsilonByNDeltaMod\bigr\}^{1/2},
\end{align}
where the final inequality follows from another application of Lemma~\ref{lem:scoreErrPrelim}.  We conclude from the second part of Lemma~\ref{lem:confidenceBandStatisticalErrorControl},~\eqref{eq:MHatLength} and~\eqref{eq:varphiIneq} that
\begin{align*}
0 &\leq \omega\,\biggl\{\frac{\avgDerivSignNoArg\scoreConfidenceBoundByBandwidthCentreMagnitudeNumerator[\omega \avgDerivSignNoArg](z)}{(1 + 4\epsilonByNDeltaMod)(1 - 2/n)}-\probDistribution(\varphi_1^+) + \probDistribution(\varphi_1^-)\biggr\} \\
&\leq 374\epsilonByNDeltaMod ^{1/2}\max_{\varphi \in \{\varphi_1^+, \varphi_1^-, \varphi_2/2\}}\bigl\{\probDistribution(\varphi)\vee \epsilonByNDeltaMod\bigr\}^{1/2} = \CTwoZH \cdot \epsilonByNDeltaMod^{1/2}\bigl\{hf_h(z)\vee \epsilonByNDeltaMod\bigr\}^{1/2},
\end{align*} 
as required.
\end{proof}

\begin{proof}[Proof of Proposition~\ref{prop:scorePointwiseErr}]
Fix $n \geq 3$, $\delta \in (0,1)$ and work throughout on the event $\goodEvent$. Fix $x_0 \in \mathbb{R}$.

\medskip
\noindent \textit{(a)} Let $\omega \in \{-1,1\}$. If $\mathcal{H}_{f_0}^{(\omega)}(x_0) \neq \emptyset$, then $\Delta_{f_0}^{(\omega)}(x_0) = \infty$ by definition, so the result holds trivially. Otherwise, consider any $(z,h) \in \mathcal{H}_{f_0}^{(\omega)}(x_0) $. By the log-concavity of $f_0$ and~\eqref{eq:psi_h-psi_0} in Lemma~\ref{lem:biasControl}, we have $ \omega\bigl(\psi_h(z) - \psi_0(x_0)\bigr) \leq \omega\bigl(\psi_0(z - \omega h) - \psi_0(x_0)\bigr)\leq |\psi_0(z - \omega h) - \psi_0(x_0)|$, so 
\begin{align*}
\omega\bigl(\scoreConfidenceBoundByBandwidthCentre[\omega](z) - \psi_0(x_0)\bigr) 
&\leq \omega\bigl(\scoreConfidenceBoundByBandwidthCentre[\omega](z) - \psi_h(z)\bigr) + |\psi_0(z - \omega h) - \psi_0(x_0)|.
\end{align*}
Since $\mathcal{H}_{f_0}^{(\omega)}(x_0) \subseteq \{(z,h) \in \mathbb{R} \times (0,\infty) : \omega(x_0 - z) \geq h\}$, it follows that
\begin{align}
\omega\bigl(\scoreConfidenceBoundShapeConstrained[\omega](x_0) &- \psi_0(x_0)\bigr) \leq  \inf\bigl\{\omega\bigl(\scoreConfidenceBoundByBandwidthCentre[\omega](z) - \psi_0(x_0)\bigr) : (z,h) \in \mathcal{H}_{f_0}^{(\omega)}(x_0)\bigr\} \notag \\ 
\label{eq:scoreErrCase1}
&\leq \inf\bigl\{\omega\bigl(\scoreConfidenceBoundByBandwidthCentre[\omega](z) - \psi_h(z)\bigr)+|\psi_0(z-\omega h) - \psi_0(x_0)|: (z,h) \in \mathcal{H}_{f_0}^{(\omega)}(x_0)\bigr\}.
\end{align}
We now bound $\omega\bigl(\scoreConfidenceBoundByBandwidthCentre[\omega](z) - \psi_h(z)\bigr)$ for $(z,h) \in \mathcal{H}_{f_0}^{(\omega)}(x_0)$. Recalling that $c_{n,\delta} := (1 + 4\epsilonByNDeltaMod)(1 - 2/n)$, we apply~\eqref{eq:scoreErrDenom} in Lemma~\ref{lem:scoreErrNumDenom} and the fact that $\alpha_{z,h} \leq 1/4$ 
to obtain
\begin{align}
\label{eq:denom-relative-lower-bd}
\scoreConfidenceBoundByBandwidthCentreDenominator[1](z) \geq \scoreConfidenceBoundByBandwidthCentreDenominator[-1](z) &\geq c_{n,\delta}h^2f_h(z)\biggl(1 - \COneZH\Bigl\{\Bigl(\frac{\epsilonByNDeltaMod}{hf_h(z)}\Bigr)^{1/2}\vee \frac{\epsilonByNDeltaMod}{hf_h(z)}\Bigr\}\biggr) \geq \frac{c_{n,\delta}h^2f_h(z)}{2} > 0.
\end{align}
If $\psi_h(z) = 0$, then invoking Lemma~\ref{lem:confidenceBandStatisticalErrorControl} and combining~\eqref{eq:denom-relative-lower-bd} with~\eqref{eq:scoreErrNum} in Lemma~\ref{lem:scoreErrNumDenom} yields
\begin{align}
\label{eq:scoreErr-1}
0 \leq \omega\scoreConfidenceBoundByBandwidthCentre(z) \leq \frac{\omega \avgDerivSignNoArg\,\scoreConfidenceBoundByBandwidthCentreMagnitudeNumerator[\omega \avgDerivSignNoArg](z)}{\scoreConfidenceBoundByBandwidthCentreDenominator[-1](z)} \leq \frac{2 \CTwoZH}{h}\biggl\{\Bigl(\frac{\epsilonByNDeltaMod}{hf_h(z)}\Bigr)^{1/2}\vee \frac{\epsilonByNDeltaMod}{hf_h(z)}\biggr\} \leq \frac{\beta_{z,h}}{2},
\end{align}
where $\avgDerivSignNoArg \equiv \avgDerivSign{n}{h}(z)$. On the other hand, if $\psi_h(z) \neq 0$, then for some $\omega' \in \{-1,1\}$, it follows from Lemma~\ref{lem:confidenceBandStatisticalErrorControl},~\eqref{eq:denom-relative-lower-bd} and Lemma~\ref{lem:scoreErrNumDenom} that
\begin{align}
0 &\leq \omega\bigl(\scoreConfidenceBoundByBandwidthCentre(z) - \psi_h(z)\bigr)
= \omega \psi_h(z)\biggl\{\frac{\avgDerivSignNoArg\scoreConfidenceBoundByBandwidthCentreMagnitudeNumerator[\omega \avgDerivSignNoArg](z)/\bigl(c_{n,\delta}h^2f'_h(z)\bigr)}{\scoreConfidenceBoundByBandwidthCentreDenominator[\omega'](z)/\bigl(c_{n,\delta}h^2f_h(z)\bigr)} - 1\biggr\} \nonumber \\
&\leq \omega \psi_h(z) \frac{\scoreConfidenceBoundByBandwidthCentreDenominator[-1](z)}{c_{n,\delta}h^2 f_h(z)} \biggl\{\biggl(\frac{\avgDerivSignNoArg\scoreConfidenceBoundByBandwidthCentreMagnitudeNumerator[\omega \avgDerivSignNoArg](z)}{c_{n,\delta}h^2f'_h(z)} - 1\biggr) + \biggl(1 - \frac{\scoreConfidenceBoundByBandwidthCentreDenominator[\omega'](z)}{c_{n,\delta}h^2f_h(z)} \biggr)\biggr\} \notag \\
&\leq 2\,|\psi_h(z)|\,\biggl\{\omega\biggl(\frac{\avgDerivSignNoArg \scoreConfidenceBoundByBandwidthCentreMagnitudeNumerator[\omega \avgDerivSignNoArg](z)}{|c_{n,\delta}h^2f'_h(z)|} - \sgn f_h'(z)\biggr) + \Bigl| 1 - \frac{\scoreConfidenceBoundByBandwidthCentreDenominator[\omega'](z)}{c_{n,\delta}h^2f_h(z)}\Bigr|\biggr\} \nonumber \\
\label{eq:scoreErr}
&\leq 2\,|\psi_h(z)|\,\biggl(\frac{\CTwoZH}{h|\psi_h(z)|} + \COneZH\biggr) \biggl\{\Bigl(\frac{\epsilonByNDeltaMod}{hf_h(z)}\Bigr)^{1/2}\vee \frac{\epsilonByNDeltaMod}{hf_h(z)}\biggr\} \leq \beta_{z,h}
\end{align}
Since $(z,h) \in \mathcal{H}_{f_0}^{(\omega)}(x_0)$ was arbitrary, combining~\eqref{eq:scoreErr} with~\eqref{eq:scoreErr-1} and~\eqref{eq:scoreErrCase1} yields the result.

\medskip
\noindent \textit{(b)} By~\eqref{eq:confidence-interval-valid} and the definition of $\scoreEstimateShapeConstrained(x_0)$, we have $\psi_0(x_0), \scoreEstimateShapeConstrained(x_0)\in \hat{\mathcal{I}}_{n,\delta}(x_0)$, so
\begin{align*}
\bigl|\scoreEstimateShapeConstrained(x_0) - \psi_0(x_0)\bigr| \leq \lambda\bigl(\hat{\mathcal{I}}_{n,\delta}(x_0)\bigr) = \bigl(\psi_0(x_0) - \scoreConfidenceBoundShapeConstrained[-1](x_0)\bigr) + \bigl(\scoreConfidenceBoundShapeConstrained[1](x_0) - \psi_0(x_0)\bigr)
\end{align*}
and the result follows from part~\textit{(a)}. 

\medskip
\noindent \textit{(c)} Since $\scoreEstimateShapeConstrained(x_0)$ is defined in \eqref{eq:scoreEstimateShapeConstrained} as the Euclidean projection of $0$ onto $\hat{\mathcal{I}}_{n,\delta}(x_0)$, we have $\scoreEstimateShapeConstrained(x_0)\bigl(y - \scoreEstimateShapeConstrained(x_0)\bigr) \geq 0$ for all $y \in \hat{\mathcal{I}}_{n,\delta}(x_0)$.  Moreover, $\psi_0(x_0) \in \hat{\mathcal{I}}_{n,\delta}(x_0)$ by~\eqref{eq:confidence-interval-valid}, so considering $y = \psi_0(x_0)$ yields
\begin{align*}
\bigl(\scoreEstimateShapeConstrained(x_0) - \psi_0(x_0)\bigr)^2 &\leq \bigl(\psi_0(x_0) - \scoreEstimateShapeConstrained(x_0)\bigr)^2 + \hat{\psi}_{n,\delta}(x_0)^2 \\
&=\psi_0(x_0)^2 - 2\scoreEstimateShapeConstrained(x_0)\bigl(\psi_0(x_0) - \scoreEstimateShapeConstrained(x_0)\bigr) \leq \psi_0(x_0)^2
\end{align*}
and hence $|\hat{\psi}_{n,\delta}(x_0) - \psi_0(x_0)| \leq |\psi_0(x_0)|$.
\end{proof}

\section{Auxiliary results}
\label{sec:auxiliary}

\begin{lemma}
\label{lem:log-concave-continuity}
Let $I \subseteq \R$ be a closed interval. Then an upper semicontinuous log-concave density $f_0 \colon \R \to [0,\infty)$ is continuous on $I$ if and only if $f_0$ is absolutely continuous on $I$.
\end{lemma}

\begin{proof}
Let $a := \inf\supp f_0 \equiv F_0^{-1}(0)$ and $b := \sup\supp f_0 \equiv F_0^{-1}(1)$, which are the two points at which $f_0$ may be discontinuous (if they are finite). Since $f_0$ is log-concave, we have $f_0(\pm \infty) := \lim_{x \to \pm\infty}f_0(x) = 0$ \citep[Lemma~1]{cule2010theoretical}, and $J_0 = f_0 \circ F_0^{-1}$ is non-negative and concave on $[0,1]$ \citep[Proposition~A.1(c)]{bobkov1996extremal}, with $J_0(0) = f_0(a)$ and $J_0(1) = b$. Therefore, by \citet[Corollary~24.2.1]{rockafellar97convex}, $J_0$ is absolutely continuous on $[0,1]$. Moreover, $\norm{f_0}_\infty = \norm{J_0}_\infty < \infty$, so $F_0$ is Lipschitz on $\R$. It follows that $f_0 = J_0 \circ F_0$ is absolutely continuous on $[a,b]$. If $f_0$ is also continuous at either $a$ or $b$, then $f_0$ is absolutely continuous on $(-\infty,b]$ or $[a,\infty)$ respectively.
\end{proof}

The proofs of our upper bounds over the classes $\mathcal{F}_{\beta,L}$ in Section~\ref{subsec:holder} rely on the following envelope functions for the associated density quantile functions and their derivatives.

\begin{lemma}
\label{lem:holder-density-quantile}
For $\beta \in [1,2]$, $L > 0$ and $u \in (0,1)$, we have
\begin{align}
\label{eq:holder-density-quantile-deriv}
\sup_{f_0 \in \mathcal{F}_{\beta,L}} |J_0'(u)| &\leq 2L^{1/\beta}\log_+^{(\beta - 1)/\beta}\Bigl(\frac{C_\beta}{u \wedge (1 - u)}\Bigr), \\
\label{eq:holder-density-quantile}
\sup_{f_0 \in \mathcal{F}_{\beta,L}} |J_0(u)| &\leq 4L^{1/\beta}\bigl(u \wedge (1 - u)\bigr)\log_+^{(\beta - 1)/\beta}\Bigl(\frac{C_\beta}{u \wedge (1 - u)}\Bigr), \\
\label{eq:holder-fisher-information}
\sup_{f_0 \in \mathcal{F}_{\beta,L}} i(f_0) &\leq 8L^{2/\beta},
\end{align}
where we can take $C_\beta := \sqrt{2\beta^{1/(\beta - 1)}(\beta - 1)/\pi} \leq 2/\sqrt{\pi}$.
\end{lemma}

\begin{proof}
If $\beta = 1$, then $|J_0'(u)| = |(\psi_0 \circ F_0^{-1})(u)| \leq L$ for all $u \in (0,1)$, so by Lemma~\ref{lem:density-quantile-basics}\textit{(b)}, $i(f_0) = \int_0^1 (J_0')^2 \leq L^2$. Moreover, $J_0(0) = J_0(1) = 0$ by Lemma~\ref{lem:density-quantile-basics}\textit{(a)}, so $|J_0(u)| \leq \int_0^u |J_0'| \wedge \int_u^1 |J_0'| \leq L\bigl(u \wedge (1 - u)\bigr)$ for all $u \in [0,1]$, so the result holds. 

Next, consider $\beta \in (1,2]$ and fix $u \in (0,1)$. Let $x := F_0^{-1}(u)$ and $\rho := \psi_0(x) = J_0'(u)$. Since $\phi_0 := \log f_0$ is concave,
\[
\phi_0(y) \leq \phi_0(x) + \rho(y - x)
\]
for all $y \in \R$. Moreover, if $y \geq x$, then by the H\"older condition,
\[
\phi_0(y) - \phi_0(x) = \int_x^y \psi_0 \geq \int_x^y \bigl(\rho - L(z - x)^{\beta - 1}\bigr) \laplaced z = \rho(y - x) - \frac{L}{\beta}(y - x)^\beta.
\]
Suppose first that $\rho > 0$ and let $t := (\rho^\beta\beta/L)^{1/(\beta - 1)}$. Then by the change of variables $z = \rho(y - x)/t$, we have
\begin{equation}
\label{eq:holder-quantile-bound}
\frac{u}{1 - u} = \frac{\int_{-\infty}^x e^{\phi_0}}{\int_x^\infty e^{\phi_0}} \leq \frac{f_0(x)\int_{-\infty}^x e^{\rho(y - x)} \laplaced y}{f_0(x)\int_x^\infty e^{\rho(y - x) - L(y - x)^\beta/\beta} \laplaced y} = \frac{1/\rho}{(t/\rho)\int_0^\infty e^{t(z - z^\beta)} \laplaced z}.
\end{equation}
The function $z \mapsto z - z^\beta =: \zeta(z)$ on $(0,\infty)$ is maximised by $z_* := \beta^{-1/(\beta - 1)}$, with $\zeta(z_*) = (\beta - 1)/\beta^{\beta/(\beta - 1)}$. If $w \geq 0$, then by Taylor's theorem with the mean value form of the remainder, there exists $\bar{w} \in [0,w]$ such that
\begin{align*}
\zeta(z_* + w) - \zeta(z_*) = w\zeta'(z_*) + \frac{w^2}{2}\zeta''(z_* + \bar{w}) = -\frac{\beta(\beta - 1)w^2}{2(z_* + \bar{w})^{2 - \beta}} \geq -\frac{\beta(\beta - 1)w^2}{2z_*^{2 - \beta}}.
\end{align*}
Hence,
\begin{align*}
\int_0^\infty e^{t(z - z^\beta)} \laplaced z \geq \int_{z_*}^\infty e^{t\zeta(z)} \laplaced z &\geq e^{t\zeta(z_*)}\int_0^\infty \exp\Bigl(-\frac{t\beta(\beta - 1)z_*^{\beta - 2}w^2}{2}\Bigr) \laplaced w = \frac{e^{t\zeta(z_*)}}{C_\beta t^{1/2}},
\end{align*}
so by~\eqref{eq:holder-quantile-bound},
\[
\frac{t^{1/2}e^{t\zeta(z_*)}}{C_\beta} \leq t\int_0^\infty e^{t(z - z^\beta)} \leq \frac{1 - u}{u}.
\]
Thus, when $\rho > 0$, we have
\[
\Bigl(\frac{\rho^\beta\beta}{L}\Bigr)^{1/(\beta - 1)} = t \leq \max\biggl\{\frac{1}{\zeta(z_*)}\log\biggl(\frac{C_\beta(1 - u)}{u}\biggr),\,1\biggr\},
\]
so
\[
\rho \leq \Bigl(\frac{L}{\beta}\Bigr)^{1/\beta}\max\biggl\{\frac{\beta}{(\beta - 1)^{(\beta - 1)/\beta}}\log^{(\beta - 1)/\beta}\biggl(\frac{C_\beta(1 - u)}{u}\biggr),\,1\biggr\}.
\]
Otherwise, if $\rho < 0$, then by applying the above reasoning to the log-concave density $x \mapsto f_0(-x)$ instead, we obtain
\[
\rho \geq -\Bigl(\frac{L}{\beta}\Bigr)^{1/\beta}\max\biggl\{\frac{\beta}{(\beta - 1)^{(\beta - 1)/\beta}}\log^{(\beta - 1)/\beta}\biggl(\frac{C_\beta u}{1 - u}\biggr),\,1\biggr\}.
\]
Therefore, in all cases,
\[
|J_0'(u)| = |\rho| \leq 2L^{1/\beta}\log_+^{(\beta - 1)/\beta}\Bigl(\frac{C_\beta}{u \wedge (1 - u)}\Bigr)
\]
for every $u \in (0,1)$, which establishes~\eqref{eq:holder-density-quantile-deriv}. If $u \in (0,1/2]$, then since $J_0(0) = 0$, integrating by parts yields
\begin{align*}
\frac{J_0(u)}{2L^{1/\beta}} \leq \int_0^u \log_+^{(\beta - 1)/\beta}\Bigl(\frac{C_\beta}{v}\Bigr) \laplaced v &= \Bigl[v\log_+^{(\beta - 1)/\beta}\Bigl(\frac{C_\beta}{v}\Bigr)\Bigr]_0^u + \frac{\beta - 1}{\beta} \int_0^{u \wedge (C_\beta/e)} \log^{-1/\beta}\Bigl(\frac{C_\beta}{v}\Bigr) \laplaced v \\
&\leq u\log_+^{(\beta - 1)/\beta}\Bigl(\frac{C_\beta}{u}\Bigr) + \frac{\beta - 1}{\beta}\Bigl(u \wedge \frac{C_\beta}{e}\Bigr)\log_+^{-1/\beta}\Bigl(\frac{C_\beta}{u}\Bigr) \\
&\leq \frac{2\beta - 1}{\beta}u\log_+^{(\beta - 1)/\beta}\Bigl(\frac{C_\beta}{u}\Bigr).
\end{align*}
Similarly, for $u \in (1/2,1)$, this bound holds with $u$ replaced with $1 - u$, so~\eqref{eq:holder-density-quantile} follows. Finally, by Lemma~\ref{lem:density-quantile-basics}\textit{(b)} and integrating by parts as above,
\begin{align}
\frac{i(f_0)}{L^{2/\beta}} = \int_0^1 \frac{(J_0')^2}{L^{2/\beta}} &\leq 2 \times 4\int_0^{1/2} \log_+^{2(\beta - 1)/\beta}\Bigl(\frac{C_\beta}{v}\Bigr) \laplaced v \notag \\
&= 8\biggl(\frac{1}{2}\log_+^{2(\beta - 1)/\beta}(2C_\beta) + \frac{2(\beta - 1)}{\beta} \int_0^{C_\beta/e} \log    ^{(\beta - 2)/\beta}\Bigl(\frac{C_\beta}{v}\Bigr) \laplaced v\biggr) \notag \\
\label{eq:holder-information-integral}
&\leq 8\Bigl(\frac{1}{2} + \frac{2}{e\sqrt{\pi}}\Bigr) < 8,
\end{align}
where the penultimate and final lines rely on $\beta \leq 2$ and $C_\beta \leq 2/\sqrt{\pi} < e/2$, so~\eqref{eq:holder-fisher-information} holds.
\end{proof}

The remaining results establish some fundamental properties of the $\mathrm{kl}$ function defined in~\eqref{eq:kl-bernoulli}, which is central to the multiscale construction in Section~\ref{sec:multiscale}.

\begin{lemma}
\label{lem:kl-binomial}
Suppose that $Y_1,\dotsc,Y_n \iid \mathrm{Bern}(p)$ for some $p \in (0,1)$ and let $\bar{Y} := n^{-1}\sum_{i=1}^n Y_i$. Then for every $\varepsilon > 0$, we have
\[
\Prob\bigl(\kullbackLeibler_+(\bar{Y}, p) \geq \varepsilon\bigr) \leq e^{-n\varepsilon} \quad\text{and}\quad \Prob\bigl(\kullbackLeibler(\bar{Y}, p) \geq \varepsilon\bigr) \leq 2e^{-n\varepsilon}.
\]
\end{lemma}

\begin{proof}
The function $\eta \mapsto \kullbackLeibler(p + \eta, p)$ is continuous and strictly increasing on $[0, 1 - p)$ with $\lim_{\eta \nearrow 1 - p} \kullbackLeibler(p + \eta, p) = \log(1/p)$. Thus, if $\varepsilon > \log(1/p)$, then 
\[
\Prob\bigl(\kullbackLeibler_+(\bar{Y}, p) \geq \varepsilon\bigr) = 0 \leq e^{-n\varepsilon}.
\]
On the other hand, if $\varepsilon \in \bigl(0, \log(1/p)\bigr]$, then there exists a unique $\eta_0 \in [0, 1 - p)$ such that $\kullbackLeibler(p + \eta_0, p) = \varepsilon$~\citep{reeve2024short}, so by a version of Hoeffding's inequality~\citep[e.g.][Exercise~10.6.7\textit{(b)}]{samworth24modern},
\[
\Prob\bigl(\kullbackLeibler_+(\bar{Y}, p) \geq \varepsilon\bigr) = \Prob(\bar{Y} - p \geq \eta_0) \leq e^{-n\kullbackLeibler(p + \eta_0, p)} = e^{-n\varepsilon}.
\]
Similarly, by replacing $p$ and $\bar{Y}$ with $1 - p$ and $1 - \bar{Y}$ respectively, we obtain
\[
\Prob\bigl(\kullbackLeibler_+(1 - \bar{Y}, 1 - p) \geq \varepsilon\bigr) \leq e^{-n\varepsilon},
\]
so
\[
\Prob\bigl(\kullbackLeibler(\bar{Y}, p) \geq \varepsilon\bigr) = \Prob\bigl(\kullbackLeibler_+(\bar{Y}, p) \geq \varepsilon\bigr) + \Prob\bigl(\kullbackLeibler_+(1 - \bar{Y}, 1 - p) \geq \varepsilon\bigr) \leq 2e^{-n\varepsilon},
\]
as required.
\end{proof}

For two probability measures $P$ and $Q$ on a measurable space $(\mathcal{X},\mathcal{A})$, let 
\[
\mathrm{KL}(P,Q) := 
\begin{cases}
\log\bigl(\frac{dP}{dQ}\bigr) \laplaced P & \text{if }P \ll Q \\
\infty & \text{otherwise} 
\end{cases}
\]
denote the Kullback--Leibler divergence from $Q$ to $P$.

\begin{lemma}
\label{lem:klJointConvex}
The function $\kullbackLeibler:[0,1] \times [0,1] \to [0,\infty]$ is jointly convex. Hence, for any Borel measurable functions $p,q:[0,1]\to [0,1]$,
\[ 
\kullbackLeibler\biggl(\int_0^1 p(\gamma) \laplaced \gamma, \int_0^1 q(\gamma)\, \laplaced \gamma\biggr) \leq \int_0^1\kullbackLeibler\bigl(p(\gamma),q(\gamma)\bigr) \laplaced\gamma.
\]
\end{lemma}

\begin{proof}
Let $p_1,p_2,q_1,q_2 \in [0,1]$ and $\lambda \in [0,1]$. For $\ell \in \{1,2\}$, let $P_\ell := \mathrm{Bern}(p_\ell)$ and $Q_\ell := \mathrm{Bern}(q_\ell)$.  Then $\lambda P_1+(1 - \lambda)P_2$ is a Bernoulli distribution with parameter $\lambda p_1 + (1 - \lambda)p_2$, so
\begin{align*}
\kullbackLeibler\bigl(\lambda p_1 + (1 - \lambda)p_2, \lambda q_1 + (1 - \lambda)q_2\bigr) &=\mathrm{KL}\bigl(\lambda P_1+(1 - \lambda)P_2, \lambda Q_1+(1 - \lambda)Q_2\bigr)\\
&\leq \lambda \mathrm{KL}(P_1,Q_1) + (1 - \lambda)\mathrm{KL}(P_2,Q_2)\\
&=\lambda \kullbackLeibler(p_1,q_1) + (1 - \lambda)\kullbackLeibler(p_2,q_2),
\end{align*}
where we have used joint convexity of the Kullback--Leibler divergence \citep[Exercise~8.7]{samworth24modern}. Thus, $\kullbackLeibler:[0,1]\times[0,1]\to [0,\infty]$ is jointly convex. 

For the second claim, let $\Gamma \sim \text{Unif}(0,1)$, let $X_1:=p(\Gamma)$ and let $X_2:= q(\Gamma)$. Applying the multivariate version of Jensen's inequality \citep[Theorem~10.3.11]{samworth24modern}, we obtain
\begin{align*}
\kullbackLeibler\biggl(\int_0^1 p(\gamma) \laplaced\gamma, \int_0^1 q(\gamma) \laplaced\gamma\biggr) = \kullbackLeibler\bigl(\mathbb{E}(X_1), \mathbb{E}(X_2)\bigr) \leq \mathbb{E}\bigl(\kullbackLeibler(X_1,X_2)\bigr) = \int_0^1\kullbackLeibler\bigl(p(\gamma),q(\gamma)\bigr) \laplaced\gamma,
\end{align*}
as required.
\end{proof}

\begin{lemma}
\label{lem:kl-joint-convexity}
The function $\kullbackLeibler_+:[0,1]\times [0,1] \to [0,\infty]$ is jointly convex.
\end{lemma}

\begin{proof}
Let $p_1,p_2,q_1,q_2 \in [0,1]$ and $\lambda \in [0,1]$. Define $p := \lambda p_1 + (1 - \lambda)p_2$ and $q := \lambda q_1 + (1 - \lambda)q_2$. If $p \leq q$, then
\[
\kullbackLeibler_+(p,q) = 0 \leq \lambda \kullbackLeibler_+(p_1,q_1) + (1 - \lambda)\kullbackLeibler_+(p_2,q_2).
\]
Now suppose that $p > q$. If both $p_1 > q_1$ and $p_2 > q_2$, then by Lemma~\ref{lem:klJointConvex},
\begin{align*}
\kullbackLeibler_+(p,q) = \kullbackLeibler(p,q) &\leq \lambda \kullbackLeibler(p_1,q_1) + (1 - \lambda)\kullbackLeibler(p_2,q_2) \\
&=\lambda\kullbackLeibler_+(p_1,q_1) + (1 - \lambda)\kullbackLeibler_+(p_2,q_2).
\end{align*}
If $p_1 \leq q_1$, then we must have $p_2 > q_2$, and
\begin{align*}
\lambda p_1 + (1 - \lambda) q_2 \leq \lambda q_1 + (1 - \lambda)q_2 = q < p.
\end{align*}
It follows that
\begin{align*}
\kullbackLeibler_+(p,q) &\leq \kullbackLeibler_+\bigl(\lambda p_1 + (1 - \lambda)p_2, \lambda p_1 + (1 - \lambda)q_2\bigr)\\
&=\kullbackLeibler\bigl(\lambda p_1 + (1 - \lambda)p_2, \lambda p_1 + (1 - \lambda)q_2\bigr) \\
&\leq (1 - \lambda)\kullbackLeibler(p_2,q_2) = \lambda\kullbackLeibler_+(p_1,q_1) + (1 - \lambda)\kullbackLeibler_+(p_2,q_2),
\end{align*}
where we have used the fact that $\kullbackLeibler_+(p',\cdot)$ is decreasing on $[0,1]$ for each $p' \in [0,1]$, as well as Lemma~\ref{lem:klJointConvex}. The case where $p_2 \leq q_2$ follows similarly.
\end{proof}

\begin{lemma}
\label{lem:klLowerBound}
Given any $p,q \in [0,1]^2$, we have
\begin{align*}
\frac{9(p-q)^2}{2(p+2q)(3-p-2q)} \leq \kullbackLeibler(p,q).
\end{align*}
Consequently, for $\eta > 0$, if $\kullbackLeibler(p,q) < 9\eta/2$ then 
\[
\Bigl|p - \frac{2(1 - 2\eta)q + 3\eta}{2(1 + \eta)}\Bigr| \leq \frac{3\sqrt{4q(1 - q)\eta + \eta^2}}{2(1 + \eta)}
\]
and
\[
\Bigl|q - \frac{(1 - 2\eta)p + 3\eta}{1 + 4\eta}\Bigr| \leq \frac{3\sqrt{p(1 - p)\eta + \eta^2}}{1 + 4\eta}.
\]
\end{lemma}

\begin{proof}
For the first bound, see for example~\citet[Lemma~7]{reeve2024short}. Given $\eta > 0$, rearrangement of the resulting inequality yields
\begin{align*}
0 &\geq p^2(1 + \eta) - \bigl(2(1 - 2\eta)q + 3\eta\bigr)p + (1 + 4\eta)q^2 - 6q\eta,
\end{align*}
which is equivalent to both the second and third bounds.
\end{proof}

\end{document}